\documentclass[12pt, a4paper]{article}
\usepackage{amsmath,amsthm,amssymb,amsfonts}
\usepackage[utf8]{inputenc}
\usepackage[T1]{fontenc}
\usepackage{lmodern}
\usepackage[english]{babel}
\usepackage{a4wide}
\usepackage{graphicx}
\usepackage[usenames,dvipsnames]{color}
\usepackage[colorlinks=true, pdfstartview=FitV, linkcolor=black, citecolor=blue,urlcolor=blue]{hyperref}

\newtheorem{theorem}{Theorem}
\newtheorem{lemma}[theorem]{Lemma}
\newtheorem{proposition}[theorem]{Proposition}
\newtheorem{definition}[theorem]{Definition}

\newtheorem{corollary}[theorem]{Corollary}

\theoremstyle{remark}

\newtheorem{remark}[theorem]{Remark}

\numberwithin{equation}{section}
\numberwithin{theorem}{section}

\def\Xint#1{\mathchoice
{\XXint\displaystyle\textstyle{#1}}%
{\XXint\textstyle\scriptstyle{#1}}%
{\XXint\scriptstyle\scriptscriptstyle{#1}}%
{\XXint\scriptscriptstyle\scriptscriptstyle{#1}}%
\!\int}
\def\XXint#1#2#3{{\setbox0=\hbox{$#1{#2#3}{\int}$ }
\vcenter{\hbox{$#2#3$ }}\kern-.6\wd0}}

\def\dashint{\Xint-}


\newcommand{\N}{{\mathbb N}}

\newcommand{\R}{{\mathbb R}}

\newcommand{\pa}{{\partial}}
\newcommand{\na}{{\nabla}}

\newcommand{\eps}{{\varepsilon}}

\def\curl{\hbox{curl \!}}
\def\div{\hbox{div \!}}

\title{Weak solutions of a viscous model \\ for fluid-bubbles interaction} 
\author{Cosmin Burtea\thanks{cosmin.burtea@imj-prg.fr, Université Paris Cité, Sorbonne Université, CNRS, IMJ-PRG, F-75013 Paris, France}, David Gérard-Varet \thanks{david.gerard-varet@imj-prg.fr, Université Paris Cité, Sorbonne Université, CNRS, IMJ-PRG, F-75013 Paris, France}}

\begin{document}
\maketitle

\begin{abstract}
    We present a system of Navier-Stokes type that describes the dynamics of several spherical bubbles of gas in a liquid. It is derived from a more complete model, where the bubbles are seen as inclusions of gas of homogeneous barotropic pressure with free surfaces. The usual condition of continuity of the stress is relaxed in order to preserve the sphericity of the bubbles through time. We construct weak solutions {\em \`a la Leray} for this relaxed system, up to collision between the bubbles. Although these solutions are reminiscent of weak solutions for fluid-solid interaction systems, accounting for the compression/dilation of the bubbles creates new and significant mathematical difficulties. 
\end{abstract}

\section{A new model for fluid-bubbles interaction} \label{sec1}

The evolution of one or several bubbles of gas in a liquid is a classical topic of fluid mechanics. One main domain of interest is bubble acoustics, as compression and dilation of the bubbles create sound waves in the liquid. Such waves are met in natural phenomena, like the roar of a waterfall, but also in many medical, industrial or military processes. They may be sometimes undesirable, when they scatter the signal emitted or received by sonars, for instance in military or geophysical surveying.  On the contrary, they may facilitate medical diagnosis, amplifying the signal in the regions where they are injected.
Another phenomenon of current interest is {\em sonoluminescence}, which is the emission of photons through the collapse of bubbles. It has potential  applications like in dental ultrasonics, but also potential negative impact, such as increasing erosion processes. Let us stress that in many contexts, the population of bubbles in a liquid has to be monitored and controlled, like in the oceans due to their greenhouse effect, or in many industrial processes such as production of paints. We refer to the remarkable review paper \cite{leighton}, see also  \cite{Prosperetti} for more. 

From the modeling point of view, it is natural to describe the liquid and the bubbles as immiscible domains with  interfaces. The liquid part is typically described by the incompressible Navier-Stokes equation, while the bubbles are governed by a compressible Navier-Stokes or Euler equation. Nevertheless, it is often argued that the pressure in the bubble can be considered as spatially homogeneous, see \cite{Prosperetti}, or \cite[Appendix A]{Biro}. One main justification is that the Mach number is so low that the pressure gradient at leading order is zero.  Under a barotropic law, the pressure at time $t$ in a bubble is then given for some $c_0 > 0$ by 
$$ p(t) = c_0 \rho(t)^\gamma = c |B(t)|^{-\gamma}, \quad c := c_0 M_0^\gamma$$
where $c$ is a positive constant, $\gamma > 1$ is the adiabatic exponent,  $M_0$ is the mass of the bubble (constant in time), and $|B(t)|$ is the volume of the bubble at time $t$. This leads to a system of the  following form, where the $N$ bubbles occupy the domains $B_1(t), \dots, B_N(t) \subset \R^3$ at time $t$, and the fluid domain  is 
$\Omega(t) := \R^3 \setminus  \cup_{i=1}^N B_i(t)$:   
\begin{equation} \label{fullFB}
\begin{aligned}
        \pa_t u + u \cdot \na u + \na p - \nu \Delta u  = 0, & \quad t > 0, \quad  x \in \Omega(t)   \\
        \div u  = 0, & \quad t > 0, \quad  x \in \Omega(t)   \\
         p_i(t) = c_i |B_i(t)|^{-\gamma}, & \quad t > 0,  \quad i \in \overline{1,N} \\
        \mathrm{T}(u,p) \mathfrak{n}\vert_{\pa B_i(t)} = - p_i \mathfrak{n}\vert_{\pa B_i(t)}  , & \quad t > 0,  \quad i \in \overline{1,N} \\
       \pa B_i(t) = \Theta_{t,0}(\pa B_i(0)),  & \quad t > 0,  \quad i \in \overline{1,N}. 
\end{aligned}
\end{equation}
The first two lines are the standard incompressible Navier-Stokes equations for the liquid (Euler equations when $\nu = 0$). In the third line, the pressure in the bubble $i$, denoted $p_i(t)$,  obeys the relation that we have just described. The fourth line expresses the continuity of the stress at the boundary of each bubble:  the stress $\mathrm{T}(u,p) \mathfrak{n}$ exerted by the fluid, with  $\mathrm{T}(u,p) := 2 \nu D(u) - p \mathrm{Id}$ the Newtonian stress tensor, is equal to the stress $-p_i \mathfrak{n}$ exerted by the bubble $i$.  Note that for any domain $O$, $\mathfrak{n}\vert_{\pa O}$ will always refer to the unit normal vector pointing outward $O$. We  could also have included surface tension, by replacing in the fourth line  $p_i$  by $p_i - 2 \alpha \kappa_i$  with $\alpha > 0$ and $\kappa_i$ the mean curvature of the surface $\pa B_i(t)$. As this additional surface tension would not play any significant role in the mathematical analysis of our relaxed model later, we neglect it.  Finally, the last line of the model expresses that the boundary of each bubble is preserved by the flow $\Theta_{t,0}$ of the fluid. This is a way to express that the boundary of each bubble is a material surface: the normal velocity of the bubble and the normal velocity of the fluid coincide, so that no exchange of mass across this boundary occurs. 

The system \eqref{fullFB}, although natural, is very complex, especially as it involves multiple free surfaces. It relates to the  free surface Navier-Stokes equations (or Euler when $\nu = 0$), which is a classical model for water waves:  there, the fluid domain is typically of the form $\Omega(t) = \{ z < \varphi(t,x,y) \}$, and  the stress tensor $\mathrm{T}(u,p)\mathfrak{n}$ at the boundary of the free-surface is given by the atmospheric pressure $p_{atm} \mathfrak{n}$. We refer to \cite{Beale,Tani,Guo-Tice} in the Navier-Stokes case, and to the book \cite{Lannes} for the inviscid equations. Generically, one obtains local existence of smooth solutions for smooth data. In the context of fluid-bubbles interaction, the local strong well-posedness of system \eqref{fullFB} was recently established in \cite{Marcel}. As far as we know, the only setting in which more was done on \eqref{fullFB} is the case of a single bubble $N=1$, when this bubble is initially spherical. Under radial symmetry of the initial data, the sphericity of the bubble is preserved, as well as the symmetry of the fields, providing a simpler set of equations. Denoting $B(t) = B(x(t), r(t))$ the spherical bubble at time $t$, an important subclass of solutions is then given by $x(t)= 0$, 
$u(t,x) = - \frac{\dot{r} r}{|x|^3} x$, and $r(t)$ satisfies the famous {\em Rayleigh-Plesset} equation  
$$ r \ddot{r} + \frac{3}{2} \dot{r}^2 + 4 \nu \frac{\dot{r}}{r}  = p(t) - p_\infty  $$
where $p_\infty$ is the uniform pressure at infinity. Substituting $p(t) = c |B(t)|^{-\gamma} = c \left(\frac{4\pi}{3}\right)^{-\gamma} r(t)^{-3\gamma}$ into this relation, we end up with a closed differential equation on $r$. 
Still in the radial setting, several variations are possible, for instance replacing the barotropic law in the single bubble by the law of perfect gas and adding an evolution equation on temperature. We refer to \cite{Biro} for an in-depth analysis of the well-posedness properties of such system. Recently, a nice  stability analysis of this radially symmetric setting was conducted in \cite{Weinstein}. See also \cite{Weinstein2} with an additional periodic forcing, and \cite{Weinstein3} where the introduction of asymmetric deformations of the bubbles is shown to trigger some strong instability. 

Clearly, the model \eqref{fullFB} is too complex to give rise to tractable numerical schemes, or even to get some qualitative insight into the dynamics of the bubbles.  This is especially true when the number of bubbles $N$ is very large and some simpler macroscopic effective model is needed. It is not even guaranteed that the time of existence of \eqref{fullFB} does not shrink to zero when $N$ goes to infinity. 
In order to simplify the analysis, many physics papers assume that the bubbles remain spherical, notably when performing heuristically the asymptotics $N \rightarrow +\infty$: see for instance \cite{Caflisch}. 
However, the preservation of sphericity of the bubbles through time is generically wrong under the evolution of \eqref{fullFB}. This can be understood as follows. Assuming we have a solution with spherical bubbles, we can consider the pressures $p_i(t)$ and the bubble trajectories  $B_i(t) = B(x_i(t), r_i(t))$  as given, and focus on the subproblem made of the  Navier-Stokes equations (\ref{fullFB}a)-(\ref{fullFB}b) together with the stress condition (\ref{fullFB}d). This problem has already enough boundary conditions, so that no further constraint can be put on $u \cdot \mathfrak{n}\vert_{\pa B_i}$. But for general configurations of bubbles, there is no reason for this normal velocity to be compatible with maintaining their sphericity.

Therefore, if we wish to describe the evolution of spherical bubbles, system \eqref{fullFB} needs to be relaxed. Such relaxation is not discussed in \cite{Caflisch}, where the manipulations are mostly formal,  and it does not seem to be much discussed in the literature. The only references we know in which a relaxed model is explicitly written are \cite{Gavrilyuk,Smereka, Hillairet1,Hillairet2}. The recent papers \cite{Hillairet1,Hillairet2} will be discussed specifically at the end of this section. Articles \cite{Gavrilyuk,Smereka} deal with the inviscid and irrotational case:  $\nu=0, \curl u  =0$. In particular, $u = \na \varphi$ for some potential $\varphi$ (say vanishing at infinity). The common  idea in both papers is to assume that the bubbles  $B_i = B(x_i, r_i)$ are spherical  and to write the equation satisfied by $\varphi$
$$ \Delta \varphi = 0 \quad \text{ in } \: \Omega, \quad \pa_\mathfrak{n}\phi\vert_{\pa B_i} = \dot{r}_i + \dot{x}_i \cdot \mathfrak{n}. $$  
(the time dependence is omitted to lighten notations). The potential  $\varphi$ can be seen as a functional of 
$X = (x_i)_{i \in \overline{1,N}}$, $R = (r_i)_{i \in \overline{1,N}}$
and of their time derivatives. The same is then true for the velocity:   $u =\mathcal{U}\left[\dot{X},\dot{R},X,R\right]$. The keypoint is then to introduce a Lagrangian  
$$\displaystyle L\left(\dot{X}, \dot{R}, X,R\right) := E_k - E_p := \frac{1}{2} \int_{\Omega} |u|^2 + \sum_{i=1}^N \frac{c_i}{3 \gamma - 3} r_i^{3-3\gamma}$$ 
where  $E_k$ is the kinetic energy of the fluid, and $E_p$ the sum of the potential energy of the bubbles. Writing the Lagrange equations associated with this Lagrangian, the authors derive a set of ordinary differential equations on $X,R$. Under the assumption that the bubbles are far enough from one another, this set of equations can  be approximated by a more explicit one  using  an  approximation of the field  $u = \mathcal{U}[X,R]$ given by the so-called method of reflections \cite{Richard-Juan}. Although efficient, the approach of \cite{Gavrilyuk, Smereka} is limited to the inviscid and curl-free setting, and does not connect directly  to the PDE system \eqref{fullFB}. 

Our ambition is to present and study a relaxed model for spherical bubbles, derived from \eqref{fullFB} and valid in both the inviscid and the viscous case. In the present paper, we will introduce this model, explain briefly its origin, and build weak solutions {\em \`a la Leray} in the viscous case ($\nu > 0)$. A much more complete discussion of the derivation of the model will be given in the companion paper \cite{Nous}. 

\medskip
Let us introduce the relaxation of \eqref{fullFB} that we have in mind. The bubbles are described at time $t$ by balls
\begin{equation} \label{def_balls}
    B_i(t) := B(x_i(t),r_i(t)), \quad i \in \overline{1,N}. 
    \end{equation}
The fluid domain is 
    \begin{equation} \label{def_fluid_domain}
    \Omega(t) := \R^3 \setminus \cup_{i=1}^n \overline{B_i(t)}.
    \end{equation} 
The unknowns of the model are the centers and radii of the bubbles:   
 \begin{equation} \label{defXR}
 X := (x_i)_{i \in \overline{1,N}}, \quad  R := (r_i)_{i \in \overline{1,N}}
 \end{equation}
as well as the fluid velocity and pressure $u,p$ in the fluid domain. Given $c_1, \dots, c_N > 0$, we claim that a natural set of governing equations is:    
\begin{equation} \label{main}
\begin{aligned}
        \pa_t u + u \cdot \na u + \na p - \nu \Delta u  = 0, & \quad t > 0, \quad x \in \Omega(t)  \\
        \div u  = 0, & \quad t > 0, \quad x \in \Omega(t) \\
         p_i = \frac{c_i}{4\pi} r_i^{-3\gamma}, & \quad t > 0,  \quad i \in \overline{1,N} \\
        \mathrm{T}(u,p) \mathfrak{n} \times \mathfrak{n}\vert_{\pa B_i}  = 0, &\quad t > 0, \quad i \in \overline{1,N},    \\
        \dashint_{\pa B_i} \mathrm{T}(u,p) \mathfrak{n} \cdot \mathfrak{n}  = - p_i, &\quad t > 0,  \quad  i \in \overline{1,N},  \\
        \dashint_{\pa B_i} \mathrm{T}(u,p) \mathfrak{n}  = 0, &\quad t > 0,  \quad  i \in \overline{1,N},  \\
         u \cdot \mathfrak{n}\vert_{\pa B_i(t)}  = \dot{r}_i + \dot{x}_i \cdot \mathfrak{n}, & \quad t > 0 , \quad i \in \overline{1,N}.   
\end{aligned}    
\end{equation}
It is completed with initial conditions: 
$$ X(0) = X_0, \quad R(0) = R_0, \quad u\vert_{t=0} = u_0 \: \text{ in } \: \Omega_0 := \R^3 \setminus \cup_{i=1}^N B(x_{0,i}, r_{0,i}).$$
Let us compare \eqref{main} to the full system \eqref{fullFB}. We notice that the Navier-Stokes equations, and the formula giving the pressure $p_i$ in $B_i$ stay the same, {\em after proper redefinition of the constant $c_i$} (the normalization by $\frac{1}{4\pi}$ will allow to have lighter formulas later). The last condition expresses again that $\pa B_i$ is a material interface between the fluid and  bubble $i$, as the normal component of the fluid velocity $u \cdot \mathfrak{n}$ equals the normal component of the bubble velocity $u_i \cdot \mathfrak{n} := \dot{x}_i \cdot \mathfrak{n} + \dot{r}_i$. But of course, the novelty is that the sphericity of the bubbles puts a restriction on the structure of $u \cdot \mathfrak{n}$ at $\pa B_i$. Contrary to the condition $\Theta_{t,0}(\pa B_i(0)) = \pa B_i(t)$ of \eqref{fullFB}, which can be interpreted as an evolution equation on the surface of the bubble, this evolution has now only limited degrees of freedom, namely the velocity of the center $\dot{x}_i$ and the compression/dilation speed $\dot{r}_i$. This imposes  to relax the original relation of continuity of the stress in \eqref{fullFB}. The continuity of the tangential stress  
$$\mathrm{T}(u,p) \mathfrak{n} \times \mathfrak{n}\vert_{\pa B_i} = 0$$
is kept, but  the continuity of the normal stress $\mathrm{T}(u,p) \mathfrak{n} = - p_i \mathfrak{n}$ is now replaced by two of its consequences, namely
$$   \dashint_{\pa B_i} \mathrm{T}(u,p) \mathfrak{n} \cdot \mathfrak{n}  = - p_i, \quad   \dashint_{\pa B_i} \mathrm{T}(u,p) \mathfrak{n}  = 0.$$
The unknowns $\dot{r}_i$ and $\dot{x}_i$ are somehow Lagrange multipliers associated with these two relations. 

The derivation of system \eqref{main} will be discussed in depth in \cite{Nous}. We give here only a short explanation of the relaxed conditions on the stress tensor. To understand where they come from, it is easier to start from the toy model 
\begin{equation}
\begin{aligned}
   u - \nu \Delta u +\na p   = f, & \quad \text{in } \Omega, \\
   \div u   = 0, & \quad \text{in } \Omega, \\
   \mathrm{T}(u,p) \mathfrak{n}\vert_{\pa B_i}  = - p_i \mathfrak{n}, & \quad i \in \overline{1,N}.
\end{aligned}
\end{equation}
This toy model is inspired by \eqref{fullFB}. Vaguely, it corresponds to a crude time discretization of the Navier-Stokes equation: given the domains $B_i$, $\Omega$ and pressure $p_i$ at time step $t_k$, we wish to compute the velocity field $u$ at time step $t_{k+1}$, in order then to update the domains of the bubbles. The boundary condition is  the same as in \eqref{fullFB}. Under appropriate regularity and decay assumptions on the data, one can show that this Stokes problem with a stress boundary condition is well-posed: see \cite[Chapter 7]{Boyer} for a detailed discussion (the discussion is carried in a bounded domain $\Omega$, but the addition of a zero order term in our toy model allows to adapt everything). Importanty, the solution $u$ has a variational characterization: it is the minimizer in the space 
$$ V := \{ v \in H^1(\Omega), \quad \div v = 0 \}$$
of the functional 
$$ \mathcal{F}(v) := \frac{1}{2} \int_{\Omega} (|v|^2 + 2\nu |D(v)|^2) -  \sum_{i=1}^N  \int_{\pa B_i} p_i \mathfrak{n} \cdot v  - \int_{\Omega} f \cdot v .$$
The problem with this toy model is the same as with \eqref{fullFB}:  $u \cdot \mathfrak{n}$ is then fully determined, and there is no reason that it has the right structure to preserve sphericity. Namely, there is no reason that it belongs to the space 
$$ W := \left\{ v \in V \: \text{ s. t. for all 
$i \in \overline{1,N}, \: \exists s_i \in \R, \: v_i \in \R^3; \: $ }  v \cdot \mathfrak{n} \vert_{\pa B_i} = s_i + v_i \cdot \mathfrak{n} \right\}.$$
In order to enforce this constraint, a natural idea is to minimise $\mathcal{F}$ on $W$ instead of $V$. But it can be shown by standard arguments that such minimizer is the solution $u \in W$ of 
\begin{equation}
\begin{aligned}
   u - \nu \Delta u +\na p  = f, &\quad \text{in } \Omega, \\
   \div u = 0, & \quad \text{in } \Omega, \\
     \mathrm{T}(u,p) \mathfrak{n} \times \mathfrak{n}\vert_{\pa B_i}  = 0, & \quad i \in \overline{1,N},    \\
        \dashint_{\pa B_i} \mathrm{T}(u,p) \mathfrak{n} \cdot \mathfrak{n}   = - p_i, &  \quad  i \in \overline{1,N},  \\
        \dashint_{\pa B_i} \mathrm{T}(u,p) \mathfrak{n}  = 0,  & \quad  i \in \overline{1,N}. 
\end{aligned}
\end{equation}
We recover in this way the integral constraints on the stress tensor. Note that in \eqref{main}, in the case $\nu=0$, the constraint $\displaystyle \dashint_{\pa B_i} \mathrm{T}(u,p) \mathfrak{n} \cdot \mathfrak{n}   = - p_i$ resumes to $\displaystyle \dashint_{\pa B_i} p = p_i$. Such constraint on the average pressure is briefly mentioned in \cite[p472]{Gavrilyuk}.  For more on the derivation of \eqref{main}, we refer again to our upcoming paper \cite{Nous} where it will be shown that system \eqref{main}, at least when $\nu=0$, can be obtained via Hamilton's stationary action principle using the formalism introduced in \cite{BurteaGavrilyukPerrin2024}.

\medskip
Our main ambition in the present paper is to build weak solutions of system \eqref{main}, until collision between the bubbles. Precise statements will be given in Section \ref{sec_statement} and the general strategy of the proof will be explained in Section \ref{sec_strategy}.  This work, and more generally the mathematical analysis of  bubbles in a fluid, are natural in view of the progress made in the well-posedness theory of incompressible fluid-solid systems, or in the derivation of effective models for rigid suspensions. Classical references in the former field are \cite{SanMartin,Feireisl,Takahashi,Grandmont,SueurGlass,Hillairet}. Classical ones in the latter are \cite{Hofer,Mecherbet,SanchezPalencia,Mazzucato,DuerinckxGloria,Gerard-VaretHillairet1}. These lists are by no mean exhaustive.  Let us stress that the boundary conditions (\ref{main}d) and (\ref{main}g) are of Navier type, so that they relate naturally to more recent works such as \cite{Gerard-VaretHillairet2,Sarka,Sarka2,Sarka3,Neustupa,Muha}. This appearance of Navier boundary conditions is of no surprise, considering that our starting point is \eqref{fullFB}: the connection between free-surface problems and Navier condition has been emphasized for long, see for instance \cite{MasmoudiRousset1,MasmoudiRousset2}.

Although related to the studies just mentioned, the analysis of \eqref{main} raises new and serious difficulties. First, and prior  to any specific construction scheme, both {\it a priori} estimates and proof of sequential continuity require some new arguments, to be  developped in Section \ref{sec_prelim}. Roughly, due to the new kind of  boundary condition (\ref{main}g), one relies on a new decomposition of the solution $u$ (resp. of a sequence of solutions  $(u^n)_{n \in \N}$) as 
$$ u = \tilde u + w, \quad \text{ resp. } u^n =  \tilde u^n + w^n  $$ 
where $\tilde u$ is tangent at $\pa \Omega$, resp. $\tilde u^n$ tangent at $\pa \Omega^n$ and where $w$, resp. $w^n$ is a harmonic vector field. See Section \ref{sec_functional_spaces} for more on this decomposition. The convergence of $\tilde u^n$ to $\tilde u$ (including strong convergence in $L^2$) relies on a generalization of Aubin-Lions lemma to moving domains taken from \cite{Moussa}. The convergence of $w^n$ to $w$ relies on domain continuity results for solutions of elliptic equations inspired by  \cite{Henrot}. 

Beyond sequential continuity, to design a good construction scheme also requires new ideas.  Indeed, standard approximation schemes for fluid-solid problems rely on {\em global approximations}, that is on approximate equations set on the union of the fluid and the solid domains (which is $\R^3$ in our setting). Such equations typically include penalization terms supported in the solid part. For instance, if $\chi^n_S$ is the indicator function of the (approximate) solid domain, these terms may take the form $ - n \div( \chi^n_S D(u^n))$ in the case of no-slip conditions \cite{SanMartin}, or $n \chi^n_S P_S^n(u^n)$ with $P^n_S$ some projector on the space of rigid velocity fields in the case of Navier conditions \cite{Gerard-VaretHillairet2}. In the case of bubbles, how to find such global approximation is unclear. On one hand, a global approximate solution $u^n$ should  be divergence-free inside the approximate fluid domain $\Omega^n$ to avoid dealing with the pressure. On the other hand, its divergence should be non-zero in the approximate bubbly part $\R^3 \setminus \Omega^n$ in order to accomodate the condition on $u \cdot \mathfrak{n}$. But then, there is no natural functional space independent of the domain in which finding $u^n$: choosing $L^2(\R^3)$ is not appropriate to the fluid part, choosing $L^2_\sigma(\R^3)$ is not appropriate to the bubbly part.  
Several other technical difficulties add up to this one, that will be described in due course. All the aspects related to the construction and compactness of approximate solutions  will be developed in Sections \ref{sec_prescribed} and \ref{sec_approx_compact}. 

\medskip
We conclude this introduction with the discussion of another interesting model for spherical bubbles, that was  introduced recently in \cite{Hillairet1} and completed by a one-dimensional study in \cite{Hillairet2}. The starting point of \cite{Hillairet1} is a system of immiscible compressible Navier-Stokes flows, one with velocity $u$ for the surrounding fluid, the others with velocities $u_i$ for the bubbles. The authors impose the continuity of the stress but also {\em the continuity of the velocities} at each interface. To simplify the dynamics of the bubbles, the limit of infinite shear viscosity is considered: the Navier-Stokes evolution inside the bubbles is replaced by the constraint 
$$ D(u_i) - \frac{1}{3} (\div u_i) \, \mathrm{Id}= 0, \quad i \in \overline{1,N}. $$
Under the assumptions that the bubbles remain spherical and  under an additional symmetry assumption on $u_i$, the previous relation implies: 
$$ u_i\vert_{B_i} = \dot{x}_i + \omega_i \times (x-x_i) + \frac{\dot{r}_i}{r_i} (x-x_i).$$
Continuity of the velocity at the interfaces yields a Dirichlet condition on $u$: for all $i \in \overline{1,N}$, 
$$ u\vert_{\pa B_i} = \dot{x}_i + \omega_i \times (x-x_i) + \frac{\dot{r}_i}{r_i} (x-x_i).$$
The final step to close the system is to derive an evolution equation for $\dot{x}_i$, $\dot{\omega}_i$ and $\dot{r}_i$. This is done by coming back to the  system of compressible fluid flows (with large but finite shear viscosity in the spherical bubbles) and considering test fields $\psi$ that satisfy also the constraint
$D(\psi) -  \frac{1}{3} \div \psi \mathrm{Id}= 0$. 

The system derived in \cite{Hillairet1}, for which the construction of weak solutions remains open, is very different from \eqref{main}. It makes a crucial use of the viscosity of the bubbles:  a no-slip condition is imposed at the boundary of the bubbles, and the limit of infinite shear viscosity is considered. On the contrary, in \eqref{main}, the viscosity of the bubble is fully neglected, and tangential discontinuity of the  velocity at the interface is allowed, like in the original free surface system \eqref{fullFB}. This makes possible to consider \eqref{main} in the purely inviscid case, where diffusion is omitted in the fluid part as well. This is common practice in physics studies related to bubbles. Another distinct feature is that the inertia of the bubbles does not show up in \eqref{main} (and already in \eqref{fullFB}), while it appears in \cite{Hillairet1}. This is again a common approximation, arguing that the mass of the bubbles is very small. Finally, the angular velocity $\omega_i$ of the bubble $i$ does not show up in \eqref{main}, as replacing $u_i$ by $u_i + \omega_i \times (x-x_i)$ does not affect the normal component at $\pa B_i$.

\section{Statement of the main result} \label{sec_statement}
\subsection{Functional spaces} \label{sec_functional_spaces}
Let $U_1, \dots, U_N$ smooth bounded domains of $\R^3$ with disjoint closures,  and let $U := \R^3 \setminus \cup_{i=1}^N \overline{U_i}$. We denote by $L^{2}\left(U\right)  $ the space of square-integrable functions over $U$ (scalar or vector valued depending on the context), 
and we distinguish two of its
subspaces: the classical 
\[
L_{\sigma}^{2}\left(U\right)  :=\left\{  w\in L^{2}\left(U\right)  :\operatorname{div}w=0\text{ and }w\cdot \mathfrak{n}=0\text{
on }\partial U\right\}
\]
and the less classical space 
\[
L_{\mathrm{dil}}^{2}\left(U\right)  :=\left\{  w\in
L^{2}\left(U\right)  :%
\begin{array}
[c]{l}%
\operatorname{div} w =0\text{ and}\\
\forall i\in \overline{1,N}:w\cdot \mathfrak{n}=s_{i}+v_{i}\cdot \mathfrak{n}\text{ on }\partial U_{i},  \text{where }\\
s_{i}\in\mathbb{R}\text{ and }v_{i}\in\mathbb{R}^{3}.
\end{array}
\right\}
\]
which allows for dilation (or contraction) of the boundaries of the bubbles. The space $L_{\mathrm{dil}}^{2}\left(  U\right)$ will be crucial to the analysis of \eqref{main}.   Obviously, $L_{\sigma}^{2}\left(  U\right)  $ is a subspace of
$L_{\mathrm{dil}}^{2}\left(  U\right)  $ which is a subspace of
$L^{2}\left(  U\right)  .$ Moreover,
\[
L_{\sigma}^{2}\left(  U\right)  =\overline{\left\{  v\in
C_{0}^{\infty}\left(  U\right)  :\operatorname{div}v=0\right\}
}^{\left\Vert \cdot\right\Vert _{L^{2}}},
\]
see \cite[Theorem III.2.3]{Galdi}. We consider
\[
\mathbb{G}\left(  U\right)  :=\left\{  w\in L_{\mathrm{dil}}%
^{2}\left(  U\right)  ,\left\langle w,u\right\rangle =0:\forall
u\in L_{\sigma}^{2}\left(  U\right)  \right\}
\]
so that 
\[
L_{\mathrm{dil}}^{2}\left(  U\right)  =L_{\sigma}%
^{2}\left(  U\right)  \oplus\mathbb{G}\left(  U\right)  .
\]
Moreover, we claim that $\mathbb{G}\left(  U\right)  $ is finite
dimensional (which implies that $L_{\mathrm{dil}}^{2}\left(  U\right)$ is closed in $L^2(U)$). Indeed, consider $u\in L^{2}_{\mathrm{dil}}\left(  U\right)  $ and
recall the Liouville decomposition
\[
u=\mathbb{P}u+\nabla q
\]
where $\mathbb{P}u\in L_{\sigma}^{2}\left(  U\right)  $ and
$q\in\dot{H}^{1}\left(  U\right) \cap L^2_{loc}(U)$  satisfies 
\begin{equation} \label{q}
\left\{
\begin{array}
[c]{l}%
-\Delta q=0,\text{ in }U\\
\frac{\partial q}{\partial \mathfrak{n}}=u\cdot \mathfrak{n}\text{ on }\partial U.
\end{array}
\right.    
\end{equation}
Thus, elements from $\mathbb{G}\left(  U\right)  $ are gradient fields
verifying%
\[
\left\{
\begin{array}
[c]{l}%
-\Delta q=0,\\
\frac{\partial q}{\partial \mathfrak{n}}=s_{i}+v_{i}\cdot \mathfrak{n}\text{ on }\partial
U_{i}  \text{ for all }i\in \overline{1,N}.
\end{array}
\right.
\]
Any element from $\mathbb{G}\left(  U\right)  $ can be written as a
linear combination of gradients of solutions of the equations
\begin{equation} \label{eq_qi}
\left\{
\begin{array}
[c]{l}%
-\Delta q_{i}=0\text{ on }U,\\
\frac{\partial q_i}{\partial \mathfrak{n}}=\delta_{ij}\text{ on }\partial U_{j},
\end{array}
\right.    
\end{equation}
for all $i\in \overline{1,N}$, and
\begin{equation} \label{eq_qik}
\left\{
\begin{array}
[c]{l}%
-\Delta q_{i}^{k}  =0\text{ on }U
,\\
\frac{\partial q_{i}^{k}  }{\partial \mathfrak{n}}=\delta_{ij}\mathfrak{n}^{k}\text{
on }\partial U_{j},
\end{array}
\right.
\end{equation}
for all $i\in \overline{1,N},$ $k\in \overline{1,3}.$ Thus these fields
generate $\mathbb{G}\left(  U\right)  $. Let us suppose that
\[
\sum_{i=1}^N  \alpha_{i}\nabla q_{i}+ \sum_{i=1}^N \sum_{k=1}^3\alpha_{i}^{k}\nabla q_{i}^{k} =0
\]
for some $\left(  \alpha_{i},\alpha_{i}^{k}\right)  _{i\in\overline{1,N},\text{ }k\in
\overline{1,3}}.$ Then, for any fixed $j\in\overline{1,N}$ and any $\ell\in\overline{1,3}$
we have that%
\begin{align*}
0 &  = \sum_{i}  \int_{\pa B_j} \alpha_{i} \pa_{\mathfrak{n}} q_{i}  + \sum_{i,k}  \int_{\pa B_j} \alpha_{i}^{k}\pa_{\mathfrak{n}} q_{i}^{k}  =\int_{\pa B_j} \alpha_{j} + \sum_k    \int_{\pa B_j} \alpha_j^k \mathfrak{n}^k \mathrm{d}S = \int_{\pa B_j} \alpha_{j}  ,\\
0 &  = \sum_{i}  \int_{\pa B_j} \alpha_{i}\pa_{\mathfrak{n}}  q_{i}  \, \mathfrak{n}^\ell + \sum_{i,k}  \int_{\pa B_j} \alpha_{i}^{k} \pa_{\mathfrak{n}} q_{i}^{k}    \, \mathfrak{n}^\ell = \int_{\partial U}\alpha_{j}^\ell (\mathfrak{n}^\ell)^2
\mathrm{d}S,
\end{align*}
relations that show the linear independence of $\left\{  \nabla q_{i},\nabla
q_{i}^{k}\right\}  _{i\in\overline{1,N},k\in\overline{1,3}}$. Thus the space
$\mathbb{G}\left(  U\right)  $ is indeed $4N$ dimensional. We can
orthonormalize it to produce, with a slight abuse of notations%
\[
\left\{  \nabla q_{i}\right\}  _{i\in\overline{1,4N}}%
\]
an orthonormal basis.
\begin{remark}
    The solution $q$ of system \eqref{q} is  defined up to an additive constant, which of course does not affect its gradient. However, it is known that there is a unique choice of this constant such that $q \in L^6(U)$, with inequality $\|q\|_{L^6} \le C \|\na q\|_{L^2}$. The same holds for $q_i, q_i^k$ solutions of \eqref{q1} and \eqref{q1k}. See \cite[Theorem II.6.1]{Galdi} for details. We make implicitly this choice in all what follows. 
\end{remark}

In our context of time dependent domains, $U_i$ will be replaced by $B_i = B_i(t)$ and $U$ by $\Omega = \Omega(t)$. Given an interval $I \subset \R$, we introduce the notation
\begin{equation} \label{def_QT}
Q(I) := \cup_{t \in I} \{t\} \times \Omega(t)       
\end{equation}
For  $1 \le p < \infty$, $1 \le q \le \infty$, we define the space 
\begin{equation*}
L^p(I ; L^q(\Omega)) := \left\{ u : Q(I)  \rightarrow \R^3 \: \text{ measurable}, \quad \int_0^T \| u(t, \cdot) \|_{L^q(\Omega(t))}^p dt < \infty \right\} 
\end{equation*}
One can define similarly $L^\infty(I ; L^q(\Omega))$, $\: L^p(I ; H^s(\Omega))$ for $s \in \N$. We further define 
$$ C^\infty_c(I ; L^q(\Omega)) := \left\{ 
\begin{aligned}
& u : Q(I) \rightarrow \R^3 \: \text{ measurable},  \:  \pa_t^k u \in L^\infty(I ; L^q(\Omega)) \: \forall k \in \N  \\
& u(t, \cdot) = 0 \text{ for $t$ outside a compact set of $I$}, 
\end{aligned}
\right\},      $$
One can define similarly  $C^\infty_c(I ; H^s(\Omega))$ for $s \in \N$. Eventually, we define for any $1 \le p \le \infty$
$$ L^p(I ; L^2_{\mathrm{dil}}(\Omega)) := \left\{  u \in L^p(I ; L^2(\Omega)), \: u(t) \in L^2_{\mathrm{dil}}(\Omega(t)) \: \text{ for a.e. } t \in I \right\}.  $$

\subsection{Definition of weak solutions}
In this paragraph, we introduce a notion of weak Leray type solution for system \eqref{main}. We  restrict to situations where the bubbles do not shrink to points, and do not collide or merge. This means that the families $X$ and $R$ in \eqref{defXR}  will always be considered on time intervals $[0,T)$ such that: for all $t \in (0,T)$, $\displaystyle (X(t),R(t))$ belongs to the admissible set 
 \begin{equation} \label{def_calA}
 \mathcal{A} := \left\{ (X_0,R_0) \in \R^{3N} \times (\R_+^*)^N, \: \forall i \neq j, \: |x_{0,i} - x_{0,j}| > r_{0,i} + r_{0,j}\right\}.    
 \end{equation}

As usual, the notion of weak solution is based on a variational formulation, deduced from formal multiplication of the momentum equation in \eqref{main} by a suitable test field $\psi$. Namely, we take $\psi$ such that
$$\psi \in C^\infty_c\left([0,T) ; H^s(\Omega)\right) \:  \forall s \in \N, \quad \psi(t) \in L^2_{\mathrm{dil}}(\Omega(t)) \:  \forall t \in [0,T) $$
and integrate over $\Omega(t)$. We get 
\begin{align*}
\int_{\Omega\left(t\right)}(\partial_{t}u +u \cdot \nabla u) \cdot \psi  &  =\int_{\Omega\left(  t\right)  }(\partial_{t}(u \cdot \psi)+ u \cdot \na (u \cdot  \psi))-\int_{\Omega  \left(  t\right)
} u \cdot (\partial_{t}\psi+ u \cdot \na \psi) \\
&  =\frac{d}{dt}\int_{\Omega\left(  t\right)  }(u \cdot \psi) -\int_{\Omega\left(  t\right)  }u \cdot (\partial_{t}\psi+ u \cdot \na \psi).
\end{align*}
The last equality comes from the following general formula about integrals on moving domains, to be used several times in the paper: for any function $f(t,\cdot): \Omega(t) \rightarrow \R$, and for $V_{\mathfrak{n}}$ the normal velocity of $\pa \Omega$: 
$$ \frac{d}{dt} \int_{\Omega(t)} f(t,x) dx = \int_{\Omega(t)} \pa_t f(t,x) dx + \int_{\pa \Omega(t)} f(t,x) V_{\mathfrak{n}}(t,x) d\sigma(x)  .$$
We recall that the normal velocity is defined as follows. Let $t \ge 0$, $x \in \pa \Omega(t)$. Given a neighborhood  $I_t$ of $t$  and a $C^1$ field $\gamma : I_t \rightarrow \R^3$ such that  $\gamma(s) \in \pa \Omega(s)$ for all $s \in I_t$ with $\gamma(t) = x$,  we define 
$$ V_\mathfrak{n}(t,x) := \gamma'(t) \cdot \mathfrak{n}(t,x) .$$
This formula  can be shown to be independent of the choice of $\gamma$. We refer to \cite[Chapter 2.5]{Pruss} for all necessary details.  In particular, if $\Theta_t$ is a diffeomorphism that sends $\Omega(0)$ to $\Omega(t)$, and $v$ is the associated velocity field: 
\begin{equation} \label{time_deriv_moving}
\begin{aligned}
 \frac{d}{dt} \int_{\Omega(t)} f(t,x) dx & = \int_{\Omega(t)} \pa_t f(t,x) dx + \int_{\pa \Omega(t)} f(t,x) V_{\mathfrak{n}}(t,x) d\sigma(x) \\ 
 & =  \int_{\Omega(t)} \pa_t f(t,x) dx + \int_{\pa \Omega(t)} f(t,x) v(t,x) \cdot \mathfrak{n}(t,x) d\sigma(x) \\
 & = \int_{\Omega(t)} \pa_t f(t,x) dx + \int_{\Omega(t)} \div(v f) dx .
\end{aligned}
\end{equation}
This is the identity that we used with $v = u$ and $f = u \cdot \psi$, taking into account that $\div (u (u \cdot \psi)) = u \cdot \na (u \cdot \psi)$ as $\div u = 0$. 

Now, note that  for all $i\in \overline{1,N}$, there exists $s_i \in \R$, $v_i \in \R^3$ such that 
\[
\psi  \cdot \mathfrak{n}   = s_i + v_{i} \cdot \mathfrak{n} \text{ on }
\pa B_{i}\left(  t\right).
\]
We compute, with $\mathrm{T} = \mathrm{T}(u,p)$:
\begin{align*}
&  \int_{\Omega\left(  t\right)  } (\na p - \nu \Delta u) \cdot \psi =  -\int_{\Omega\left(  t\right)  } \div \mathrm{T} \cdot \psi \\
 = &\int_{\Omega\left(  t\right)  }\mathrm{T} : \mathbb{D}\left(
\psi\right)  + \sum_{i=1}^{N}\int_{\partial B_{i}\left(
t\right)  }(\mathrm{T} \mathfrak{n}) \cdot \psi\mathrm{d} \sigma\\
 = & \int_{\Omega\left(  t\right)  }\mathrm{T} : \mathbb{D}\left(
\psi\right) +\sum_{i=1}^{N} \left[\int_{\partial B_{i}\left(
t\right)  }  (\mathrm{T} \mathfrak{n} \cdot \mathfrak{n}) (\psi \cdot \mathfrak{n}) \mathrm{d} \sigma + \int_{\partial B_{i}\left(  t\right)  }  (\mathrm{T} \mathfrak{n}\times \mathfrak{n})\cdot(\psi\times \mathfrak{n}) \mathrm{d} \sigma\right]\\
 = & \int_{\Omega\left(  t\right)  }\mathrm{T} : \mathbb{D}\left(
\psi\right) +\sum_{i=1}^{N} \left[\int_{\partial B_{i}\left(
t\right)  } (\mathrm{T} \mathfrak{n} \cdot \mathfrak{n})  \mathrm{d} \sigma \, s_i  + \int_{\partial B_{i}\left(
t\right)  } (\mathrm{T} \mathfrak{n} \cdot \mathfrak{n}) (v_i \cdot \mathfrak{n}) \mathrm{d} \sigma + \int_{\partial B_{i}\left(  t\right)  }  (\mathrm{T} \mathfrak{n}\times \mathfrak{n})\cdot(\psi\times \mathfrak{n}) \mathrm{d} \sigma\right]\\
  =  & \int_{\Omega\left(  t\right)  }\mathrm{T} : \mathbb{D}\left(
\psi\right) +\sum_{i=1}^{N} \left[\int_{\partial B_{i}\left(
t\right)  } (\mathrm{T} \mathfrak{n} \cdot \mathfrak{n})  \mathrm{d} \sigma \, s_i  + \int_{\partial B_{i}\left(
t\right)  } \mathrm{T} \mathfrak{n} \mathrm{d} \sigma \cdot v_i   + \int_{\partial B_{i}\left(  t\right)  }  (\mathrm{T} \mathfrak{n}\times \mathfrak{n})\cdot(\psi\times \mathfrak{n} - v_i \times \mathfrak{n}) \mathrm{d} \sigma\right].
\end{align*}
Taking into account the conditions on $\mathrm{T}$  in \eqref{main}, we find
\begin{align*}
 \int_{\Omega\left(  t\right)  } (\na p - \nu \Delta u) \cdot \psi  & = 2\nu\int_{\Omega\left(  t\right)  }\mathbb{D}(u) : \mathbb{D}\left(
\psi\right) - \sum_{i=1}^N c_i r_i^{
2 - 3\gamma} s_i \\
 = &  2\nu\int_{\Omega\left(  t\right)  }\mathbb{D}(u) : \mathbb{D}\left(
\psi\right) - \sum_{i=1}^N c_i r_i^{2- 3\gamma} \dashint_{\pa B_i(t)} (\psi \cdot \mathfrak{n})  \mathrm{d} \sigma.
\end{align*}
Integrating from $0$ to $t \le T$, we deduce the following variational formulation: 
\begin{equation} 
\begin{aligned}
\int_{\Omega\left(  t\right)  } u\left(  t\right)  \cdot\psi\left(
t\right)  +2\nu \int_{0}^{t}\int_{\Omega\left(  s\right)  }\mathbb{D}\left(
u\right)  \cdot\mathbb{D}\left(  \psi\right)  \mathrm{d}s - \sum_{i=1}^{N} \int_{0}^{t} c_i r_i^{2-3\gamma} 
\dashint_{\partial B_{i}\left(  s\right)  } \psi  \cdot \mathfrak{n}\mathrm{d}s\\
=\int_{\Omega(0)}  u_{0} \cdot\psi\left(0\right)
+\int_{0}^{t}\int_{\Omega\left(  s\right)  }  u \cdot (\partial_{t}\psi+ u \cdot \na \psi)\mathrm{d}s.
\end{aligned}
 \label{weak_form_fluid_bubble}
\end{equation}
This formal calculation leads to the following  
\begin{definition} ({\bf weak solutions})

\medskip
\noindent
Let $T > 0$, $(X_{0}, R_0) \in \mathcal{A}$, {\em cf.} \eqref{def_calA}.  Let   $\displaystyle \Omega(0) := \R^3 \setminus \cup_{i=1}^N \overline{B(x_{0,i}, r_{0,i})}$ and  
 $u_0 \in L^2_{\mathrm{dil}}(\Omega(0))$. 

\medskip
\noindent
A weak solution of \eqref{main} on $[0,T)$ with data $(X_0, R_0, u_0)$ is a triple $(X,R,u)$ such that 
\begin{itemize}
    \item[i)] $(X,R) \in W^{1,\infty}_{loc}([0,T); \R^{3N} \times \R^N)$ 
    \item[ii)] $(X(0),R(0)) = (X_0, R_0)$ and for all $t \in [0,T)$, $(X(t),R(t)) \in \mathcal{A}$.
    \item[iii)] $u \in L^\infty\left(0,T ; L^2_{\mathrm{dil}}(\Omega)\right)$, $ \: \na u \in L^2\left(0,T ; L^2(\Omega)\right)$, where $\Omega = \R^3 \setminus \cup_{i=1}^N \overline{B(x_i, r_i)}$. 
    \item[iv)] For all  $\psi\in L^{\infty}\left(0,T ; L_{\mathrm{dil}}^{2}(\Omega)\right)$ with $\partial_{t}\psi,\nabla\psi\in
L^2\left(0,T) ; L^{2}( \Omega)\right)$, the variational formulation \eqref{weak_form_fluid_bubble} holds for a.e. $t \in (0,T)$. 
\item[v)] For a.e. $t \in (0,T)$, for all $i \in \overline{1,N}$, $\: u \cdot \mathfrak{n} \vert_{\pa B_i} = \dot{x}_i \cdot \mathfrak{n} + \dot{r}_i$.  
\item[vi)] The inequality 
\begin{equation} \label{energy_inequality}
\frac12 \int_{\Omega(t)} |u(t)|^2 + 2\nu \int_{\Omega(t)} |\mathbb{D}(u)(s)|^2 \mathrm{d}s + \sum_{i=1}^N \frac{c_i}{3\gamma -3} r_i(t)^{3-3\gamma}  \le E_0 
\end{equation}
where 
\begin{equation} \label{def_E0}
   E_0 :=  \frac12 \int_{\Omega_0} |u_0|^2 + \sum_{i=1}^N \frac{c_i}{3\gamma -3} r_{0,i}^{3-3\gamma}   
\end{equation}
holds for almost every $t \in [0,T)$.
\end{itemize}
\end{definition}
\begin{remark}
We have chosen to consider settings in which the fluid domain $\Omega(0)$, and more generally $\Omega(t)$ for all $t$, is an exterior domain. As a consequence, given any initial family of dilation speeds $S_0 = (s_{0,i})_{i \in \overline{1,N}}$ and of translation velocities $V_0 = (v_{0,i})_{i \in \overline{1,N}}$, there exists  $u_0 \in L^2(\Omega(0))$ such that 
$$ \div u_0 = 0, \quad  u_0\cdot \mathfrak{n} \vert_{\pa B_i}  = s_{0,i} + v_{0,i} \mathfrak{n}, \quad \forall i \in \overline{1,N}. $$
Namely, one can take 
$$u_0 = \sum_{i=1}^N s_{0,i} \na q_i(0) + \sum_{i=1}^N \sum_{k=1}^3 (v_{0,i} \cdot e_k) \na q_i^k(0)$$
where the fields  $\na q_i(0)$ and $\na q_i^k(0)$ are the harmonic fields introduced in Section \ref{sec_functional_spaces}, in the special case $U = \Omega(0)$. This means that the space $L^2_{\mathrm{dil}}(\Omega(0))$ is reasonably  large. The situation would have been different in the case of a fluid in a bounded cavity, that is $\Omega(t) := U \setminus \cup_{i=1}^N B_i(t)$ with $U$ a bounded open domain. In such a case, under a natural non-penetration  condition $u_0 \cdot \mathfrak{n}\vert_{\pa U} = 0$  at the boundary of the cavity, one has the necessary compatibility condition: 
$$ \sum_{i=1}^n \int_{\pa B_i(0)} u_0 \cdot \mathfrak{n} =  \sum_{i=1}^n \int_{\pa B_i(0)} u_0 \cdot \mathfrak{n} - \int_{\pa U} u_0 \cdot \mathfrak{n} = - \int_{\Omega} \div u_0 = 0.$$
This condition, which can be shown to be also sufficient, reads  $\sum_{i=1}^N r_{0,i}^2 s_{0,i} = 0$. It is very stringent: it can in particular not be satisfied in the  case of a single bubble unless the bubble does not dilate or contract: incompressiblity of the liquid prevents any variation of volume. Considering an exterior domain allows to overcome this difficulty: the fluid somehow dilates or compresses at infinity. Another way out of this difficulty is to let the boundary of the cavity dilate or compress at the rate given by the bubbles. 
\end{remark}

\subsection{Existence of weak solutions up to collision}
Our main result is 
\begin{theorem} \label{main_thm}
Let $(X_{0}, R_0) \in \mathcal{A}$.  Let  $\displaystyle \Omega(0) := \R^3 \setminus \cup_{i=1}^N \overline{B(x_{0,i}, r_{0,i})}$ and   $u_0 \in L^2_{\mathrm{dil}}(\Omega(0))$. There exists $T > 0$ and a weak solution of \eqref{main} on $[0,T)$ with data $(X_0,R_0,u_0)$.  Moreover, either one can take $T$ arbitrary, or one can  choose $T$ such that
\begin{equation} \label{collision}
 \liminf_{t \rightarrow T_-} \: \inf_{i \neq j} \left(|x_i(t) - x_j(t)| - r_i(t) - r_j(t)\right) = 0.
 \end{equation}
\end{theorem} 
\begin{remark}
From the energy inequality \eqref{energy_inequality}, one can deduce a lower bound on $r_i(t)$ for all $i$: this rules out the collapse of a bubble on a single point. This helps to understand  why the limitation on the time of existence of weak solutions may only occur from \eqref{collision}, that is from the collision between bubbles. 
\end{remark}
\begin{remark} \label{rem_H2}
    It will be shown in Section \ref{sec_a_priori} that the positions and radii of the particles $(X,R)$ are more than Lipschiz in time: they obey $H^2$ bounds. This feature will be crucial to the derivation of strong compactness of the velocity field. 
\end{remark}

\section{Strategy of proof} \label{sec_strategy}
We explain here the main steps of the proof of Theorem \ref{main_thm}. First, we will tackle in Section \ref{sec_prelim}  the derivation of {\em a priori} estimates and the proof of sequential continuity for weak solutions of system \eqref{main}.The decomposition 
$$ L^2_{\mathrm{dil}}(U) = L^2_\sigma(U) \oplus \mathbb{G}(U)$$
introduced in Section \ref{sec_functional_spaces} will be crucial here. It will be applied with $U = \Omega(t)$ and $U = \Omega^n(t)$, leading to a decomposition $u = \tilde u + w$, $u^n = \tilde u^n + w^n$, with $\tilde u$, resp. $\tilde u^n$ tangent at the boundary of $\pa \Omega$, resp. $\pa \Omega^n$, and $w,w^n$ gradient fields. While the compactness of $\tilde u^n$ will rely on a version  of Aubin-Lions compactness result in moving domains from  \cite{Moussa}, the compactness of $w^n$ will notably rely on $H^2$ uniform estimates on $(X^n,R^n)$, see Remark \ref{rem_H2}, and on the analysis of the basis $\na q_i^n$, $i \in \overline{1,4N}$ seen in Section \ref{sec_functional_spaces}. 

The results of Section \ref{sec_prelim} will  not be used as such  in the subsequent sections, but will be applied with appropriate modifications to approximate solutions.  As  emphasized in Section \ref{sec1}, standard  approximation schemes, based on  variational formulations on $\R^3$ with penalization terms, seem hard to adapt. Therefore, we  rather follow a time-stepping scheme. The broad idea is to divide the interval $(0,T)$ in small subintervals of length $h \ll 1$. Then, iteratively on each time interval $[kh,(k+1)h]$, one considers a system analogue to \eqref{main},  where: 
{\em i) the trajectory of the bubbles is prescribed, ii)  this trajectory is based on the velocity field computed at the previous step on $[(k-1)h,kh]$}. We will in this way obtain an approximate solution $u^h$, which will converge to a weak solution of \eqref{main} as $h \rightarrow 0$. This iteration and convergence process  will be described briefly in the next three paragraphs, and developped rigorously in Sections \ref{sec_prescribed} and \ref{sec_approx_compact}. 

\subsection{Simplified model with prescribed bubble dynamics.} \label{sec_simpler_model}
As alluded to above, our construction of weak solutions relies on the study of a subproblem where the dynamics of the bubbles is prescribed. Namely, we give ourselves the two fields 
$$X = (x_i)_{i \in \overline{1,N}} : \R_+ \rightarrow \R^{3N}, \quad R = (r_i)_{i \in \overline{1,N}} : \R_+ \rightarrow (\R_+^*)^{N}.$$
More on $(X,R)$ will be specified just below. We define 
$$ B_i[X,R] := B(x_i,r_i), \quad i \in \overline{1,N}, \quad \Omega[X,R] := \R^3 \setminus \cup_{i=1}^N \overline{B_i[X,R]}.$$
Let $u_0 \in L^2_{\mathrm{dil}}(\Omega[X,R](0))$ and $\tilde{R}_0 = (\tilde{r}_{0,i})_{i \in \overline{1,N}}\in (\R_+^*)^N$.  We consider the following variational formulation:  find $T > 0$ and $u$ such that for all  $\psi\in L^{\infty}\left(0,T ; L_{\mathrm{dil}}^{2}(\Omega[X,R])\right)$ with $\partial_{t}\psi,\nabla\psi\in
L^2\left(0,T ; L^{2}( \Omega[X,R])\right)$,

\begin{equation}  \label{VF_prescribed}
\begin{aligned}
& \int_{\Omega[X,R]\left(  t\right)  } u\left(  t\right)  \cdot\psi\left(
t\right)  +2\nu \int_{0}^{t}\int_{\Omega[X,R]\left(  s\right)  }\mathbb{D}\left(
u\right)  \cdot\mathbb{D}\left(  \psi\right)  \mathrm{d}s \\
& - \sum_{i=1}^{N}  \int_0^t c_i \left( r_i[u] \right)^{2-3\gamma} 
\dashint_{\partial B_{i}[X,R]\left(  s\right)  } \psi  \cdot \mathfrak{n}\mathrm{d}s\\
& - \frac12 \sum_{i=1}^{N} \int_{0}^{t}  \int_{\partial B_{i}[X,R]\left(  s\right)  } \left(u \cdot \mathfrak{n} - \dot{r}_i - \dot{x}_i \cdot \mathfrak{n}\right) (\psi  \cdot u) \mathrm{d}s \\
& =\int_{\Omega[X,R](0)}  u_{0} \cdot\psi\left(0\right)
+\int_{0}^{t}\int_{\Omega[X,R]\left(  s\right)  }  u \cdot (\partial_{t}\psi+ u \cdot \na \psi)\mathrm{d}s,
\end{aligned}
\end{equation}
where 
\begin{equation}  \label{def_ri_u}
   r_i[u] :=  \tilde{r}_{0,i} + \int_0^s \dashint_{\partial B_{i}[X,R]\left(  s'\right)} u \cdot \mathfrak{n} \,  \mathrm{d}s'.
\end{equation}
For later purpose, we also define 
\begin{equation}  \label{def_xi_u}
   x_i[u] :=  x_i(0) + 3 \int_0^s \dashint_{\partial B_{i}[X,R]\left(  s'\right)} u \cdot \mathfrak{n} \, \mathfrak{n}  \mathrm{d}s'.
\end{equation}
A key block of our analysis is the following theorem 
\begin{theorem} \label{thm_prescribed}
Let $X,R$ {\em affine} fields, meaning that the time derivatives $\dot{X}$, $\dot{S}$ are constants. Let $T > 0$ such that 
\begin{equation} \label{separation_condition}
\delta := \frac{1}{4}\inf_{[0,T]} \: \inf_{i \neq j } \, \left(|x_i - x_j| - r_i - r_j \right) > 0    
\end{equation}
Let $\displaystyle u_0 \in L^2_{\mathrm{dil}}(\Omega[X,R](0))$ and $\tilde{R}_0 \in (\R_+^*)^N$. 
There exists $\displaystyle u  \in L^\infty\left(0,T ; L^2_{\mathrm{dil}}(\Omega[X,R])\right)$, such that $\na u \in L^2\left(0,T ; L^2(\Omega[X,R])\right)$ and  such that for all $\psi\in L^{\infty}\left(0,T ; L_{\mathrm{dil}}^{2}(\Omega[X,R])\right)$ with $\displaystyle \partial_{t}\psi,\nabla\psi\in
L^2\left(0,T ; L^{2}( \Omega[X,R])\right)$, \eqref{VF_prescribed} holds for a.e. $t \in [0,T)$. 

\medskip
Moreover, the following energy inequality holds for a.e. $t \in (0,T)$:
\begin{equation*}
\frac12 \int_{\Omega[X,R](t)} |u(t)|^2 + 2 \nu \int_{\Omega[X,R](t)} |\mathbb{D}(u)(s)|^2 \mathrm{d}s + \sum_{i=1}^N \frac{c_i}{3\gamma -3} r_i[u](t)^{3-3\gamma}  \le \tilde E_0 
\end{equation*}
where $\tilde E_0 := \frac12 \int_{\Omega[X_0,R_0]} |u_0|^2 + \sum_{i=1}^N \frac{c_i}{3\gamma -3} \tilde r_{0,i}^{3-3\gamma}$. 
\end{theorem}
Several remarks are in order. 
\begin{remark}
    The variational formulation \eqref{VF_prescribed} is  similar to \eqref{weak_form_fluid_bubble}, but with a major difference: there is a mismatch between the normal velocity of $\pa B_i[X,R](t)$, which is prescribed and given by $\dot{r}_i + \dot{x}_{i}\cdot \mathfrak{n}$,  and the normal fluid velocity $u \cdot \mathfrak{n}\vert_{\pa B_i[X,R](t)}$, which is given by $\dot{r}_i[u] + \dot{x}_{i}[u]\cdot \mathfrak{n}$, see \eqref{def_ri_u}-\eqref{def_xi_u},  taking into account that $u(t) \in L^2_{\mathrm{dil}}(\Omega(t))$.  This mismatch explains the need for the extra boundary term  $\frac12 \sum_{i=1}^{N} \int_{0}^{t}  \int_{\partial B_{i}[X,R]\left(  s\right)  } \left(u \cdot \mathfrak{n} - \dot{r}_i - \dot{x}_i \cdot \mathfrak{n}\right) (\psi  \cdot u) \mathrm{d}s$. The change in the pressure contribution $ \sum_{i=1}^{N}  \int_0^t c_i \left( r_i[u] \right)^{2-3\gamma} 
\dashint_{\partial B_{i}[X,R]\left(  s\right)  } \psi  \cdot \mathfrak{n}\mathrm{d}s$ is also related to this mismatch. These modifications ensure (at least formally) that  taking $\psi = u$ in \eqref{VF_prescribed} and applying the formula \eqref{time_deriv_moving}, one recovers the energy inequality. 
\end{remark}
\begin{remark}
As explained in the previous remark, due to the mismatch between the normal components of the fluid and the bubbles at their interfaces, one had to add an extra boundary term and modify the pressure contribution so as to preserve the energy inequality. 
This modifies in turn the strong formulation associated to \eqref{VF_prescribed}. Assuming that $u$ and $\psi$ are smooth and integrating by parts in \eqref{VF_prescribed}, one obtains after tedious manipulations the following modified boundary conditions (in addition to the Navier-Stokes equation in the fluid domain): 
\begin{equation} \label{strong_formulation_prescribed}
\begin{aligned}
& \mathrm{T}(u,p) \mathfrak{n} \times \mathfrak{n}\vert_{\pa B_i[X,R]} - \frac{1}{2}( \left(u \cdot \mathfrak{n} - \dot{r}_i - \dot{x}_i \cdot \mathfrak{n}\right) u \times \mathfrak{n}\vert_{\pa B_i[X,R]} = 0, \\
& \dashint_{\pa B_i[X,R]} \mathrm{T}(u,p) \mathfrak{n} \cdot \mathfrak{n} - \frac{1}{2} \dashint_{\pa B_i[X,R]} \left(u \cdot \mathfrak{n} - \dot{r}_i - \dot{x}_i \cdot \mathfrak{n}\right) u \cdot \mathfrak{n} = - \frac{r_i[u]^2}{r_i^2} p_i, \\
& \dashint_{\pa B_i[X,R]} \mathrm{T}(u,p) \mathfrak{n} - \frac{1}{2} \dashint_{\pa B_i[X,R]} \left(u \cdot \mathfrak{n} - \dot{r}_i - \dot{x}_i \cdot \mathfrak{n}\right) u = 0. 
\end{aligned}
\end{equation}
Let us note that similar  nonlinear boundary conditions are sometimes used as outflow conditions in the simulation of channel flows, see \cite[chapter 7]{Boyer}. 
\end{remark}
\begin{remark}
Theorem \ref{thm_prescribed} is reminiscent of article \cite{Neustupa}, where the Navier-Stokes equation is studied outside a set of rigid obstacles with {\em prescribed trajectories}, and with  {\em Navier slip conditions} at the boundary of the obstacles.  However,  the normal components of the fluid and of the obstacles are taken equal in \cite{Neustupa}. This is possible because  there is no extra conditions such as (\ref{strong_formulation_prescribed}b)-(\ref{strong_formulation_prescribed}c) to be fulfilled. In our case, these conditions determine $\dot{r}_i[u]$, $\dot{x}_i[u]$, which are then different from  $\dot{r}_i$, $\dot{x}_i$. Also, as will be seen in Section \ref{sec_prescribed}, we will use a Galerkin scheme to solve \eqref{VF_prescribed}, instead of a time discretization method like in \cite{Neustupa}. 
\end{remark}

\subsection{Construction of approximate solutions through time stepping.} \label{sec_time_stepping}
To obtain a weak solution of \eqref{main} on $[0,T)$ for some $T > 0$, we will rely on a time discretization, where each time step will be handled thanks to Theorem \ref{thm_prescribed}. Let $h$ such that $\frac{T}{h}$ belongs to $\N$. The idea is to build an approximation $(X^h,R^h,u^h)$  iteratively on each $[kh, (k+1)h]$, $k = 0 \dots \frac{T}{h}-1$, as follows.   
\begin{itemize}
    \item[i)] {\em Initialization}.  Let $(X_0,R_0,u_0)$ as in Theorem \ref{main_thm}, and $\tilde{R}_0 = R_0$. In particular, 
    \begin{equation} \label{def_delta0}
        \delta_0 := \frac{1}{4}\inf_{i \neq j} \left( |x_{0,i} - x_{0,j}| - r_{0,i} - r_{0,j}  \right) > 0. 
    \end{equation}
    Let then 
    $$ X(t) := X_0, \quad R(t) := R_0, \quad t \in \R_+$$
    and  $u = u[X,R,u_0,\tilde{R}_0]$ as in Theorem \ref{thm_prescribed}. We set: 
    $$(X^h,R^h,u^h) := (X,R,u) \: \text{ on } [0,h].$$ 
    \item[ii)] Induction step. Let $1  \le k \le \frac{T}{h}-1$.  Assume that we have defined $(X^h,R^h,u^h)$ on $[0,kh]$, with $X^h, R^h$ continuous and piecewise affine, and  $u^h \in L^\infty(0,kh ; L_{\mathrm{dil}}^{2}(\Omega[X^h,R^h])$. For a.e. $t \in [0,kh)$, 
       $$u^h(t) \cdot \mathfrak{n}\vert_{\pa B(x^h_i(t), r^h_i(t))} = \dot{r}_i[u^h](t) + \dot{x}_i[u^h](t)  \cdot \mathfrak{n}, \quad i \in \overline{1,N}$$ 
       with $r_i[u^h]$,  $x_i[u^h]$ defined in \eqref{def_ri_uh}-\eqref{def_xi_uh}. Denoting 
       $$R[u^h] = \left(r_i[u^h]\right)_{ i \in \overline{1,N}},  \quad X[u^h] = \left(x_i[u^h]\right)_{ i \in \overline{1,N}},$$
       we  define, for all $t \in [0,h]$,
    \begin{equation} \label{def_iter_R_X}
    R(t) := R^h(kh) + \left( \dashint_{(k-1)h}^{kh} \dot{R}[u^h] \right) t, \quad X(t) := X^h(kh) + \left( \dashint_{(k-1)h}^{kh} \dot{X}[u^h] \right) t.
     \end{equation}
   We also define $u_0$ and $\tilde{R}_0 = (r_{0,i})_{1 \le i \le N}$ by the formulas
    \begin{align*}
    u_0  := u^h(kh), \quad \tilde{R}_{0}  := R[u^h](kh).
    \end{align*}
    Let  $u = u[X,R,u_0,\tilde{R}_0]$ as in Theorem \ref{thm_prescribed}. We set: 
    $$(X^h,R^h,u^h)(t) := (X,R,u)(t-kh) \: \text{ on } [kh,(k+1)h].$$ 
\end{itemize}
The point is to show that this formal iteration can be  implemented. Note that, strictly speaking,  $u^h$ may not be defined at $t =k h$, as it is defined only a.e. on $[0, kh]$. In such a case, one replaces  in the above formulae  $kh$ by some $t_{k,h} \approx kh$, say $t_{k,h} \in [kh - \frac{1}{2^k}h]$, such that $u^h(t_{k,h})$ is well-defined. We shall not comment more on this technicality.  The main result of this iterative approach is the following 
\begin{proposition} \label{prop_uh}
Let $(X_0,R_0,u_0)$ as in Theorem \ref{main_thm}, $E_0$ and $\delta_0$ defined in \eqref{def_E0} and \eqref{def_delta0}. There exists $T_0 > 0$  such that for all $T \le T_0$, for all $h$ with $\frac{T}{h} \in \N$, the induction process above works and provides a triple $(X^h,R^h,u^h)$ such that $u^h  \in L^\infty\left(0,T ; L^2_{\mathrm{dil}}(\Omega[X^h,R^h])\right)$,  $\na u^h \in L^2\left(0,T ; L^2(\Omega[X^h,R^h])\right)$ and  such that for all $\psi^h \in L^{\infty}\left(0,T ; L_{\mathrm{dil}}^{2}(\Omega[X^h,R^h])\right)$ with $\partial_{t}\psi^h,\nabla\psi^h\in
L^2\left(0,T ; L^{2}( \Omega[X^h,R^h])\right)$,  for a.e. $t \in [0,T)$,   
\begin{equation}  \label{VF_prescribed_h}
\begin{aligned}
& \int_{\Omega[X^h,R^h]\left(  t\right)  } u^h\left(  t\right)  \cdot\psi^h\left(
t\right)  +2\nu \int_{0}^{t}\int_{\Omega[X^h,R^h]\left(  s\right)  }\mathbb{D}\left(
u^h\right) : \mathbb{D}\left(  \psi^h\right)  \mathrm{d}s \\
& - \sum_{i=1}^{N}  \int_0^t c_i \, \left(r_i[u^h]\right)^{2 - 3 \gamma} \,  \dashint_{\partial B_{i}[X^h,R^h]\left(  s\right)  } \psi^h  \cdot \mathfrak{n} \, \mathrm{d}s\\
& - \frac12 \sum_{i=1}^{N} \int_{0}^{t}  \int_{\partial B_{i}[X^h,R^h]\left(  s\right)  } \left(u^h \cdot \mathfrak{n} - \dot{r}_i^h - \dot{x}_i^h \cdot \mathfrak{n}\right) (\psi^h  \cdot u^h) \mathrm{d}s \\
& =\int_{\Omega(0)}  u_{0} \cdot\psi\left(0\right)
+\int_{0}^{t}\int_{\Omega[X^h,R^h]\left(  s\right)  }  u^h \cdot (\partial_{t}\psi^h + u^h \cdot \na \psi^h)\mathrm{d}s,
\end{aligned}
\end{equation}
where 
\begin{equation} \label{def_ri_uh}
    r_i[u^h](t) := r_{0,i} + \int_{0}^t \dashint_{\pa B_i[X^h,R^h](s)} u^h \cdot \mathfrak{n}   \, \mathrm{d}s.
\end{equation}
Moreover, the following energy inequality holds for a.e. $t \in (0,T)$:
\begin{equation*}
\frac12 \int_{\Omega[X^h,R^h](t)} |u^h(t)|^2 + 2 \nu \int_{\Omega[X^h,R^h](t)} |\mathbb{D}(u^h)(s)|^2 \mathrm{d}s + \sum_{i=1}^N \frac{c_i}{3\gamma -3} r_i[u^h](t)^{3-3\gamma}  \le E_0 .
\end{equation*}
\end{proposition}
We recall that $E_0$ was defined in \eqref{def_E0}. For later purpose, we also define 
\begin{equation} \label{def_xi_uh}
   x_i[u^h](t) :=  x_{0,i}  + 3 \int_0^s \dashint_{\partial B_{i}[X,R]\left(  s'\right)} u^h \cdot \mathfrak{n} \, \mathfrak{n}  \mathrm{d}s'.
\end{equation}
The proof of Proposition \ref{prop_uh} will be given in Section \ref{sec_prescribed}. Namely, it means that for $T \le T_0$ {\em small enough}, one can apply Theorem \ref{thm_prescribed}  on each subinterval $[kh,(k+1)h]$ of $[0,T]$, in the way indicated in the induction step. This is not obvious, as the theorem requires that the bubbles $B_i(t)$ remain away from one another, {\it cf.} condition \eqref{separation_condition}. We need to show that this separation condition holds uniformly over the subintervals. This is where the smallness of $T_0$ is required. 


\subsection{Compactness of $u^h$ and conclusion.} The last main step of the proof is to show that any accumulation point of $(u^h)_{h > 0}$ as $h \rightarrow 0$ is a weak solution of \eqref{main}. The whole analysis of Section \ref{sec_prelim} is needed here. It provides a weak solution {\em in short time}, {\em a priori} limited by the condition $T \le T_0$ in Proposition \ref{prop_uh}. The final part of the proof is to establish that this solution can be extended in time until the first collision between bubbles. 

\section{Preliminaries} \label{sec_prelim}
We collect in this part properties of weak solutions of system \eqref{main}. These properties will still be shared by the solutions of the approximate systems that we will introduce later, although their proof will require a few changes.  We believe that to present these properties in the simpler context of system \eqref{main} will help to understand the next sections. 

\subsection{A priori estimates} \label{sec_a_priori}
Let $T > 0$, $\delta > 0$, $(X_0,R_0,u_0)$ as in Theorem \ref{main_thm}. We assume that there exists $(X,R,u)$ a weak solution of \eqref{main} on $[0,T)$ with data  $(X_0,R_0,u_0)$, such that
$$ \forall i \neq j, \quad |x_i - x_j| \ge  r_i + r_j + 4 \delta. $$
We have the energy inequality: for a.e.  $t < T$ 
\begin{equation} \label{a_priori_energy}
\frac12 \int_{\Omega(t)} |u(t)|^2 + 2 \nu \int_0^t \int_{\Omega(t)} |\mathbb{D}(u)(s)|^2 \mathrm{d}s + \sum_{i=1}^N \frac{c_i}{3\gamma -3} r_i(t)^{3-3\gamma}  \le  E_0 
\end{equation}
where 
$$ E_0 := \frac12 \int_{\Omega_{0,F}} |u_0|^2 + \sum_{i=1}^N \frac{c_i}{3\gamma -3} r_{0,i}^{3-3\gamma}    $$
which provides uniform bounds on $\|u\|_{L^\infty(0,T ; L^2(\Omega))}$ and $\|\na u\|_{L^2(0,T ; L^2(\Omega))}$, as well as the uniform lower bound 
\begin{equation} \label{lower_bound_R}
  \forall i \in \overline{1,N},  \: \forall t \le T, \quad  r_i(t) \ge c(E_0) > 0.  
\end{equation}

\paragraph{Lipschitz bounds on $X,R$.}

 We introduce for all $t\in\left(
0,T\right)$, $i\in 1,N$ \ the solution
$\varphi_i^\delta$ of
\begin{equation}
\left\{
\begin{aligned}
-\Delta\varphi_i^\delta\left(  t\right)  & =0\text{ on } C_i(t) := \left\{ r_i\left(
t\right)  <\left\vert x-x_i\left(  t\right)  \right\vert <r_i\left(
t\right)  +\delta\right\},\\
\varphi_i^\delta\left(  t\right) & =0\text{ on }\left\vert x-x_i\left(
t\right)  \right\vert =r_i\left(  t\right)  +\delta,\\
\varphi_i^\delta \left(  t\right)  & =1\text{ on }\left\vert x-x_i\left(
t\right)  \right\vert =r_i\left(  t\right)  .
\end{aligned}
\right.  \label{phi_delta}%
\end{equation}
which is given explicitly by%
\[
\varphi_i^\delta\left(  t,x\right)  =\frac{(\delta+r_i\left(  t\right)
)r_i\left(  t\right)  }{\delta}\frac{1}{\left\vert x-x_i\left(  t\right)
\right\vert }-\frac{r_i\left(  t\right)  }{\delta}.
\]
Note that $0 \le \varphi_i^\delta \le 1$. Furthermore,
\[
\nabla\varphi_i^\delta\left(  t,x\right)  =-\frac{(\delta+r_i\left(
t\right)  )r_i\left(  t\right)  }{\delta}\frac{x-x_i\left(  t\right)
}{\left\vert x-x_i\left(  t\right)  \right\vert ^{3}}%
\]
from which we deduce that%
\[
\left\Vert \nabla\varphi_i^\delta\left(  t\right)  \right\Vert
_{L^{\infty}(C_i(t))}\leq\left(  \frac{1}{r_i\left(  t\right)  }+\frac
{1}{\delta}\right)  \text{ and }\left\Vert \nabla\varphi_{i}^{\delta}\left(
t\right)  \right\Vert _{L^{2}(C_{i}(t))}\leq(r_i\left(  t\right)
)^{\frac{3}{2}}\left(  \frac{1}{r_i\left(  t\right)  }+\frac{1}{\delta
}\right)  .
\]
From the first bound and the inequality 
$$ \left(\frac{c_i}{(3 \gamma -3) E_0}\right)^{\frac{1}{3\gamma-3}} \le r_i(t)  $$
deduced from \eqref{a_priori_energy}, we get
\begin{align*}
& 4\pi\left(  \frac{c_i}{(3\gamma-3)E_0}\right)^{\frac{2}{3\gamma-3}}  \left\vert \dot{r}_{i}\left(
t\right)  \right\vert    \leq4\pi r_{i}^{2}\left(  t\right)  \left\vert
\dot{r}_{i}\left(  t\right)  \right\vert =\left\vert \int_{\partial
B_{i}\left(  t\right)  }u\left(  t\right)  \cdot \mathfrak{n}\mathrm{d}\sigma\right\vert
\\
&  \leq\left\Vert \operatorname{div}(\varphi_{i}^\delta u)\right\Vert
_{L^{2}\left(  C_{i}\left(  t\right)  \right)  } \| \leq \left\Vert \nabla\varphi_i^\delta\left(  t\right)  \right\Vert
_{L^{\infty}(C_i(t))} \|u(t)\|_{L^2(\Omega(t))}\leq \sqrt{2 E_{0}} \left(  \frac{1}{\delta}+\frac{1}{r_i}\right)  .
\end{align*}
Combining this last inequality with \eqref{lower_bound_R} we deduce the uniform bound  
\begin{equation}  \label{upper_bound_dotR}
    |\dot{R}(t)| \le C(E_0, \delta).
\end{equation}
It follows that
\begin{equation} \label{upper_bound_R}
|R(t)| \le C(R_0, E_0, \delta,T).
\end{equation}
We proceed similarly for the estimate of $X\left(  t\right)$, using the relation
\[
\frac{4\pi}{3} (r_i(t))^{2} \dot{x}_{i}\left(  t\right)  =\int_{\partial
B_i\left(  t\right)  }(u\left(  t\right)  \cdot \mathfrak{n})\mathfrak{n} \mathrm{d}\sigma.
\]
We obtain that%
\begin{align*}
&  \frac{4\pi}{3}\left(  \frac{c_i}{(3\gamma-3)E_{0}}   \right)^{\frac{2}{3\gamma-3}} \left\vert \dot{x}_{i}\left(t\right)  \right\vert \\
&  \leq\left\vert \int_{\partial B_i\left(  t\right)  }\left(  u\left(
t\right)  \cdot \mathfrak{n}\right)  \mathfrak{n} \mathrm{d}\sigma\right\vert \leq\left\Vert
(\varphi_i^\delta u) \cdot \mathfrak{n} \right\Vert _{H^{-\frac{1}{2}}\left(  \partial
B_i\left(  t\right)  \right)  }\left\Vert \varphi_i^\delta\left(
t\right) \mathfrak{n}  \right\Vert _{H^{\frac{1}{2}}\left(  \partial B_i\left(  t\right)
\right)  }\\
&  \leq C \left( \left\Vert \operatorname{div}(\varphi_i^\delta u)\right\Vert
_{L^{2}\left(  C_i\left(  t\right)  \right)  } + \left\Vert (\varphi_i^\delta u)\right\Vert
_{L^{2}\left(  C_i\left(  t\right)  \right)  }   \right) \left\Vert \varphi_i^\delta\left(  t\right) \mathfrak{n}  \right\Vert _{H^{1}\left(  C_i\left(  t\right)
\right)  } \\
& \leq C \left\Vert \varphi
_{i}^\delta\left(  t\right) \mathfrak{n} \right\Vert _{H^{1}\left(  C_i\left(
t\right)  \right)  }.
\end{align*}
The last term is estimated as follows. First, as $0 \le \varphi_i^\delta \le 1$,
\[
\left\Vert \varphi_i^\delta\left(  t\right)   \mathfrak{n} \right\Vert _{L^{2}\left(
C_{i}^n\left(  t\right)  \right)  }^{2}=\int_{C_{i}^n\left(  t\right)  }\left\vert
\varphi_i^\delta\right\vert ^{2}\leq C   r_{i}^n(t)^2 \leq C(R_0,E_0,\delta,T)
\]
thanks to \eqref{upper_bound_R}. Second, 
\[
\left\Vert \nabla(\varphi_i^\delta\left(  t\right) \mathfrak{n} )\right\Vert
_{L^{2}\left(  C_i\left(  t\right)  \right)  }^{2}=\left\Vert (\nabla
\varphi_i^\delta(t))  \mathfrak{n} \right\Vert _{L^{2}\left(  C_i\left(
t\right)  \right)  }^{2}+\left\Vert \varphi_i^\delta\left(
t\right)  \nabla \mathfrak{n}\ \right\Vert _{L^{2}\left(  C_i\left(  t\right)  \right)  }^{2}.
\]
Using the  $L^2$ estimate of $\varphi_i^\delta$ just above, the previous $L^2$ estimate of $\na \varphi_i^\delta$  together with \eqref{upper_bound_R}, we find 
$$ \left\Vert \nabla(\varphi_i^\delta\left(  t\right) \mathfrak{n}  )\right\Vert
_{L^{2}\left(  C_i\left(  t\right)  \right)  }^{2} \le C(R_0,E_0,\delta,T)  $$
and eventually
\begin{equation} \label{upper_bound_dotX}
|\dot{X}(t)| \le C(R_0,E_0,\delta,T)    
\end{equation}
so that in turn
\begin{equation} \label{upper_bound_X}
|X(t)| \le C(X_0,R_0,E_0,\delta,T).
\end{equation}

\paragraph{$H^2(0,T)$ bounds on $X,R$.}
We rely here on the space of gradient fields $\mathbb{G}(U)$ introduced in 
Section \ref{sec_functional_spaces}, and on the families $(\na q_i)_{1 \le i \le N}$ and $(\na q_i^k)_{1 \le i \le N, 1 \le k \le 3}$ defined through \eqref{eq_qi}-\eqref{eq_qik}. We shall consider $U = \Omega(t)$, and denote $q_i = q_i(t)$, $q_i^k = q_i^k(t)$. We claim that we have for all $i \in \overline{1,N}$, for all $i \in \overline{1,3}$, for all $t \in (0,T)$
\begin{align}
\left\Vert \nabla q_{i}\left(  t\right)  \right\Vert _{H^1\left(
\Omega\left(  t\right)  \right)  }+\left\Vert \partial_{t}\nabla
q_{i}\left(  t\right)  \right\Vert _{L^{2}\left(  \Omega\left(  t\right)
\right)  }  &  \leq  C\left(X_0,R_0,E_0,\delta,T\right)  ,\label{q1}\\
\left\Vert \nabla q_{i}^{k}\left(  t\right)  \right\Vert _{H^1\left(
\Omega\left(  t\right)  \right)  }+\left\Vert \partial_{t}\nabla q_{i}%
^{k}\left(  t\right)  \right\Vert _{L^{2}\left(  \Omega\left(  t\right)
\right)  }  &  \leq C\left(X_0,R_0,E_0,\delta,T\right) . \label{q1k}%
\end{align}
The estimates of $\na q_i, \na q_i^k$ are standard. The estimates of  $\pa_t \na q_i, \pa_t \na q_i^k$ are related to the more general question of differentiating with respect to the domain.   In the case of Neumann problems like \eqref{q1}-\eqref{q1k}, we refer to \cite[chapter p. 202]{Henrot} for an in-depth analysis. 
The dependence on $X_0,R_0,E_0,T,\delta$ of the constants at the right-hand side comes from the dependence on the domain $\Omega(t)$ of these  estimates. Both estimates involve upper and lower bounds on the radii of the balls and the distance between the balls, while the time derivative estimate involves additionaly an upper bound on  $\dot{X},\dot{R}$.  All these quantities are controlled in terms of $X_0,R_0,E_0,T,\delta$ by previous estimates. 

By the Gram-Schmidt orthogonalization process we construct an orthonormal basis in
$\mathbb{G}(\Omega(t))$ (for the $L^{2}$ scalar product) which by slightly abusing the
notation, we denote $\left\{  \nabla q_{i}(t)\right\}  _{i\in
\overline{1,4N}}$. These functions still satisfy  \eqref{q1} (this time for all $i \in \overline{1,4N}$). Let  $s\left(  t\right)  \in C_{0}^{\infty}\left(  0,T\right)$. We observe that the above estimates allow us to use $\psi = s\left(  t\right)  \nabla
q_{i}\left(  t\right)$ in \eqref{weak_form_fluid_bubble}  to obtain
\begin{align*}
&  \int_{0}^{T}s\left(  t\right)  \int_{\Omega\left(  t\right)
}\mathbb{D}\left(  u\right)  :\mathbb{D}\left( \na q_{i}\right)
\mathrm{d}t-\sum_{j=1}^{N}\int_{0}^{T} \frac{c_j}{4\pi r_{j}\left(  t\right)  ^{3\gamma}}s\left(  t\right) \int_{\partial B_{j}\left(  t\right)
}
\nabla q_{i}\left(  t\right)  \cdot \mathfrak{n}\mathrm{d}t\\
&  =\int_{0}^{T}\int_{\Omega\left(  t\right)  }\left[  \dot
{s}\left(  t\right)  u\left(  t\right) \cdot \nabla q_{i}\left(  t\right)
+s\left(  t\right)  u\left(  t\right)  \cdot\partial_{t}\nabla q_{i}\left(
t\right)  +s\left(  t\right)  u\left(  t\right)  \otimes u\left(  t\right)
:\nabla^{2}q_{i}\left(  t\right)  \right]  \mathrm{d}t.
\end{align*}
We get%
\begin{equation}
\left\vert \int_{0}^{T}\int_{\Omega\left(  t\right)  }\dot{s}\left(
t\right)  u\left(  t\right) \cdot \nabla q_{i}\left(  t\right)  \right\vert \leq
C\left(X_0,R_0,E_0,\delta,T\right)  \left\Vert s\right\Vert _{L^{2}%
}.\label{coef_deriv}%
\end{equation}
Thanks to the orthogonal decomposition
\[
u\left(  t\right)  =\mathbb{P}\left(  t\right)  u\left(  t\right)  +\sum
_{j=1}^{4N}\alpha_{j}\left(  t\right)  \nabla q_{j}\left(  t\right)
\]
where $\mathbb{P}\left(  t\right)$ is the Leray projector on $\Omega(t)$, we deduce from $\left(  \text{\ref{coef_deriv}}\right)  $  that
for all $i\in 1,4N$
\begin{equation} \label{coef_alphai}
\left\vert \int_{0}^{T}  \dot{s}\left(
t\right)  \alpha_{i}\left(  t\right)  \right\vert \leq C\left(X_0,R_0,E_0,T,\delta\right)   \left\Vert s\right\Vert _{L^{2}}.
\end{equation}
Now, we know that%
\begin{align*}
\int_{0}^{T}\dot{s}\left(  t\right)  \dot{r}_{i}\left(  t\right)  \mathrm{d}t
&  =\int_{0}^{T} \dot{s}\left(  t\right)\dashint_{\partial B_{i}\left(  t\right)  }u\left(  t\right)  \cdot
\mathfrak{n} \mathrm{d}\sigma\mathrm{d}t\\
&  =\int_{0}^{T} \dot{s}\left(  t\right) \sum_{j=1}^{4N}\dashint_{\partial B_{i}\left(  t\right)  }\alpha_{j}\left(
t\right)  \frac{\partial q_{j}\left(  t\right)  }{\partial \mathfrak{n}}\mathrm{d}%
\sigma\mathrm{d}t.
\end{align*}
Using the estimates $\left(  \text{\ref{lower_bound_R}}\right)  $, $\left(
\text{\ref{q1}}\right)$ and a similar argument for $x_{i}$ leads to the
estimate%
\begin{equation} \label{estimateRdotXdot}
\left\Vert \dot{R}\right\Vert _{H^{1}\left(  0,T\right)  }+\left\Vert \dot
{X}\right\Vert _{H^{1}\left(  0,T\right)  }\leq  C\left(X_0,R_0,E_0,T,\delta\right) .
\end{equation}

\subsection{Sequential continuity} \label{sec_seq_con}
We show in this section that any sequence $(X^n,R^n,u^n)_{n \in \N}$ of weak solutions has a subsequence converging to a weak solution $(X,R,u)$. 
Let again $T > 0$, $\delta > 0$, $(X_0,R_0,u_0)$ as in Theorem \ref{main_thm}. We assume that there exists $(X^n,R^n,u^n)_{n \in \N}$ a sequence of weak solutions of \eqref{main} on $[0,T)$ with data  $(X_0,R_0,u_0)$, such that for all $n$, over $(0,T)$ 
$$ \forall i \neq j, \quad |x_i^n - x_j^n| \ge  r^n_i + r_j^n + 4 \delta. $$
The point is to show that an accumulation point of this sequence is a weak solution. 
By estimates of Section \ref{sec_a_priori}, we get the uniform bounds 
\begin{equation} \label{uniform_bounds}
\begin{aligned}
& \| u^n \|^2_{L^\infty(0,T ; L^2(\Omega^n))} + \| \na u^n \|^2_{_{L^2(0,T ; L^2(\Omega^n))}} \le C E_0, \\
&  \|(X^n,R^n)\|_{H^2(0,T)} \le C\left(X_0,R_0,E_0,\delta,T\right).
\end{aligned}
\end{equation}
We can extend $u^n$ inside the balls setting 
$$\overline{u}^n = u^n \text{ in } \Omega^n, \quad  \overline{u}^n(t,x)  = \dot{x}_i^n(t) + \dot{r}_i^n(t) \frac{x-x_i^n(t)}{r_i^n(t)}, \quad x \in B_i^n(t).$$
The extended field $\overline{u}^n$ belongs to $L^\infty(0,T ; L^2(\R^3))$, but not to $L^2(0,T ; H^1(\R^3))$ due to discontinuity of the tangential components of the velocity across the spheres $\pa B_i^n$. By standard compactness arguments, we find that, after extraction of subsequence,
\begin{align*}
  & \overline{u}^n \rightarrow  \overline{u} \quad \text{ weakly* in } L^\infty(0,T ; L^2(\R^3)), \\
  & (X^n,R^n) \rightarrow  (X,R) \quad \text{ weakly* in } H^2(0,T) \text{ and strongly in } W^{1,\infty}([0,T]).    
\end{align*}
We can then define, for all $i \in \overline{1,N}$, 
$$B_i := B(x_i, r_i), \quad \Omega := \R^3 \setminus \cup_{i=1}^N \overline{B_i},  \quad u = \overline{u}\vert_{\Omega} . $$ 
By the uniform convergence of $(X^n,R^n)$ to $(X,R)$, one has for all $i$
$$ \sup_{t \in [0,T)} d_{Haus}(B_i^n(t), B_i(t)) \xrightarrow[n \rightarrow +\infty]{}  0  $$
where $d_{Haus}$ is the Hausdorff distance. Let $\eta \in (0, \delta)$. From the Hausdorff convergence, we get that for $n$ large enough, for all $i$, for all $t$,  
$$B_i^n(t) \subset B_i^\eta(t), \quad \text{ where } \:   B_i^\eta(t) := B(x_i(t), r_i(t)+\eta), \quad \Omega^\eta:= \R^3 \setminus \cup_{i=1}^N \overline{B_i^\eta} . $$ 
Domains $B_i^\eta$ and $\Omega^\eta$ will be used later on.
\begin{remark} \label{rem_D(u)}
Let $Q$ an arbitrary compact subset of the time-space domain $\cup_{t \in (0,T)} \{t\} \times \Omega(t)$. By the Hausdorff convergence above, we find that for $n$ large enough
$$ \|\mathbb{D}(u^n)\|_{L^2(Q)} \le \|\mathbb{D}(u^n)\|_{L^2(0,T ; \Omega^n)} \le C_0 $$
where $C_0$ only depends on $E_0$.
As $n \rightarrow +\infty$, we recover that $\mathbb{D}(u)$ belongs to $L^2(Q)$ with $\|\mathbb{D}(u)\|_{L^2(Q)} \le C_0$. As $Q$ is an arbitrary compact subset, we find that
$$ \mathbb{D}(u) \in L^2(0,T ; L^2(\Omega)) .$$
\end{remark}

We wish to show that $(X,R,u)$ is a weak solution on $(0,T)$ starting from $(X_0, R_0, u_0)$. The keypoint is to derive \eqref{weak_form_fluid_bubble} from the variational formulation 
\begin{equation}  \label{VF_n}
\begin{aligned}
\int_{\Omega^n\left(  t\right)  } u^n\left(  t\right)  \cdot\psi^n\left(
t\right)  +2\nu \int_{0}^{t}\int_{\Omega^n\left(  s\right)  }\mathbb{D}\left(
u^n\right)  \cdot\mathbb{D}\left(  \psi^n\right)  \mathrm{d}s - \sum_{i=1}^{N} \int_{0}^{t} \frac{c_i}{4\pi} (r^n_i)^{-3\gamma} 
\int_{\partial B^n_{i}\left(  s\right)  } \psi^n  \cdot \mathfrak{n}\mathrm{d}s\\
=\int_{\Omega_{0,F}  }  u_{0} \cdot\psi^n\left(0\right)
+\int_{0}^{t}\int_{\Omega^n\left(  s\right)  }  u^n \cdot (\partial_{t}\psi^n + u^n \cdot \na \psi^n)\mathrm{d}s,
\end{aligned}
\end{equation}
valid for any  $\psi^n\in L^{\infty}\left(0,T ; L_{\mathrm{dil}}^{2}(\Omega^n)\right)$ with $\partial_{t}\psi^n,\nabla\psi^n\in
L^2\left(0,T ; L^{2}( \Omega^n)\right)$. 

We fix  a test field 
$$ \psi \in L^\infty\left(0,T ; L_{\mathrm{dil}}^{2}(\Omega)\right) \cap   W^{1,\infty}\left(0,T ; H^s(\Omega) \right) \: \forall s \in \N. $$
It is enough to establish \eqref{weak_form_fluid_bubble} for this kind of $\psi$, more general test fields being handled by a density argument, see Corollary \ref{cor_test_functions_dil}. {\em Note however that we can not restrict to fields with compact support in $\overline{\Omega}$}. Otherwise, the divergence-free condition  would impose the additional constraint that  $\int_{\pa \Omega(t)} \psi \cdot \mathfrak{n} = 0$, a constraint that would forbid to reach any test field by density. We remind for instance that in the simple case $U = \R^3\setminus B(0,1)$, 
the field 
$$ \psi(x) = \frac{x}{|x|^3} $$
belongs to $L^2_{\mathrm{dil}}(U)$ but does not satisfy condition $\int_{\pa U} \psi \cdot \mathfrak{n} = 0$. 

One can not use directly $\psi$ as a test function in \eqref{VF_n}. We need to transform $\psi$ with the help of diffeomorphisms that map  $B_i(t)$ to $B_i^n(t)$ for all $i$. The existence of such diffeomorphisms is guaranteed by the following lemma, whose proof is in Appendix \ref{appendixA}. 

\begin{lemma} \label{lem_diffeo}
For any $(X_0,R_0), (\tilde X_0, \tilde R_0) \in \mathcal{A}$, see \eqref{def_calA}, there exists a smooth diffeomorphism 
$\Theta = \Theta[X_0,R_0,\tilde X_0, \tilde R_0] : \R^3 \rightarrow \R^3$ depending smoothly on its parameters $X_0,R_0,\tilde X_0, \tilde R_0$, with the following properties:
    \begin{itemize}
    \item If  $(X_0,R_0) = (\tilde X_0, \tilde R_0)$, \quad 
    $\Theta = \mathrm{Id}$.
    \item  $\Theta(x) = x$ for $|x|$ large enough. 
    \item $ \Theta(x) = x_i  + \frac{r_i}{\tilde r_i} (x-\tilde x_i)$ in the vicinity of $B(\tilde x_i,\tilde r_i)$, for all $1 \le i \le N$. In particular, it send $B(\tilde x_i,\tilde r_i)$ to $B(x_i,r_i)$ for all $i \in \overline{1,N}$. 
    \end{itemize}
\end{lemma}
We set $\Theta^n(t,x) :=  \Theta[X(t),R(t),X^n(t),R^n(t)](x)$, with $\Theta$ given by Lemma \ref{lem_diffeo}, and set
\begin{equation} \label{def_psi_n}
 \psi^n(t,x) := \det\left(D\Theta^n(t,x)\right) \left(D\Theta^n(t,x)\right)^{-1}\psi(t,\Theta^n(t,x)).    
\end{equation}
One can check that $\div \psi = 0 \Rightarrow \div \psi^n = 0$,   see Appendix \ref{appendixB}. Moreover, from the explicit formula for $\Theta^n$ in the vicinity of $B_i^n$, one has easily  $\psi(t) \in L^2_{\mathrm{dil}}(\Omega(t)) \Rightarrow \psi^n(t) \in L^2_{\mathrm{dil}}(\Omega^n(t))$ for all $t$. Finally,  the regularity properties of $\psi$ and the  Lipschitz regularity in time of $(X,R,X^n,R^n)$  allow to conclude that 
$$ \psi \in L^\infty\left(0,T ; L_{\mathrm{dil}}^{2}(\Omega^n)\right) \cap   W^{1,\infty}\left(0,T ; H^s(\Omega^n) \right) \: \forall s \in \N. $$
We can therefore take $\psi^n$ as a test function in the variational formulation \eqref{VF_n}. Proving directly the convergence of each term in this variational formulation is hard, as  integrals are taken on a domain that depends on  $n$. Our treatment will involve the auxiliary domain $\Omega^\eta$, notably the property that 
\begin{equation}
\psi^n \xrightarrow[n \rightarrow +\infty]{} \psi \quad \text{in }  W^{1,\infty}(0,T ; H^s(\Omega^\eta)) 
\end{equation}
which follows  from the strong convergence of $(X^n,R^n)$ to $(X,R)$ in $W^{1,\infty}(0,T)$. 

\paragraph{Linear terms.}
The convergence in the Hausdorff distance of $B_i^n$ to $B_i$ implies  
$$ 1_{\Omega^n} \rightarrow 1_{\Omega} \quad \text{in }L^\infty(0,T ; L^p(\R^3)), \quad \forall 1 \le p < \infty $$
It follows that
\begin{align*}
& \limsup_{n \rightarrow +\infty}\left| \int_{\Omega^n(t) \setminus \Omega^\eta(t)} u^n(t) \cdot \psi^n(t) \right| \\
 \le & \limsup_{n \rightarrow +\infty} \| u^n \|_{L^\infty(0,T ; L^2(\Omega^n))} \,  \|\psi^n\|_{L^\infty(0,T ; L^\infty(\Omega^n))} \, \|1_{\Omega^n \setminus \Omega^\eta}\|_{L^\infty(0,T ; L^2(\Omega^n))} \\
 \le & C  \|1_{\Omega \setminus \Omega^\eta}\|_{L^\infty(0,T ; L^2(\Omega^n))}  = C' \eta^{1/2} .
\end{align*}
Moreover, from the weak convergence of $\overline{u}^n$ and the strong convergence of $\psi^n\vert_{\Omega^\eta}$, we deduce easily 
$$\int_{\Omega^\eta\left(  t\right)  } u^n \left(  t\right)  \cdot\psi^n \left(
t\right) \xrightarrow[n \rightarrow +\infty]{} \int_{\Omega^\eta\left(  t\right)  } u \left(  t\right)  \cdot\psi\left(t\right)$$ 
weakly* in $L^\infty(0,T)$, while 
$$\left| \int_{\Omega(t) \setminus \Omega^\eta(t)} u \left(  t\right)  \cdot\psi\left(t\right) \right| \le C \eta^{1/2}. $$
We conclude that 
\begin{equation} \label{lim_u_psi}
\int_{\Omega^n\left(  t\right)  } u^n\left(  t\right)  \cdot\psi^n\left(
t\right) \rightarrow  \int_{\Omega\left(  t\right)  } u\left(  t\right)  \cdot\psi\left(
t\right)  
\end{equation} 
weakly * in $L^\infty(0,T)$.  Similar arguments yield  
\begin{equation}\label{lim_u_dtpsi}
 \int_{0}^{t}\int_{\Omega^n\left(  s\right)  }  u^n \cdot \partial_{t}\psi^n  \xrightarrow[n \rightarrow +\infty]{} \int_{0}^{t}\int_{\Omega\left(  s\right)  }  u \cdot \partial_{t}\psi
\end{equation} 
for all $t \in [0,T)$ and in  $L^1(0,T)$, 
the convergence in $L^1$ coming from the pointwise convergence and the dominated convergence theorem.  Also, from the uniform bound on $\| \mathbb{D}(u^n) \|_{L^2(0,T ; L^2(\Omega^n))}$ and Remark \ref{rem_D(u)}, we deduce similarly that 
\begin{equation}  \label{DuDpsi}
\int_0^t \int_{\Omega^n(s)} \mathbb{D}(u^n) : \mathbb{D}(\psi^n) \mathrm{d}s  \xrightarrow[n \rightarrow +\infty]{}  \int_0^t  \int_{\Omega(s)} \mathbb{D}(u) : \mathbb{D}(\psi) \mathrm{d}s       
\end{equation} 
for all $t \in [0,T)$ and in   $L^1(0,T)$.   Finally, using the explicit formula for $\psi^n$ in the vicinity of $B_i^n$, we get 
\begin{align*}
\int_{\partial B^n_{i}} \psi^n  \cdot \mathfrak{n} = \int_{\partial B_i} \psi \cdot \mathfrak{n}.
\end{align*}
 Combining this with the uniform convergence of $R^n$ to $R$, we conclude that 
\begin{equation}\label{lim_bord}
\lim_{n \rightarrow +\infty}
\sum_{i=1}^{N} \int_{0}^{t} c_i (r^n_i)^{2-3\gamma} 
\dashint_{\partial B^n_{i}\left(  s\right)  } \psi^n  \cdot \mathfrak{n}\mathrm{d}s = \sum_{i=1}^{N} \int_{0}^{t} c_i (r_i)^{2-3\gamma} 
\dashint_{\partial B_{i}\left(  s\right)  } \psi  \cdot \mathfrak{n}\mathrm{d}s    
\end{equation}
uniformly in time. 
\paragraph{Nonlinear terms.} As usual, the main difficulty is the limit in $n$ of the nonlinear term. We have 
\begin{equation} \label{nonlinear_residu1}
\begin{aligned} 
 &\limsup_{n \rightarrow +\infty}  \left| \int_{0}^{t}\int_{\Omega^n\left(  s\right)\setminus \Omega^\eta(s) }  u^n \cdot ( u^n \cdot \na \psi^n )\mathrm{d}s \right| \\
 & \le \limsup_{n \rightarrow +\infty}   \|\na \psi^n\|_{L^\infty(0,T ; L^\infty(\Omega^n))} \|u^n\|^2_{L^2(0,T ; L^4(\Omega^n))} \|1_{\Omega^n\setminus \Omega^\eta} \|_{L^2(0,T ; L^2(\R^3))} \\
 & \le C \limsup_{n \rightarrow +\infty}   \|\na \psi^n\|_{L^\infty(0,T ; L^\infty(\Omega^n))} \|u^n\|^2_{L^2(0,T ; H^1(\Omega^n))} \|1_{\Omega^n\setminus \Omega^\eta} \|_{L^2(0,T ; L^2(\R^3))} \\
 & \le C'  \limsup_{n \rightarrow +\infty} \|1_{\Omega^n\setminus \Omega^\eta} \|_{L^2(0,T ; L^2(\R^3))} =  C' \|1_{\Omega\setminus \Omega^\eta} \|_{L^2(0,T ; L^2(\R^3))} \le C'' \eta^{1/2}.
\end{aligned}
\end{equation}
Here, $C''$ is a constant that depends on $\psi$, $X_0,R_0, E_0$,$\delta$,$T$.
Similarly, taking into account Remark \ref{rem_D(u)},
\begin{equation}  \label{nonlinear_residu2}
\begin{aligned}
 &\left| \int_{0}^{t}\int_{\Omega\left(  s\right)\setminus \Omega^\eta(s) }  u \cdot ( u \cdot \na \psi )\mathrm{d}s \right|  \le C \eta^{1/2}. 
 \end{aligned}   
\end{equation}
It remains to control 
$$ J^{n,\eta}(t) :=  \int_{0}^{t}\int_{\Omega^\eta\left(  s\right)  }  u^n \cdot ( u^n \cdot \na \psi^n )\mathrm{d}s .$$
We have the Liouville decomposition  
$$ u^n(t) = \tilde u^{n,\eta}(t)  + w^{n,\eta}(t) \quad \text{ in } \Omega^\eta(t)$$
where $\tilde u^{n,\eta}(t) := \mathbb{P}^\eta(t) \left(u^n(t)\vert_{\Omega^\eta(t)}\right)$ is the orthogonal projection of $u^n(t)\vert_{\Omega^\eta(t)}$ over $L^2_\sigma(\Omega^\eta(t))$ and $w^{n,\eta} = \nabla q^{n,\eta}(t)$ is a gradient field: as $u^n$ is divergence-free, it satisfies 
\begin{equation}
\left\{
\begin{aligned}
\Delta q^{n,\eta} & = 0 \text{ in } \Omega^\eta, \\
\pa_{\mathfrak{n}} q^{n,\eta} & = - u^n \cdot \mathfrak{n} \text{ at } \pa \Omega^\eta.  
 \end{aligned}
 \right.
\end{equation}
{\em Study of $\tilde u^{n,\eta}$}. 
The terms involving $\tilde u^{n,\eta}$ will be handled thanks to 
\begin{lemma} \label{lemma_v}
Denote $\tilde u^\eta(t) := \mathbb{P}^\eta(t)(u\vert_{\Omega^\eta(t)})$. There exists a subsequence of $(\tilde u^{n,\eta})_{n \in \N}$, still denoted $(\tilde u^{n,\eta})_{n \in \N}$ for simplicity,  such that for any $M > 0$, 
$$ \int_0^T \int_{\Omega^\eta(t) \cap B(0,M)} |\tilde u^{n,\eta} - \tilde u^\eta|^2 \mathrm{d}t \xrightarrow[n \rightarrow +\infty]{} 0 .$$
\end{lemma}
\begin{proof}
The proof relies on a nice generalization of the Aubin-Lions lemma for non-cylindrical domains  given in \cite[Theorem 2]{Moussa}. Strictly speaking, the theorem in \cite{Moussa} is stated for space-time domains of the form 
\begin{equation} \label{def_hatOmega}
\hat{\Omega} = \cup_{t \in (0,T)} \{t\} \times \Omega^t
\end{equation}
with $\Omega^t$ bounded, but a look at the proof shows that the boundedness of $\Omega^t$ can be relaxed if the final compactness statement is replaced by local compactness. Namely, we shall use 
\begin{theorem} {\bf (slight modification  of  \cite[Theorem 2]{Moussa})} \label{thm_Moussa}

Let $\hat{\Omega}$ a domain of the form \eqref{def_hatOmega}, where for all $t \in (0,T)$, $\Omega^t = \mathcal{A}_t(\Omega)$ for a fixed Lipschitz domain $\Omega$  and a family $(\mathcal{A}_t)_{t \in [0,T]}$  of $C^1$-diffeomorphisms, with
$$ (t,x) \rightarrow \mathcal{A}_t(x) \in C([0,T]; C^1(\R^d)).$$
Let $(\tilde u^n)_{n \in \N}$ a sequence of divergence-free fields in $L^2(\hat{\Omega})$ s.t. $1_{\hat{\Omega}} \tilde u^n \in L^\infty(\R, L^2(\R^3))$ for all $n$. We assume 
\begin{itemize}
 \item[i)] $\sup_{n \in \N}  \left( \|\tilde u^n\|_{L^2(\hat{\Omega})} + \|\na \tilde u^n \|_{L^2(\hat{\Omega})} \right) \le C . $ 
 \item[ii)] $\tilde u^n(t) \cdot \mathfrak{n} = 0$ at $\pa \Omega^t$ for all $t \in (0,T)$.
  \item[iii)] There exists $C > 0$ and $N_0 > 0$ such that for all $\Psi$ smooth, divergence free and compactly supported in $\hat{\Omega}$: 
  $$ \left| \langle \pa_t \tilde u^n , \Psi \rangle \right| \le C \sum_{|\alpha| \le N_0} \|\pa^\alpha \Psi\|_{L^2(\hat{\Omega})}.$$
\end{itemize}
Then, $(\tilde u^n)_{n \in\N}$ is compact in $L^2\left(\hat{\Omega}  \cap ((0,T) \times B)\right)$ for any ball $B \subset \R^3$. 
\end{theorem}
We apply this theorem with $\Omega^t = \Omega^\eta(t)$, and $\tilde u^n = \tilde u^{n,\eta}$. Assumption i) comes from the  bound 
$$ \sup_{n \in \N}  \left( \|u^n\|_{L^2(0,T ; L^2(\Omega^\eta)} + \|\na u^n\|_{L^2(0,T ; L^2(\Omega^\eta)} \right)\le C $$
 inherited from \eqref{uniform_bounds} and from the continuity of the Leray projector $\mathbb{P}^\eta(t)$ over $L^2(\Omega^\eta(t))$ and $H^1(\Omega^\eta(t))$. Assumption ii) is obvious by definition of the Leray projector. The only point that deserves attention is iii). Let $\Psi$ as above. We can take it as a test function in \eqref{VF_n}. We insist that $\Psi = 0$ at $t=0$ and $t=T$, and that $\Psi(t)$ is compactly supported in $\Omega^\eta(t)$. We find: 
\begin{align*}
\langle \pa_t \tilde u^{n,\eta} , \Psi \rangle & = - \int_0^T \int_{\Omega^\eta} \tilde u^{n,\eta}  \cdot \pa_t \Psi  =  - \int_0^T \int_{\Omega^\eta} u^n   \cdot \pa_t \Psi  \\  
& = \int_{0}^{T}\int_{\Omega^\eta\left(  s\right)  }  u^n \cdot (u^n \cdot \na \Psi)\mathrm{d}s    - 2\nu \int_{0}^{t}\int_{\Omega^\eta\left(  s\right)  }\mathbb{D}\left(
u^n\right)  \cdot\mathbb{D}\left(  \Psi\right)  \mathrm{d}s. 
\end{align*}
From the uniform bounds in \eqref{uniform_bounds} and standard manipulations, we end up with 
\begin{align*}
 |\langle \pa_t \tilde u^{n,\eta} , \Psi \rangle| & \le  C \left(  \|\na \Psi\|_{L^2(0,T ; L^\infty(\Omega^\eta))} + \|\Psi\|_{L^2(0,T ; H^1(\Omega^\eta))} \right) \\
 & \le C' \|\Psi\|_{L^2(0,T ; H^3(\Omega^\eta))}
\end{align*}
which proves that iii) is satisfied with $N_0=3$. Therefore, Theorem \ref{thm_Moussa} provides a $\tilde{v}^\eta$ such that, after extraction of a subsequence: for any $M > 0$  
$$ \int_0^T \int_{\Omega^\eta(t) \cap B(0,M)} |\tilde u^{n,\eta} - \tilde{v}^\eta|^2 \mathrm{d}t \xrightarrow[n \rightarrow +\infty]{} 0 .$$
On the other hand, from the weak convergence of $u^n$ to $u$ in $L^2(0,T ; L^2(\Omega^\eta))$ and the weak continuity of the Leray projector $\mathbb{P}^\eta(t)$, we find that $\tilde u^{n,\eta}$ converges weakly to $\tilde u^\eta$, and so $\tilde{v}^\eta = \tilde u^\eta$ which concludes the proof of the lemma.
\end{proof}
{\em Study of $w^{n,\eta}$}. We first introduce two analogues of $w^{n,\eta}$. We define for all $t \in [0,T)$
$$w^n(t)  = \na q^n(t) := (I - \mathbb{P}^n(t)) u^n(t), \quad w(t) = \na q(t) := (I - \mathbb{P}(t))u (t) . 
$$
These are the gradient parts in the Liouville decomposition of $u^n(t)$ in $L^2(\Omega^n(t))$ and  $L^2(\Omega(t))$ respectively.  We shall prove
\begin{lemma} \label{lemma_w}
The following properties hold, for some $C$ depending on $X_0,R_0,E_0,T,\delta$ : 
 \begin{itemize}
     \item[i)] $\limsup_{n \rightarrow +\infty} \|w^{n,\eta} - w^{n} \|_{L^2(0,T ; L^2(\Omega^\eta)) } \le C \eta^{\frac13}$.
     \item[ii)] $\|w^\eta - w \|_{L^2(0,T ; L^2(\Omega^\eta))} \le C \eta^{\frac13}$. 
     \item[iii)] For all $M > 0$, 
     $ \int_0^T \int_{\Omega^\eta(t) \cap B(0,M)} |w^n - w|^2 \mathrm{d}t \xrightarrow[n \rightarrow +\infty]{} 0 $.
 \end{itemize}   
\end{lemma}
\begin{proof}
For i), we write for any scalar  $\varphi \in L^2(0,T ; L^2_{loc}(\R^3))$ s.t. $\na \varphi \in L^2(0,T ; L^2(\R^3))$ :    
\begin{align*}
\int_0^T \int_{\Omega^\eta} (w^{n,\eta} - w^{n}) \cdot \na \varphi & = \int_0^T \int_{\Omega^\eta} \na q^{n,\eta} \cdot \na \varphi - \int_0^T \int_{\Omega^\eta} \na q^n \cdot \na \varphi \\
& = \int_0^T \int_{\Omega^\eta} \na q^{n,\eta} \cdot \na \varphi - \int_0^T \int_{\Omega^n} \na q^n \cdot \na \varphi + \int_0^T \int_{\Omega^n\setminus \Omega^\eta } \na q^n \cdot \na \varphi \\
& = -\int_0^T \int_{\pa \Omega^\eta} u^n \cdot \mathfrak{n} \, \varphi + \int_0^T \int_{\pa \Omega^n} u^n \cdot \mathfrak{n} \, \varphi +   \int_0^T \int_{\Omega^n\setminus \Omega^\eta } \na q^n \cdot \na \varphi \\
& = \int_0^T \int_{\Omega^n\setminus \Omega^\eta} u^n \cdot \na \varphi  +   \int_0^T \int_{\Omega^n\setminus \Omega^\eta } \na q^n \cdot \na \varphi.
\end{align*}
We have 
\begin{align*}
    \| u^n \|_{L^2(0,T ; L^2(\Omega^n\setminus \Omega^\eta))} & \le   \| u^n \|_{L^2(0,T ; L^6(\Omega^n)} | \Omega^n \setminus \Omega^\eta|^{1/3} \\
    & \le C \| u^n \|_{L^2(0,T ; H^1(\Omega^n)} | \Omega^n \setminus \Omega^\eta|^{1/3}  \le C' \eta^{1/3}. 
\end{align*}
and similarly  
$$\| \na q^n \|_{L^2(0,T ; L^2(\Omega^n\setminus \Omega^\eta))} \le C \eta^{1/3} $$
We end up with 
$$ \int_0^T \int_{\Omega^\eta} (w^{n,\eta} - w^{n}) \cdot \na \varphi  \le C \eta^{1/3} \|\na \varphi \|_{L^2(0,T ; L^2(\Omega^\eta))} .$$
The bound i) follows. The bound ii) can be proved exactly in the same way.  As regards iii), we remind that $q^n(t)$ solves 
\begin{equation}
    \left\{
\begin{aligned}
    \Delta q^n(t)  & = 0 \text{ in } \Omega^n(t) \\
   \pa_{\mathfrak{n}} q^n(t) & =  \dot{r}_i^n + \dot{x}_i^n \cdot \mathfrak{n}   \text{ at }  \pa B^n_i(t), \: 1 \le i \le N  
\end{aligned}
    \right.
\end{equation}
so that it can be decomposed over the harmonic fields $(\na q^n_i(t))_{i \in \overline{1,N}}$, $(\na q_i^{k,n}(t))_{i \in \overline{1,N}, k \in \overline{1,3},}$ analogue to the fields  $(\na q_i(t))_{i \in \overline{1,N}}$, $(\na q^{k}_i(t))_{i \in \overline{1,N}, k \in \overline{1,3},}$ mentioned in the proof of the $H^2(0,T)$ {\it a priori} estimates in Section \ref{sec_a_priori}. These harmonic fields satisfy the same bounds as in \eqref{q1}-\eqref{q1k}, which yields compactness in $L^2(0,T ; L^2(\Omega^\eta \cap B(0,M))$, see for instance \cite[Theorem 3]{Moussa}. As
$$ \na q_n = \sum_{i=1}^N \dot{r}_i^n \na q_i^n +  \sum_{i=1}^N \sum_{k=1}^3 \dot{x}_i^{k,n}  \na q_i^{k,n} $$
and as $((X^n,R^n))_{n \in \N}$ is weakly compact in $H^2(0,T)$, hence strongly in $C^1([0,T])$, one recovers that  $(\na q^n)_{n \in \N}$ is compact in $L^2(0,T ; L^2(\Omega^\eta \cap B(0,M))$. Moreover, by classical results on domain continuity for the Neumann problem, see \cite[chapter 3, p. 118]{Henrot},  $\na q_i^n$ and $\na q_i^{k,n}$ converge weakly to $\na q_i$ and $\na q_i^k$ in $L^2(0,T ; L^2(\Omega^\eta))$. Together with the uniform  convergence of $(X^n,R^n)$, this shows that the accumulation point of $w^n = \na q^n$ is $w = \na q$. This concludes the proof of iii) and therefore of the lemma. 
\end{proof}

{\em Conclusion}.
 We come back to the analysis of $J^{n,\eta}(t)$. We remind that the field $\psi^n$ is defined in \eqref{def_psi_n}, and associated to a fixed test field $\psi \in W^{1,\infty}(0,T ; H^\infty(\Omega))$.  
 
 First, for large enough $M > 0$, we find that 
\begin{align*}
& \left|  \int_{0}^t \int_{\Omega^\eta(t) \cap B(0,M)^c} u^n \otimes u^n : \na \psi^n \right| = \left|  \int_{0}^t \int_{ B(0,M)^c} u^n \otimes u^n : \na \psi \right| \\
& \le C' \|u^n\|_{L^2(0,T ; L^2(\Omega^n))}^2 \|\na \psi\|_{L^\infty(0,T ; L^\infty(B(0,M)^c))} \\
& \le C'' \|\na \psi\|_{L^\infty(0,T ; L^\infty(B(0,M)^c)}.
\end{align*}
Note that the right-hand side goes to zero when $M \rightarrow +\infty$. 
Similarly, 
\begin{align*}\left|  \int_{0}^t \int_{\Omega^\eta(t) \cap B(0,M)^c} u \otimes u : \na \psi        \right|  & \le C \|\na \psi\|_{L^\infty(0,T ; L^\infty(\Omega \cap B(0,M)^c)}
\end{align*}
with the same right-hand side going to zero as $M \rightarrow +\infty$.  Moreover, combining Lemmas \ref{lemma_v} and \ref{lemma_w}, we have that 
$$ \limsup_{n \rightarrow +\infty} \int_0^t \int_{\Omega^\eta(t) \cap B(0,M)} |u^n - u|^2 \le C \eta^{2/3}.$$
Thanks to this bound and to the strong convergence of $\na \psi^n\vert_{\Omega^\eta}$,  we find easily that
\begin{align*}
& \limsup_{n \rightarrow +\infty} \left| \int_{0}^t \int_{\Omega^\eta(t) \cap B(0,M)} u^n \otimes u^n : \na \psi^n - \int_0^t \int_{\Omega^\eta \cap B(0,M)} u \cdot (u \cdot \na \psi)\right| \le  C \eta^{2/3} 
\end{align*}
uniformly on $t \in (0,T)$.
We deduce from these  bounds  that 
\begin{equation} \label{limJn}
\begin{aligned}
      \limsup_{n \rightarrow +\infty}\left| J^{n,\eta}(t) - \int_0^t \int_{\Omega^\eta} u \cdot (u \cdot \na \psi)\right| \le C \eta^{\frac13} 
   \end{aligned}
\end{equation}
uniformly on $(0,T)$,  where $C$ depends on $\psi$, $R_0,X_0,E_0,T, \delta$. 

\paragraph{Proof of sequential continuity.}
We now have all ingredients to conclude to sequential continuity. We start from \eqref{VF_n}, and take into account \eqref{lim_u_psi}-\eqref{lim_u_dtpsi}-\eqref{DuDpsi}-\eqref{lim_bord}-\eqref{nonlinear_residu1}-\eqref{nonlinear_residu2}-\eqref{limJn}. We   recover \eqref{weak_form_fluid_bubble}, which shows sequential stability. 

\section{Auxiliary system with prescribed bubbles dynamics} \label{sec_prescribed}

The goal of this section is to prove Theorem \ref{thm_prescribed}. We use the same notations as in the first paragraph of Section \ref{sec_strategy}: we give ourselves $T > 0$ and affine fields 
\begin{equation}
(X,R)=\left(  x_{i},r_{i}\right)  _{i\in\overline{1,N}}:\mathbb{R}%
_{+}\rightarrow\mathcal{A} \label{XR_approx}%
\end{equation}
that satisfy \eqref{separation_condition}. 
The reader should keep in mind that $\delta=\delta\left(T,X,R\right)  $ but
in order to keep the notations simple we do not mark explicitly this
dependence. Then for all $i\in\overline{1,N}$ we have that%
\begin{equation}
x\in B\left(  x_{i}\left(  t\right)  ,r_{i}\left(  t\right)  +2\delta\right)
\Rightarrow x\in\mathbb{R}^{3}\backslash\cup_{j\not =i}B_{j}\left(
x_{j}\left(  t\right)  ,r_{j}\left(  t\right)  +2\delta\right)  .
\end{equation}
We want to solve the variational formulation \eqref{VF_prescribed}.  
We will construct a  solution 
$$ u \in L^{\infty}\left(0,T  ;L_{\mathrm{dil}}^{2}\left(  \Omega\left[  X,R\right]  \right)\right), \quad \na u \in L^2\left(0,T  ;L^{2}\left(  \Omega\left[X,R\right]  \right)\right)  $$
by a Galerkin method. In particular, we will need for all $t$ finite dimensional approximation spaces of  $L_{\mathrm{dil}}^{2}\left(  \Omega\left[  X,R\right](t) \right)$. The idea is to rely on the decomposition introduced in Section \ref{sec_functional_spaces}, that is 
$$ L_{\mathrm{dil}}^{2}\left(  \Omega\left[  X,R\right](t)  \right)  = L^2_\sigma\left(  \Omega\left[X,R\right](t)  \right) \oplus \mathbb{G}(\Omega\left[  X,R\right]
\left(  t\right)  ) .$$
We remind that the latter space $\mathbb{G}(\Omega\left[  X,R\right] \left(  t\right))$ is already finite dimensional, with an orthonormal basis $(\na q_i(t))_{i \in \overline{1,4N}}$, see the discussion in Section \ref{sec_functional_spaces}. As regards $L^2_\sigma\left(  \Omega\left[  X,R\right](t)  \right)$, we shall first consider a sequence of finite-dimensional approximations  of  
$$L^2_\sigma\left(  \Omega\left[  X,R\right](t) \cap B(0,m) \right), \: \text{ with $m$ such that }  \cup_{t \in [0,T]} \cup_{i=1}^N B_i[X,R](t)  \Subset B(0,m) $$
and then send $m$ to infinity. More precisely, we will construct Galerkin approximations $u^{m,n}(t) : \Omega[X,R](t) \cap B(0,m) \rightarrow \R^3$ that  will belong to 
\[ 
\operatorname*{Span}\left\{  e_{1}^{m}\left(  t\right)  ,e_{2}%
^{m}\left(  t\right)  ,\cdots,e_{n}^{m}\left(  t\right)  \right\}
\oplus \mathbb{G}(\Omega\left[  X,R\right] \left(  t\right)) 
\]
with $\left(  e_{\ell}^{m}\left(  t\right)  \right)  _{\ell\ge 1}$ a
Riesz basis of $L_{\sigma}^{2}\left(  \Omega\left[  X,R\right]  \left(
t\right)  \cap B\left(  0,m\right)  \right)  $. This basis will be obtained by the push-forward
of eigenfunctions of the Stokes operator defined on the initial domain $\Omega\left[  X,R\right]  \left(
0\right)  \cap B\left(  0,m\right)  $ (with appropriate boundary conditions). Each element $e_{\ell}^{m}\left(  t\right)$ will be differentiable in $t$, a key property that is not necessarily satisfied by other choices of basis, see Remark \ref{rem_Riesz_basis}.
The construction and properties of such a basis and the functional toolbox needed are presented in Section
\ref{GalerkinSpace} with proofs of the more technical results postponed to the
Appendix. In Section \ref{Galerkin_solve} we solve a system of ODEs which
allows us to obtain an approximate solution $u^{m,n}$
verifying uniform estimates with respect to $m$ and $n$. One main difficulty that we face is the derivation of estimates for time derivatives of $u^{m,n}$. First, as the basis $(e_{m,l}(t)$ is not orthonormal at $t > 0$, it is not obvious to obtain  a good control of the Gram matrix appearing in the variational formulation. Second,  as the basis is time dependent, the time derivative of $u^{m,n}$ does  not only involve the time derivative of the coefficients of $u^{m,n}$ in the basis, which are the natural quantities  in the variational formulation. Finally, we show that solutions of \eqref{VF_prescribed} can be obtained in the limit when $n,m\rightarrow\infty$ of $u^{m,n}$. This is the purpose of Section
\ref{limit_passage}.


\subsection{Construction of finite-dimensional approximation spaces\label{GalerkinSpace}}

The first thing we need to do is to introduce a family of diffeomorphisms $\Theta(t, \cdot)$ that maps the initial configuration $\Omega\left[  X,R\right]
\left(  0\right)$ to   $\Omega\left[  X,R\right]  \left(  t\right)$. We take inspiration from the so-called Arbitrary
Lagrangian-Eulerian (ALE) method.

\begin{lemma} 
\label{diffeoLemma}For $\left(  X,R\right)  $ and $\delta$ as in Theorem \ref{thm_prescribed},  there
exists $v^{ALE} \in C^\infty_c\left([0,T] \times \R^3 ; \R^3\right)$ and a family of diffeomorphisms $\Theta(t, \cdot)$  defined through
\[
\dot{\Theta}\left(  t,x\right)  =v^{ALE}\left(  t,\Theta\left(  t,x\right)
\right)  ,\text{ }\Theta\left(  0,x\right)  =x,
\]
such that the following holds true.

\begin{enumerate}
\item For all $t\in\left[  0,T\right]  $ and any $i\in\overline{1,N}$ we have
that%
\[
v^{ALE}\left(  t,x\right)  =\dot{x}_{i}\left(  t\right)  +\frac{\dot{r}%
_{i}\left(  t\right)  }{r_{i}\left(  t\right)  }\left(  x-x_{i}\left(
t\right)  \right)  \text{ on }B\left(  x_{i}\left(  t\right)  ,r_{i}\left(
t\right)  +\frac{\delta}{4}\right)  .
\]

\item For all $t\in\left[  0,T\right]  $ and any $i\in\overline{1,N}$ we have
that%
\[
\Theta\left(  t,x\right)  =x_{i}\left(  t\right)  +r_{i}\left(  t\right)
\frac{x-x_{i}\left(  0\right)  }{r_{i}\left(  0\right)  }\text{ on }B\left(
x_{i}\left(  0\right)  ,r_{i}\left(  0\right)  +\frac{\delta}{8}\right)  .
\]
In particular, $\Theta(t, \cdot)$ maps $\Omega\left[  X,R\right]  \left(
0\right)$ to $\Omega\left[  X,R\right]  \left(
t\right)$  
Moreover, there exists $m_0=m_0\left(X,R\right)$ such that $\cup_{i=1}^N B_i[X,R](t) \Subset B(0,m_0)$ for all $t \in [0,T]$, and such that 
$\Theta\left(  t,x\right)  = x$ for $\left\vert x\right\vert \geq
m_0$,  so that  $\Theta\left(  t,\cdot\right)  $ maps
$\Omega\left[  X,R\right]  \left(  0\right)  \cap B\left(  0,m\right)  $
onto $\Omega\left[  X,R\right]  \left(  t\right)  \cap B\left(
0,m\right)  $ for all $m \ge m_0$.

\end{enumerate}
\end{lemma}
The proof of Lemma \ref{diffeoLemma} is carried out in  Appendix \ref{appendixA}.

\begin{remark}
In the rest of Section \ref{sec_prescribed}, for  better readability,  we now omit the $\left[X,R\right]$ dependence in the notation of the various domains: we write $B_i(t)$ for $B_i[X,R](t)$, $\Omega(t)$ for $\Omega[X,R](t)$. Also, in all the inequalities below,  the constants will implicitly depend (in an increasing manner) on $\|(X,R)\|_{W^{1,\infty}}$, $1/\delta$, $T$ and $\|1/R\|_{L^\infty}$, where $\frac{1}{R} := \left(  \frac{1}{r_{1}},\frac{1}{r_{2}},\cdots,\frac{1}{r_{N}}\right)$. 
 For all $m \ge m_0$ we set
\[
\Omega^{m}\left(  t\right)  :=\Omega\left(  t\right)
\cap B\left(  0,m\right)  .
\]
Some of the constants in the inequalities will also depend on $m$, in such a case they will be denoted $C(m)$.
    
\end{remark}
Next, we consider a Hilbert basis $\left(  e_{\ell}^{m}(0)\right)_{\ell\ge 1}$ of $L_{\sigma}^{2}\left(  \Omega^{m}\left(  0\right)  \right)  $
formed by eigenvectors of the non-homogeneous Stokes operator
\begin{equation}
\left\{
\begin{aligned}
e_{\ell}^{m}\left(  0\right)  -\Delta e_{\ell}^{m}\left(  0\right)  +\nabla
\pi_{\ell}^{m}\left(  0\right)  =(1+(\lambda_{\ell}^{m})^{2})e_{\ell}%
^{m}\left(  0\right), & \quad  \text{ in }\Omega^{m}\left(  0\right)  ,\\
\operatorname{div}e_{\ell}^{m}\left(  0\right)  =0, & \quad \text{ in }\Omega
^{m}\left(  0\right)  ,\\
e_{\ell}^{m}\left(  0\right)  \cdot\mathfrak{n}=0, & \quad  \text{ on }\partial \Omega^{m}\left(  0\right) ,\\
T\left(  e_{\ell}^{m}\left(  0\right)  ,\pi_{\ell}^{m}\left(  0\right)
\right)  \mathfrak{n}\times\mathfrak{n}=0, & \quad \text{ on }\partial \Omega^{m}\left(  0\right) .
\end{aligned}
\right.  \label{Stokes_op}%
\end{equation}
where $\lambda_{\ell}^{m}\geq0$ and $\lambda_{\ell}^{m}\xrightarrow[\ell \rightarrow +\infty]{} +\infty$. We
consider the following closed subspace of the Sobolev space $H^{1}\left(
\Omega^{m}(0)\right)$:
\[
H_{\sigma}^{1}(\Omega^{m}(0)) := L_{\sigma}^{2}\left(  \Omega^{m}\left(
0\right)  \right)  \cap H^{1}(\Omega^{m}\left(  0\right)  ).
\]
Owing to Korn's inequality,
\[
\left\langle \psi,\varphi\right\rangle _{H_{\sigma}^{1}(\Omega^{m}%
(0))}=\left\langle \psi,\varphi\right\rangle _{L^{2}(\Omega^{m}(0))}%
+\left\langle \mathbb{D}\psi,\mathbb{D}\varphi\right\rangle _{L^{2}(\Omega
^{m}(0))}%
\]
is a scalar product which induces an equivalent norm on $H_{\sigma}^{1}%
(\Omega^{m}(0))$ given by:
\[
\left\Vert \psi\right\Vert _{H_{\sigma}^{1}(\Omega^{m}(0))}^{2}:=\left\Vert
\psi\right\Vert _{L^{2}(\Omega^{m}(0))}^{2}+\left\Vert \mathbb{D}%
\psi\right\Vert _{L^{2}(\Omega^{m}(0))}^{2}%
\]
It is then easily seen that $\left(e^m_\ell/\sqrt{1+(\lambda^m_\ell)^2}\right)_{\ell \ge 1}$ is a Hilbert basis of $H_{\sigma}^{1}%
(\Omega^{m}(0))$, so that
\begin{equation}
\left\Vert \psi\right\Vert _{H_{\sigma}^{1}(\Omega^{m}(0))}^{2}=\sum
_{l=1}^{\infty}(1+(\lambda_{\ell}^m)^{2})\left\langle \psi,e_{\ell}^{m}(0)\right\rangle
_{L^{2}(\Omega^{m}(0))}^{2}.\label{norm_H1_spect}%
\end{equation}

Let $\Theta$ as in Lemma \ref{diffeoLemma}. We denote $\Theta_t := \Theta(t, \cdot)$. We consider the operators
\begin{equation}
\left\{
\begin{array}
[c]{l}%
\mathcal{S}\left(  t\right): L_{\sigma}^{2}\left(  \Omega\left(  0\right)
\right)  \rightarrow L_{\sigma}^{2}\left(  \Omega\left(  t\right)  \right)  \\
v\in L_{\sigma}^{2}\left(  \Omega\left(  0\right)  \right)  \rightarrow
(\mathcal{S}\left(  t\right)  v)\left(  x\right)  =\dfrac{D\Theta_t\left(\Theta_t^{-1}(x)\right)}{\det D\Theta_t\left(\Theta_t^{-1}\left(x\right)  \right)  }v(\Theta_t^{-1}(x))\\
\text{
\ \ \ \ \ \ \ \ \ \ \ \ \ \ \ \ \ \ \ \ \ \ \ \ \ \ \ \ \ \ \ \ \ \ \ \ \ \ }%
=\left[  \operatorname{Cof}D\Theta_t^{-1}(x)\right]  ^{T}v(\Theta
^{-1}_t(x)) . 
\end{array}
\right.  \label{S(t)1}%
\end{equation}
as well as 
\begin{equation}
\left\{
\begin{array}
[c]{l}%
\mathcal{S}\left(  t\right)  ^{-1}:L_{\sigma}^{2}\left(  \Omega\left(
t\right)  \right)  \rightarrow L_{\sigma}^{2}\left(  \Omega\left(  0\right)
\right)  \\
v\in L^{2}\left(  \Omega\left(  t\right)  \right)  \rightarrow(\mathcal{S}%
\left(  t\right)  ^{-1}v)\left(  x\right)  =\dfrac{D(\Theta_t^{-1})\left(\Theta_t(x)\right)  }{\det D(\Theta_t^{-1})\left(\Theta_t(x)\right)
}v(\Theta_t(x))\\
\text{
\ \ \ \ \ \ \ \ \ \ \ \ \ \ \ \ \ \ \ \ \ \ \ \ \ \ \ \ \ \ \ \ \ \ \ \ \ \ \ \ \ }%
=\left[  \operatorname{Cof}D\Theta_t(x)\right]  ^{T}v(\Theta_t(x)) .
\end{array}
\right.  \label{S-1(t)1}%
\end{equation}
We have the important 
\begin{proposition} \label{propSt}
For all $t \in [0,T]$, $\mathcal{S}(t)$ and $\mathcal{S}(t)^{-1}$ are well-defined bounded linear maps, inverse from one another, with 
$$ \sup_{t \in [0,T]} \|S(t)\|_{L^2(\Omega(0)) \rightarrow L^2(\Omega(t))} \le C, \quad  \sup_{t \in [0,T]} \|S(t)^{-1}\|_{L^2(\Omega(t)) \rightarrow L^2(\Omega(0))} \le C .$$
Moreover, for all $m \ge m_0$, {\em cf.} Lemma \ref{diffeoLemma}, $\mathcal{S}(t)$, resp. $\mathcal{S}(t)^{-1}$, maps $L_{\sigma}^{2}\left(
\Omega^{m}\left(  0\right)  \right)  $ into $L_{\sigma}^{2}\left(  \Omega
^{m}\left(  t\right)  \right)  $, resp. $L_{\sigma}^{2}\left(
\Omega^{m}\left(  t\right)  \right)  $ into $L_{\sigma}^{2}\left(  \Omega
^{m}\left(  0\right)  \right)  $ with norms independent on $m \ge m_0$. 
\end{proposition}
The proof of Proposition \ref{propSt} will be provided in Appendix \ref{appendixB}. 

Given the regularity of $\Theta$, $\mathcal{S}\left(  t\right)  $ behaves well
with respect to space derivatives, namely

\begin{proposition}
\label{equiv_Sobolev}Consider $m\ge m_0$ and $v\in H_{\sigma}^{1}\left(
\Omega^{m}\left(  0\right)  \right)  $. Then $\mathcal{S}\left(  t\right)
v\in H_{\sigma}^{1}\left(  \Omega^{m}\left(  t\right)  \right)  $ and there
exists $C(m) > 0$ independent of $v$ such that
\begin{equation}
\frac{1}{C(m)}\left\Vert v\right\Vert _{H_{\sigma}^{1}\left(
\Omega^{m}\left(  0\right)  \right)  }^{2}\leq\left\Vert \mathcal{S}\left(
t\right)  v\right\Vert _{H_{\sigma}^{1}\left(  \Omega^{m}\left(  t\right)
\right)  }^{2}\leq C(m)  \left\Vert v\right\Vert _{H_{\sigma
}^{1}\left(  \Omega^{m}\left(  0\right)  \right)  }^{2}. \label{norm_H1_equiv}%
\end{equation}

\end{proposition}
The proof is given in Appendix \ref{appendixB}.

For all $\ell\in\mathbb{N}$, we consider
\begin{equation}
e_{\ell}^{m}\left(  t\right)  =\mathcal{S}\left(  t\right)  e_{\ell}%
^{m}\left(  0\right)  \in L_{\sigma}^{2}\left(  \Omega^{m}\left(  t\right)
\right)  ,\label{vectori_em(t)}%
\end{equation}
where the operator $\mathcal{S}\left(  t\right)  $ was introduced in $\left(
\text{\ref{S(t)1}}\right)  $. Although the family of vectors $\left(  e_{\ell}^{m}\left(  t\right)  \right)
_{\ell\geq0}$ is no longer a Hilbert basis for $t > 0$, it  still has nice properties. First,  for all $t\geq0$ and $\ell\geq0,$
$i\in\overline{1,4N}$,
\begin{equation} \label{ortho_Omega_m}
\int_{\Omega^{m}\left(  t\right)  }e_{\ell}^{m}\left(  t\right)  \cdot\nabla
q_{i}\left(  t\right)  \mathrm{d}x=\int_{\partial\Omega^{m}\left(  t\right)
}q_{i}\left(  t\right)  e_{\ell}^{m}\left(  t\right)  \cdot\mathfrak{n}\left(
t\right)  \mathrm{d}\sigma=0.
\end{equation}
for $(\na q_i(t))_{i\in\overline{1,4N}}$ the basis of $\mathbb{G}(\Omega(t))$.  Moreover, we have the following

\begin{proposition} \label{prop_Riesz_Basis}
For each $t\geq0$, the family $\left(  e_{\ell}^{m}\left(  t\right)  \right)
_{\ell\geq1}$ forms a Riesz basis of $L_{\sigma}^{2}\left(  \Omega^{m}\left(
t\right)  \right)  $ namely :
\[
\overline{\operatorname*{Span}\left(  e_{\ell}^{m}\left(  t\right)  \right)
_{\ell\geq1}}=L_{\sigma}^{2}\left(  \Omega^{m}\left(  t\right)  \right)
\]
and for any $\left(  a_{\ell}\right)  _{\ell\geq1}\in\ell^{2}\left(
\mathbb{N}\right)  $
\begin{equation}
\frac{1}{\left\Vert \mathcal{S}\left(  t\right)  ^{-1}\right\Vert }\left\Vert
\left(  a_{\ell}\right)  _{\ell\geq1}\right\Vert _{\ell^{2}}\leq\left\Vert
\sum_{\ell=1}^{\infty}a_{\ell}e_{\ell}^{m}\left(  t\right)  \right\Vert
_{L^{2}\left(  \Omega^{m}\left(  t\right)  \right)  }\leq\left\Vert
\mathcal{S}\left(  t\right)  \right\Vert \left\Vert \left(  a_{\ell}\right)
_{\ell\geq1}\right\Vert _{\ell^{2}},\label{Riesz_base}%
\end{equation}
where  $\mathcal{S}\left(  t\right)  $ was introduced in $\left(
\text{\ref{S(t)1}}\right)  $.
\end{proposition}

\begin{remark}
In the above relation $\left(  \text{\ref{Riesz_base}}\right)  $ the constants
do not depend on $m$.
\end{remark}

\begin{remark} \label{rem_Riesz_basis}
Since differentiability of non-simple eigenvalues does not hold in general, see Chapter $5$ from \cite{Henrot}, we cannot use an orthogonal basis of the Stokes operator \eqref{Stokes_op} for each $\Omega^m(t)$ in order to implement a Galerkin scheme. We will rely instead on this Riesz basis given by \eqref{vectori_em(t)}.
\end{remark}

\begin{proof}
For all $\left(  a_{\ell}\right)  _{\ell\geq1}\in\ell^{2}\left(
\mathbb{N}\right)  $ the series $\sum_{\ell=1}^{\infty}a_{\ell}e_{\ell}%
^{m}\left(  0\right)  $ is convergent in $L_{\sigma}^{2}\left(  \Omega
^{m}\left(  0\right)  \right)  $. Using the continuity of $\mathcal{S}\left(
t\right)  $ and $\mathcal{S}\left(  t\right)  ^{-1}$ we obtain the convergence
of the series $\sum_{\ell=1}^{\infty}a_{\ell}e_{\ell}^{m}\left(  t\right)  $
in $L_{\sigma}^{2}\left(  \Omega^{m}\left(  t\right)  \right)  $ and that%
\[
\sum_{\ell=1}^{\infty}a_{\ell}e_{\ell}^{m}\left(  t\right)  =\mathcal{S}%
\left(  t\right)  \left(  \sum_{\ell=1}^{\infty}a_{\ell}e_{\ell}^{m}\left(
0\right)  \right)  .
\]
Obviously,
\begin{align*}
\frac{1}{\left\Vert \mathcal{S}\left(  t\right)  ^{-1}\right\Vert }\left\Vert
\left(  a_{\ell}\right)  _{\ell}\right\Vert _{\ell^{2}} &  =\frac
{1}{\left\Vert \mathcal{S}\left(  t\right)  ^{-1}\right\Vert }\left\Vert
\sum_{\ell=1}^{\infty}a_{\ell}e_{\ell}^{m}\left(  0\right)  \right\Vert
_{L^{2}}\leq\left\Vert \sum_{\ell=1}^{\infty}a_{\ell}e_{\ell}^{m}\left(
t\right)  \right\Vert _{L^{2}}\\
&  \leq\left\Vert \mathcal{S}\left(  t\right)  \right\Vert \left\Vert
\sum_{\ell=1}^{\infty}a_{\ell}e_{\ell}^{m}\left(  0\right)  \right\Vert
_{L^{2}}=\left\Vert \mathcal{S}\left(  t\right)  \right\Vert \left\Vert
\left(  a_{\ell}\right)  _{\ell}\right\Vert _{\ell^{2}}.
\end{align*}
$\ $Let us show that $\operatorname*{Span}\left(  e_{\ell}^{m}\left(
t\right)  \right)  _{\ell\ge 1}$ is dense in $L_{\sigma}^{2}\left(
\Omega^{m}\left(  t\right)  \right)  $ : consider $u\in L_{\sigma}^{2}\left(
\Omega^{m}\left(  t\right)  \right)  $. For all $\varepsilon>0$ there exists
$n_{\varepsilon}\in\mathbb{N}$ such that%
\[
\left\Vert \mathcal{S}\left(  t\right)  ^{-1}u-\sum_{\ell=1}^{n_{\varepsilon}%
}\left\langle \mathcal{S}\left(  t\right)  ^{-1}u,e_{\ell}^{m}\left(
0\right)  \right\rangle e_{\ell}^{m}\left(  0\right)  \right\Vert _{L^{2}}%
\leq\frac{\varepsilon}{\left\Vert \mathcal{S}\left(  t\right)  \right\Vert }.
\]
Obviously, we have that%
\[
\left\Vert u-\sum_{\ell=1}^{n_{\varepsilon}}\left\langle \mathcal{S}\left(
t\right)  ^{-1}u,e_{\ell}^{m}\left(  0\right)  \right\rangle e_{\ell}%
^{m}\left(  t\right)  \right\Vert _{L^{2}}\leq\left\Vert \mathcal{S}\left(
t\right)  \right\Vert \frac{\varepsilon}{\left\Vert \mathcal{S}\left(
t\right)  \right\Vert }=\varepsilon.
\]

\end{proof}

We further need to study the connection between the $H^{1}\left(  \Omega^{m}(t)\right)
$-norm and  the family $\left(  e_{\ell}^{m}(t)\right)_{\ell \ge 1}$. We begin
with the following:

\begin{proposition}
\label{Prop_norm_equiv}Consider $(\Psi_{\ell})_{\ell}\in\mathbb{\ell}%
^{2}\left(  \mathbb{N}\right)  $ such that%
\begin{equation}
\sum_{\ell=1}^{\infty}\left(  1+(\lambda_{\ell}^{m})^{2}\right)  \left\vert
\Psi_{\ell}\right\vert ^{2}<\infty. \label{norm_Hk}%
\end{equation}
Consider
\[
\psi\left(  t,x\right)  =\sum_{\ell=1}^{\infty}\Psi_{\ell}e_{\ell}^{m}\left(
t\right)  .
\]
Then $\psi\in H_{\sigma}^{1}\left(  \Omega^{m}\left(  t\right)  \right)  $ and
there exists $C(m)  >1$ such that
\begin{equation}
\frac{1}{C(m)  }\left\Vert \psi\right\Vert _{H_{\sigma}%
^{1}\left(  \Omega^{m}\left(  t\right)  \right)  }^{2}\leq\sum_{\ell
=1}^{\infty}\left(  1+(\lambda_{\ell}^{m})^{2k}\right)  \left\vert \Psi_{\ell
}\right\vert ^{2}\leq C(m) \left\Vert \psi\right\Vert
_{H_{\sigma}^{1}\left(  \Omega^{m}\left(  t\right)  \right)  }^{2}.
\label{norm_equivalence_1}%
\end{equation}

\end{proposition}

\begin{proof}
Since the proof of this result is quite short we present it here : the
convergence of the series $\left(  \text{\ref{norm_Hk}}\right)  $ along with
the discussion leading to $\left(  \text{\ref{norm_H1_spect}}\right)  $ shows
that%
\[
\psi\left(0, \cdot\right)  \in H^1_\sigma(\Omega^m(0)), \quad \|\psi(0,\cdot)\|^2_{H^1_\sigma(\Omega^m(0))} = \sum_{\ell=1}^{\infty}\left(  1+(\lambda_{\ell}^{m})^{2}\right)  \left\vert
\Psi_{\ell}\right\vert ^{2}.
\]
The result then follows from $\left(  \text{\ref{norm_H1_equiv}}\right)  $ and from the fact that
$\psi\left(  t,x\right)  =\mathcal{S}(t)\psi\left(  0,x\right)$. 
\end{proof}

The next proposition describes the behaviour of the "Fourier" coefficients $\langle \psi , e_\ell(t) \rangle_{L^2(\Omega^m(t))}$
 of a function $\psi\in H_{\sigma}^{1}\left(  \Omega^{m}(t)\right)$. The proof is postponed to Appendix \ref{appendixB}. 

\begin{proposition}
\label{Norm_H1_Spect}Consider $\psi\in H_{\sigma}^{1}\left(  \Omega^{m}\left(
t\right)  \right)  $ and let%
\begin{equation}
\Psi_{\ell}:=\left\langle \psi,e_{\ell}^{m}(t)\right\rangle _{L^{2}\left(
\Omega^{m}(t)\right)  }. \label{coef_fourier_t}%
\end{equation}
Then, the series
\begin{equation}
\sum_{\ell=1}^{\infty}\left(  1+(\lambda_{\ell}^{m})^{2}\right)  \left\vert
\Psi_{\ell}\right\vert ^{2} \label{series_H1}%
\end{equation}
is convergent and there exists $C=C(m)  >1$ such that
\begin{equation}
\sum_{\ell=1}^{\infty}\left(  1+(\lambda_{\ell}^{m})^{2}\right)  \left\vert
\Psi_{\ell}\right\vert ^{2}\leq C \left\Vert \psi\right\Vert _{H_{\sigma}%
^{1}\left(  \Omega^{m}\left(  t\right)  \right)  }^{2}.
\label{H1_wrt_val_prop_above}%
\end{equation}
Moreover, there exists $T^{\ast}\in(0,T]$ depending only on $T$ and $m$ such
that for all $t\in\left[  0,T^{\ast}\right]  $%
\begin{equation}
\frac{1}{2}\left\Vert \psi\right\Vert _{H_{\sigma}^{1}\left(  \Omega
^{m}\left(  t\right)  \right)  }^{2}\leq\sum_{\ell=1}^{\infty}\left(
1+(\lambda_{\ell}^{m})^{2}\right)  \left\vert \Psi_{\ell}\right\vert ^{2}.
\label{H1_wrt_val_prop_below}%
\end{equation}

\end{proposition}

\begin{remark}
With the same notations as above, we also have that for all $n\in
\mathbb{N}^{\ast}:$
\begin{equation}
\left\Vert \sum_{\ell=1}^{n}\left\langle \psi,e_{\ell}^{m}(t)\right\rangle
_{L^{2}\left(  \Omega^{m}(t)\right)  }e_{\ell}^{m}(t)\right\Vert
_{H^{1}\left(  \Omega^{m}\left(  t\right)  \right)  }\leq C(m) \left\Vert
\psi\right\Vert _{H^{1}\left(  \Omega^{m}\left(  t\right)  \right)  }^{2}.
\label{norm_bas_freq}%
\end{equation}

\end{remark}

As oposed to $\left(  \text{\ref{norm_equivalence_1}}\right)  $ which is valid
for all $t\in\left[  0,T\right]  $, we are only able to show the "reverse" inequality \eqref{H1_wrt_val_prop_below}
 for a small time $T^{\ast}\in(0,T]$ which is nevertheless bounded from
below by a constant that depends only on $T$ and $m$. 


A very important consequence of Proposition \ref{Norm_H1_Spect} is some sharp control of the Gram matrix associated with $(e_\ell^m)_{\ell \in \overline{1,n}}$.

\begin{proposition}
\label{Prop_gram_surject}

\begin{enumerate}
\item For all $m,n\in\mathbb{N}^{\ast}$, we let $\mathcal{G}^{m,n}\left(
t\right)  =\left(  \left\langle e_{i}^{m}\left(  t\right)  ,e_{j}^{m}\left(
t\right)  \right\rangle _{L^{2}(\Omega^{m}(t))}\right)  _{i,j\in\overline
{1,n}}.$ Then, for all $\Psi\in\mathbb{R}^{n}$ the following inequality holds
true%
\begin{equation}
\left\vert \Psi\right\vert \leq\left\Vert \mathcal{S}\left(  t\right)
^{-1}\right\Vert ^{2}\left\vert \mathcal{G}^{m,n}\left(  t\right)
\Psi\right\vert \label{gram_surject}%
\end{equation}
where $\left\vert \cdot\right\vert $ is the euclidian norm on $\R^n$

\item With the same notations and letting $T^{\ast}\in(0,T]$ be the time
defined in Proposition \ref{Norm_H1_Spect} there exists $C = C(m) > 1$ such
that
\begin{equation}
\frac{1}{C(m)}\sum_{\ell=1}^{n}(1+(\lambda_{\ell}^{m})^{2})\Psi_{\ell}^{2}\leq\sum
_{\ell=1}^{n}(1+(\lambda_{\ell}^{m})^{2})\left[  \left(  \mathcal{G}^{m,n}\left(
t\right)  \Psi\right)_{\ell}\right]  ^{2}\leq C(m) \sum_{\ell=1}^{n}(1+(\lambda
_{\ell}^{m})^{2})\Psi_{\ell}^{2}.\label{gram_surject_H1}%
\end{equation}

\end{enumerate}
\end{proposition}

\begin{proof}
We let $\Psi=\left(  \Psi_{1},\Psi_{2},\dots,\Psi_{n}\right)  \in
\mathbb{R}^{n}$ and%
\[
\psi=\sum_{\ell=1}^{n}\Psi_{\ell}e_{\ell}^{m}\left(  0\right)  \in L_{\sigma}%
^{2}\left(  \Omega^{m}\left(  0\right)  \right)  .
\]
Observe that $\left\Vert \psi\right\Vert _{L^{2}}=\left\vert \Psi\right\vert $
along with $\left(  \text{\ref{Riesz_base}}\right)  $ implies that%
\[
\left\vert \Psi\right\vert ^{2}=\left\Vert \sum_{\ell=1}^{n}\Psi_{\ell}e_{\ell}%
^{m}\left(  0\right)  \right\Vert _{L^{2}}^{2}\leq\left\Vert \mathcal{S}%
\left(  t\right)  ^{-1}\right\Vert ^{2}\left\Vert \sum_{\ell=1}^{n}\Psi_{\ell}%
e_{\ell}^{m}\left(  t\right)  \right\Vert _{L^{2}}^{2}=\left\Vert \mathcal{S}%
\left(  t\right)  ^{-1}\right\Vert ^{2}\left\langle \Psi,\mathcal{G}^{m,n}\left(  t\right)  \Psi\right\rangle _{\mathbb{R}^{n}}%
\]
from which the conclusion $\left(  \text{\ref{gram_surject}}\right)  $ follows.

In order to prove $\left(  \text{\ref{gram_surject_H1}}\right)  $ let us
observe that%
\[
\left\{
\begin{array}
[c]{l}%
\left\Vert \psi\right\Vert _{H_{\sigma}^{1}(\Omega^{m}(0))}^{2}=\sum_{\ell=1}%
^{n}(1+(\lambda_{\ell}^{m})^{2})\Psi_{j}^{2}\text{ and}\\
\left\langle \mathcal{S}\left(  t\right)  \psi,e_{\ell}^{m}\left(  t\right)
\right\rangle _{L^{2}\left(  \Omega^{m}(t)\right)  }=\left(  \mathcal{G}^{m,n}\left(  t\right)  \Psi\right)  _{\ell}.
\end{array}
\right.
\]
Using $\left(  \text{\ref{H1_wrt_val_prop_above}}\right)  $ and $\left(
\text{\ref{H1_wrt_val_prop_below}}\right)  $ we obtain the desired conclusion.
\end{proof}

\subsection{Galerkin approximation scheme and uniform in $n$ estimates }  \label{Galerkin_solve}
Having introduced the mobile Riesz basis $(e^m_\ell(t))_{\ell \ge 1}$ in the previous section, 
the goal will be to construct some field 
\begin{equation} 
u^{m,n}\left(  t\right)  = \tilde u^{m,n}(t) + u^{m,n,\nabla}(t) :=  \sum_{\ell=1}^{n}U_{\ell+4N}^{m,n}\left(  t\right)
e_{\ell}^{m}\left(  t\right)  + \sum_{i=1}^{4N}U_{i}^{m,n}\left(  t\right) 
\nabla q_{i}\left(  t\right)  \text{,}\label{umn_approx}%
\end{equation}
 solution of a variational formulation similar to \eqref{VF_prescribed}, with test functions $\psi$ of the form
\begin{equation} \label{def_psi_mn}
    \psi\left(  t\right)  = \tilde \psi(t) + \psi^{\nabla}(t) :=  \sum_{\ell=1}^{n}\Psi_{\ell+4N}^{m,n}\left(  t\right)
e_{\ell}^{m}\left(  t\right)  + \sum_{i=1}^{4N}\Psi_{i}^{m,n}\left(  t\right) 
\nabla q_{i}\left(  t\right) 
\end{equation}
 where the coefficients $\Psi_j$, $j=1,\dots, 4N+n$ are smooth over $[0,T]$. The variational formulation takes the following form: 
\begin{equation}  \label{VF_prescribed_m_n}
\begin{aligned}
& \int_{\Omega^m\left(  t\right)  } u^{m,n}\left(  t\right)  \cdot\psi\left(
t\right)  +  \int_{B(0,m)^c} u^{m,n,\nabla}\left(  t\right)  \cdot \psi^\nabla\left(t\right) +  2\nu \int_{0}^{t}\int_{\Omega^m\left(  s\right)  }\mathbb{D}\left(
u^{m,n}\right)  \cdot\mathbb{D}\left(  \psi\right)  \mathrm{d}s \\
& - \sum_{i=1}^{N}  \int_0^t c_i \left(  r_i[u^{m,n}](s)  \right)^{2-3\gamma} 
\dashint_{\partial B_{i}\left(  s\right)  } \psi  \cdot \mathfrak{n}\mathrm{d}s\\
& - \frac12 \sum_{i=1}^{N} \int_{0}^{t}  \int_{\partial B_{i}\left(  s\right)  } \left(u^{m,n} \cdot \mathfrak{n} - \dot{r}_i - \dot{x}_i \cdot \mathfrak{n}\right) (\psi  \cdot u^{m,n}) \mathrm{d}s +  \frac12\int_{\partial B(0,m) }  (u^{m,n} \cdot \mathfrak{n}) (\psi \cdot u^{m,n})   \\
& =\int_{\Omega^m(0)}  u^{m,n}_{0} \cdot\psi\left(0\right) + \int_{B(0,m)^c} u^{\nabla}_0  \cdot \psi^\nabla\left(0\right) \\ 
& +\int_{0}^{t}\int_{\Omega^m\left(  s\right)  }  u^{m,n} \cdot (\partial_{t}\psi+ u^{m,n} \cdot \na \psi)\mathrm{d}s + \int_{0}^{t}\int_{B(0,m)^c}  u^{m,n,\nabla} \cdot \partial_{t}\psi^\nabla \mathrm{d}s ,
\end{aligned}
\end{equation}
where 
\begin{equation}
    r_i[u^{m,n}](t) := \tilde{r}_{0,i} + \int_0^t \dashint_{\partial B_{i}\left(  s\right)} u^{m,n} \cdot \mathfrak{n} \,  \mathrm{d}s .
\end{equation}
We also define 
\begin{equation}
x_{i}\left[  u^{m,n}\right]  \left(  t\right)  := x_{i}\left(  0\right)
+3 \int_{0}^{t} \dashint_{\partial B_{i}(s)} (u^{m,n} \cdot\mathfrak{n}) \mathfrak{n}\mathrm{d}s.
\end{equation}

Note that compared to the original variational formulation \eqref{VF_prescribed}, most space integrals are on $\Omega^m$ instead of $\Omega$, except for a couple of integrals that are on  $B(0,m)^c$ and involve only $u^{m,n,\na}$ and $\psi^{m,n,\na}$. This is  to accomodate the fact that the fields $\na q_i(t)$ are defined over the whole of $\Omega(t)$, and to benefit from their orthogonality in $L^2(\Omega(t))$. The  extra boundary term $\frac12\int_{\partial B(0,m) }  (u^{m,n} \cdot \mathfrak{n}) (\psi \cdot u^{m,n}) $ is added because $u^{m,n} \cdot \mathfrak{n}\vert_{\pa B(0,m)} = u^{m,n,\nabla} \cdot \mathfrak{n}\vert_{\pa B(0,m)} \neq 0$. Let us take a moment to discuss the initial data $u^{m,n}_{0}$. We first notice that $u_0$ is of the form
\[
u_{0}=\tilde u_{0} + \sum_{i=1}^{4N}U_{0,i} \nabla
q_{i}\left(  0\right)
\]
where $\tilde u_0$ is the Leray projection of $u_0$  on $L^2_\sigma(\Omega(0))$, so that 
\[
\tilde u_{0}\cdot\mathfrak{n}=0\text{ on }\cup_{j=1}^{N}\partial
B_{j}\left(  0\right)  .
\]
For all $\ell \ge 1$, we define 
\[
U_{0,\ell+4N}^{m}   :=
 \left\langle \left(  \left.  \tilde u_{0}\right\vert _{\Omega^{m}\left(
0\right)  }\right)  ,e_{\ell}^{m}\left(  0\right)  \right\rangle
_{L^{2}\left(  \Omega^{m}\left(  0\right)  \right)  }.
\] 
Finally, we set 
\begin{equation}
    u_0^{m,n} := \tilde{u}_0^{m,n} + u_0^\nabla := \sum_{\ell=1}^n  U_{0,\ell+4N}^{m} e_\ell^m(0) + \sum_{i=1}^{4N} U_{0,i} \na q_i(0).  
\end{equation}
From the  orthogonality relation \eqref{ortho_Omega_m} (applied at $t=0)$, and from the Bessel inequality applied with the orthonormal family $(e_\ell^m(0))_{\ell \ge 1}$,  we get 
\begin{align*}
\| u_0^{m,n} \|^2_{L^2(\Omega^m(0))} & = \| \tilde u_0^{m} \|^2_{L^2(\Omega^m(0))} +  \| u_0^{\nabla} \|^2_{L^2(\Omega^m(0))}   \\ 
& \le  \| \mathbb{P}u_0\vert_{\Omega^m(0)} \|^2_{L^2(\Omega^m(0))} +  \| u_0^\nabla \|^2_{L^2(\Omega^m(0))} \\ 
& \le  \| \mathbb{P}u_0 \|^2_{L^2(\Omega(0))} +  \| u_0^{\nabla} \|^2_{L^2(\Omega(0))} = \| u_0 \|_{L^2(\Omega(0))}^2 .
\end{align*}


Let 
\begin{align*}
\tilde U^{m,n} & := \left( U^{m,n}_{\ell+4N} \right)_{\ell \in \overline{1,n}}, \quad  U^{m,n,\nabla} := \left( U^{m,n}_i \right)_{i \in \overline{1,4N}}, \\
\tilde U^{m}_0 & := \left( U^{m}_{0,\ell+4N} \right)_{\ell \in \overline{1,n}}, \qquad  U^{\nabla}_0 := \left( U_{0,i} \right)_{i \in \overline{1,4N}}. 
\end{align*}
Taking into account the orthogonality relation \eqref{ortho_Omega_m} in $L^2_\sigma(\Omega^m(t))$, one can check that the variational formulation \eqref{VF_prescribed_m_n} is  equivalent to a system
of ODEs of the form
\begin{equation}
\left\{
\begin{array}
[c]{l}%
\frac{d}{dt}(\mathcal{G}^{m,n}\left(  t\right)  \tilde{U}^{m,n})=\tilde
{F}^{m,n}\left(  t,\tilde{U}^{m,n},U^{m,n,\nabla}\right)  ,\\
\frac{d}{dt}U^{m,n,\nabla}=F^{n,\nabla}\left(  t,\tilde{U}%
^{m,n},U^{m,n,\nabla}\right),
\end{array}
\right.  \label{ODE_general_form}%
\end{equation}
for $C^{1}$-functions $\tilde{F}^{m,n},F^{m,n,\nabla}$, with the Gram matrix $\mathcal{G}^{m,n}$ was defined in Proposition \ref{Prop_gram_surject}. Relation \eqref{ortho_Omega_m} explains the decoupling between the time derivative of $\tilde{U}^{m,n}$ and $U^{m,n,\nabla}$: the Gram matrix of the basis $\{e_\ell(t), \ell \in \overline{1,n} \} \cup \{ \na q_i(t)\vert_{\Omega^m(t)}, i \in \overline{1,4N} \}$ is block diagonal. Moreover, the orthogonality of $(\na q_i(t))_{i \in \overline{1,4N}}$ in $L^2(\Omega(t))$ explains that the time derivative $\frac{d}{dt}U^{m,n,\nabla}$ does not contain any additional matrix. 

Cauchy-Lipschitz's theorem ensures the existence of a  solution of 
\eqref{ODE_general_form}  on some local time interval $[0, T^{m,n}]$  with initial data $\left(\tilde U_0^m, U_0^\nabla \right)$. This gives a local in time solution $u^{m,n}$ of \eqref{VF_prescribed_m_n}.   We can take $\psi = u^{m,n}$ as a test function in this variational formulation  and
we recover the energy estimates for all $t \le T^{m,n}$: 
\begin{align} 
&  \frac{1}{2}\int_{\Omega^{m}\left(  t\right)  }\left\vert \tilde u^{m,n}\left(
t\right)  \right\vert ^{2}   +  \frac{1}{2}\int_{\Omega(t)}\left\vert u^{m,n,\nabla}\left(
t\right)  \right\vert ^{2} + \mu\int_{0}^{t}\int_{\Omega^{m}\left(  \tau\right)
}\left\vert \mathbb{D}(u^{m,n})\right\vert ^{2}\mathrm{d}\tau \nonumber \\
&  +\sum_{i=1}^{N}\frac{c_{i}}{3\gamma-3}\left[  r_{i}\left[  u^{m,n}\right]  \left(
t\right)  \right]  ^{3-3\gamma}  =\frac{1}{2}\int_{\Omega^{m}\left(  0\right)  }\left\vert u_{0}%
^{m,n}\right\vert ^{2}+\sum_{i=1}^{N}\frac{c_{i}}{3\gamma-3}\left[  \tilde
{r}_{0,i}\right]  ^{3 - 3\gamma}  \le \tilde E^0  \label{energy_approx_level}%
\end{align}
where 
\begin{equation} \label{def_tilde_E0}
   \tilde E_0 :=  \frac12 \int_{\Omega_0} |u_0|^2 + \sum_{i=1}^N \frac{c_i}{3\gamma -3} \tilde r_{0,i}^{3-3\gamma}   .
\end{equation}
In particular,
$$ \frac{1}{2}\int_{\Omega^{m}\left(  t\right)  }\left\vert \tilde u^{m,n}\left(
t\right)  \right\vert ^{2}+  \frac{1}{2}\int_{\Omega(t)}\left\vert u^{m,n,\nabla}\left(
t\right)  \right\vert ^{2} \le \tilde E_0$$
which together with Proposition \ref{prop_Riesz_Basis} yields for all $t \le T^{m,n}$
$$ \left| \tilde U^{m,n}(t) \right| +  \left| U^{m,n, \nabla}(t) \right| \le C \tilde E_0. $$
By the standard blow-up criterion for ODE's, this implies that we can take $T^{m,n} = T$. 

By the uniform control of $U^{m,n, \nabla}$  and the control of $\na q_i(t)$ in $H^1(\Omega(t))$, {\it cf.} \eqref{q1}-\eqref{q1k}, we get a uniform bound on $D(u^{m,n,\nabla})$ in $L^2(0,T; H^1(\Omega))$. Together with the gradient bound from  \eqref{energy_approx_level}, this implies 
\begin{equation}
\int_{0}^{t}\int_{\Omega^{m}\left(  \tau\right)  }\left\vert \mathbb{D}%
(\tilde{u}^{m,n})\right\vert ^{2}\mathrm{d}\tau,+\int_{0}^{t}\int_{\Omega
^{m}\left(  \tau\right)  }\left\vert \mathbb{D}(u^{m,n,\nabla})\right\vert
^{2}\mathrm{d}\tau\leq C \tilde E_0 .\label{energies}%
\end{equation}

Adapting the arguments of the last paragraph of Section \ref{sec_a_priori}, one gets analogues of \eqref{coef_alphai} and \eqref{estimateRdotXdot}: 
\begin{equation} \label{estim_H2_galerkin}
    |\dot{U}^{m,n,\nabla}| \le C(\tilde E_0), \quad \left\Vert \dot{R}\left[  u^{m,n}\right]  \right\Vert
_{H^{1}\left(  0,T\right)  }+\left\Vert \dot{X}\left[  u^{m,n}\right] \right\Vert _{H^{1}\left(  0,T\right)  }\leq C(\tilde E_0)
\end{equation} 
where $R[u^{m,n}] = (r_i[u^{m,n}])_{i \in \overline{1,N}}$ and $X[u^{m,n}] = (x_i[u^{m,n}])_{i \in \overline{1,N}}$. 
In the following Proposition we establish an estimate for the time
derivative of $\partial_{t}u^{m,n}$ which will be crucial to the asymptotics $n \rightarrow +\infty$. 

\begin{proposition}
\label{ineg_deriv} Let $T^\ast$ given by Proposition \ref{Prop_gram_surject}.  For all 
\[
\tilde{\psi}\left(  t\right)  =\sum_{\ell=1}^{n}\Psi_{\ell}\left(  t\right)
e_{\ell}^{m}\left(  t\right)  ,
\]
with smooth $\Psi = \left(\Psi_{\ell}\right)_{\ell \in \overline{1,n}}$ compactly supported in $(0,T^\ast)$, we have 
\begin{equation}
\left\vert \int_{0}^{T^\ast}\left\langle \mathcal{G}^{m,n}\left(  s\right)
\dot{U}^{m;n}\left(  s\right)  ,\Psi\left(  s\right)  \right\rangle
\mathrm{d}s\right\vert \leq C(\tilde E_0) \left\Vert \tilde{\psi
}\right\Vert _{L^{4}(0,T^\ast;H_{\sigma}^{1}(\Omega^{m}))}%
.\label{ineg_deriv_temps}%
\end{equation}
\end{proposition}
\begin{proof}
From \eqref{time_deriv_moving} and from the weak formulation \eqref{VF_prescribed_m_n} we deduce
\begin{align*}
&  -\int_{0}^{T^\ast}\int_{\Omega^{m}\left(  s\right)  }\partial_{t}u^{m,n}%
\cdot\tilde{\psi}\mathrm{d}s=\int_{0}^{T^\ast}\int_{\Omega^{m}\left(  s\right)
}u^{m,n}\cdot\partial_{t}\tilde{\psi}\mathrm{d}s+\int_{0}^{T^\ast}\int_{\Omega
^{m}\left(  s\right)  }\operatorname{div}(u^{m,n}\cdot\tilde{\psi}%
v^{ALE})\mathrm{d}s\\
&  =2\nu\int_{0}^{T^\ast}\int_{\Omega^{m}\left(  s\right)  }\mathbb{D}\left(
u^{m,n}\right)  \cdot\mathbb{D}\left(  \tilde{\psi}\right)  \mathrm{d}%
s+\frac{1}{2}\int_{0}^{T^\ast}\int_{\partial\Omega^{m}\left(  s\right)}(u^{m,n}\cdot\tilde{\psi})(v^{ALE}\cdot
\mathfrak{n}-u^{m,n}\cdot\mathfrak{n})\mathrm{d}s\\
&  -\int_{0}^{T^\ast}\int_{\Omega^{m}\left(  s\right)  }u^{m,n}\cdot \left( u^{m,n}%
\cdot\nabla\tilde{\psi} \right)\mathrm{d}s.
\end{align*}
Then standard manipulations and the uniform bound \eqref{energy_approx_level} imply
\begin{equation}
\left\vert \int_{0}^{T^\ast}\int_{\Omega^{m}\left(  s\right)  }\partial_{t}%
u^{m,n}\cdot\tilde{\psi}\mathrm{d}s\right\vert \leq C(\tilde E_{0})
\left\Vert \tilde \psi\right\Vert _{L^{4}(0,T^\ast;H_{\sigma}^{1}(\Omega^{m}))}.\label{ineg_toti_termenii_din_eq}%
\end{equation}

Next, we write that
\begin{align}
&  \int_{0}^{T^\ast}\int_{\Omega^{m}\left(  s\right)  }\partial_{t}u^{m,n}%
\cdot\tilde{\psi}\mathrm{d}s\nonumber\\
&  =\sum_{\ell,\ell'=1}^{n}\int_{0}^{T^\ast}\int_{\Omega^{m}\left(  s\right)  }\dot
{U}_{\ell+4N}^{m,n}\left(  s\right)  e_{\ell}^{m}\left(  s\right)  \cdot e_{\ell'}%
^{m}\left(  s\right)  \Psi_{\ell'}\left(  s\right)  \mathrm{d}s+\sum_{\ell=1}%
^{n}\int_{0}^{T^\ast}\int_{\Omega^{m}\left(  s\right)  }U_{\ell+4N}^{m,n}\left(
s\right)  \partial_{t}e_{\ell}^{m}\left(  s\right)  \cdot\tilde{\psi
}(s)\mathrm{d}s\nonumber\\
&  +\sum_{i=1}^{4N}\int_{0}^{T^\ast}\left\langle \tilde{\psi}\left(  s\right)
,\nabla q_{i}\left(  s\right)  \right\rangle _{L^{2}\left(  \Omega^m\left(
s\right)\right)}\dot{U}_{i}^{m,n}\left(s\right) \mathrm{ds}  +\sum_{i=1}%
^{4N}\int_{0}^{T^\ast}\left\langle \tilde{\psi}\left(  s\right)  ,\partial
_{t}\nabla q_{i}\left(  s\right)  \right\rangle _{L^{2}\left(
\Omega^m\left(  s\right)  \right)  }U_{i}^{m,n}\left(  s\right) \mathrm{d}s .
\label{descompunere}%
\end{align}
Observe that the first term is exactly what we want to estimate in $\left(\text{\ref{ineg_deriv_temps}}\right)$. The second term is treated using $\left(  \text{\ref{transport_o}}\right)  $
\begin{align}
&  \sum_{\ell=1}^{n}\int_{0}^{T^\ast}\int_{\Omega^{m}\left(  s\right)  }U_{\ell+4N
}^{m}\left(  s\right)  \partial_{t}e_{\ell}^{m}\left(  s\right)  \cdot
\tilde{\psi}(s) \mathrm{d}s\nonumber\\
&  =\sum_{\ell=1}^{n}\int_{0}^{T^\ast}\int_{\Omega^{m}\left(  s\right)  }U_{\ell+4N
}^{m}\left(  s\right)  \left(  (\nabla v^{ALE})^{T}e_{\ell}^{m}\left(
s\right)  -\operatorname{div}\left(  v^{ALE} \otimes e_{\ell}^{m} \right)
\right)  \cdot\tilde{\psi}(s)\mathrm{d}s\nonumber\\
&  =\int_{0}^{T^\ast}\int_{\Omega^{m}\left(  s\right)  }\left(  (\nabla
v^{ALE})^{T}\tilde{u}^{m,n}-\operatorname{div}\left( v^{ALE}  \otimes \tilde{u}^{m,n}
\right)  \right)  \cdot\tilde{\psi}\mathrm{d}s , 
\label{transport_aplicat_pe_ei}%
\end{align}
which is bounded by $C \|\tilde\psi\|_{L^2(0,T^\ast; L^2(\Omega^m))}$ thanks to the uniform estimates \eqref{energy_approx_level}-\eqref{energies}. The third term vanishes due to the
fact that $\tilde{\psi}\cdot\mathfrak{n}=0.$ The fourth term is controlled thanks to the uniform estimate on $U^{m,n,\nabla}$ and 
the bound on $\pa_t \na q_i$, {\it cf.} \eqref{q1}-\eqref{q1k}. With this we conclude the proof of Proposition
\ref{ineg_deriv}.
\end{proof}

\bigskip

Let us observe that owing to Proposition \ref{equiv_Sobolev} we have that :
\[
\mathcal{S}(t)^{-1}\tilde{u}^{m,n}\in L^{\infty}(0,T;L_{\sigma}^{2}(\Omega
^{m}(0))\cap L^{2}(0,T;H_{\sigma}^{1}(\Omega^{m}(0)))
\]
Also, observe that
\[
\partial_{t}(\mathcal{S}(t)^{-1}\tilde{u}^{m,n})=\sum_{\ell=1}^{n}\dot
{U}_{\ell+4N}^{m,n}(t)e_{\ell}^{m}(0)
\]
The following proposition is crucial in order to recover compactness for the
solution. We denote by $H_{\sigma}^{-1}(\Omega^{m}(0))$ the dual space of
$\overline{C_{\sigma,c}^{\infty}(\Omega^{m}(0)))}^{\left\Vert \cdot\right\Vert
_{H^{1}}}.$

\begin{proposition}
\label{compactness} Let  $T^{\ast}$ given by Proposition
\ref{Prop_gram_surject}. Then,   the following estimate hold true :%
\begin{equation}
\left\Vert \partial_{t}(\mathcal{S}(t)^{-1}\tilde{u}^{m,n})\right\Vert
_{L^{\frac{4}{3}}(0,T^\ast;H_{\sigma}^{-1}(\Omega^{m}(0)))}\leq C\left(m,\tilde E_0\right)
. \label{estimate_1_compact}%
\end{equation}

\end{proposition}

\begin{proof}
Using $\left(  \text{\ref{ineg_deriv_temps}}\right)  $ from Proposition
\ref{ineg_deriv} we obtain
\begin{align*}
 \left\Vert \partial_{t}(\mathcal{S}(t)^{-1}\tilde{u}^{m,n})\right\Vert
_{L^{\frac{4}{3}}(0,T^\ast;H_{\sigma}^{-1}(\Omega^{m}(0)))} & =\sup_{\substack{ \psi\in
C_c((0,T^\ast); H_{\sigma,0}^{1}(\Omega^{m}(0))) \\ \|\psi\|_{L^4(H^1)} = 1}}\left\langle \partial
_{t}(\mathcal{S}(s)^{-1}\tilde{u}^{m,n}),\psi\right\rangle \\
&  = \sup_{\substack{\psi\in
C_c((0,T^\ast); H_{\sigma,0}^{1}(\Omega^{m}(0))) \\ \|\psi\|_{L^4(H^1)} = 1}} \int_{0}^{T^\ast}\sum_{\ell=1}^{n}\dot{U}_{\ell+4N}^{m,n}(s)\left\langle
\psi(s),e_{\ell}^{m}(0)\right\rangle _{L^{2}(\Omega^{m}(0)))} \mathrm{d}s\\
&  =\int_{0}^{T^\ast}\left\langle   \mathcal{G}^{m,n}\left(  s\right)
\dot{U}^{m}(s)  ,\left(  \mathcal{G}^{m,n}\left(  s\right)  \right)
^{-1}\Psi^m(s)\right\rangle  \mathrm{d}s\\
&  \leq C\left(\tilde E_{0}\right)  \left\Vert \check{\psi}^{m,n}\right\Vert
_{L^{4}(0,t;H^{1}(\Omega^{m}))}%
\end{align*}
where we let%
\[
\left\{
\begin{array}
[c]{l}%
\Psi_{\ell}^{m}\left(  t\right)  :=\left\langle \psi\left(  t\right)
,e_{\ell}^{m}\left(  0\right)  \right\rangle _{L^{2}\left(  \Omega\left(
t\right)  \right)  },\text{ }\\
\Psi^{m}\left(  t\right)  :=\left(  \Psi_{1}^{m}\left(  t\right)  ,\dots
,\Psi_{n}^{m}\left(  t\right)  \right)  \in\mathbb{R}^{n},
\end{array}
\right.
\]
respectively,%
\[
\left\{
\begin{array}
[c]{l}%
\check{\psi}^{m,n}\left(  t\right)  :=\sum_{\ell=1}^{n}\left(  \left(
\mathcal{G}^{m,n}\left(  t\right)  \right)  ^{-1}\Psi^{m}\left(  t\right)
\right)  _{\ell}e_{\ell}^{m}\left(  t\right)  ,\\
\tilde{\psi}^{m,n}\left(  t\right)  :=\sum_{\ell=1}^{n}\Psi_{\ell}^{m}\left(
t\right)  e_{\ell}^{m}\left(  t\right)  .
\end{array}
\right.
\]
Note that the last inequality follows from Proposition \ref{ineg_deriv}. Observe that using Proposition \ref{Norm_H1_Spect} and the estimate $\left(
\text{\ref{gram_surject_H1}}\right)  $ we end up with
\begin{align*}
\left\Vert \check{\psi}^{m,n}(t)\right\Vert _{H_{\sigma}^{1}(\Omega^{m}%
(t))}^{2} &  \leq C \left\Vert \mathcal{S}(t)^{-1}\check{\psi}^{m,n}%
(t)\right\Vert _{H_{\sigma}^{1}(\Omega^{m}(t))}^{2}\leq C\left(m\right)
\sum_{\ell=1}^{n}(1+(\lambda_{\ell}^{m})^{2})\left[  \left(  \left(
\mathcal{G}^{m,n}\left(  t\right)  \right)  ^{-1}\Psi^{m}\left(  t\right)
\right)  _{\ell}\right]  ^{2}\\
&  \leq C\left(m\right)  \sum_{\ell=1}^{n}(1+(\lambda_{\ell}^{m}%
)^{2})\left[  \Psi_{\ell}^{m}\left(  t\right)  \right]  ^{2}:=C\left(
m\right)  \left\Vert \mathcal{S}(t)^{-1}(\tilde{\psi}^{m,n})\right\Vert
_{H_{\sigma}^{1}(\Omega^{m}(0))}^{2}\\
&  \leq C\left(m\right)  \left\Vert \psi\left(  t\right)  \right\Vert
_{H_{\sigma}^{1}(\Omega^{m}(0))}^{2}.
\end{align*}

\end{proof}

\subsection{Asymptotics $n,m \rightarrow +\infty$} \label{limit_passage}

The only thing which remains to be clarified is how to recover an equation in
the limit $m,n\rightarrow\infty$. We first fix $m$, fix some $\ell \in \overline{1,n}$, and fix a test field
$$ \psi(t) = \Psi_{\ell+4N}(t) e_\ell^m + \sum_{i=1}^{4N} \Psi_i(t) \na q_i$$
for some smooth functions $\Psi_{\ell+4N}, \Psi_i$ over $[0,T]$. The goal is to send $n$ to infinity in \eqref{VF_prescribed_m_n}. This sequential continuity problem is of course very similar to the one of Section \ref{sec_seq_con}, except that the domain here does not depend on $n$, which makes it easier. There is in particular no need to modify the test function $\psi$. We start from the uniform in $n$ estimates \eqref{energy_approx_level}-\eqref{energies}-\eqref{estim_H2_galerkin} and argue as in the beginning of Section \ref{sec_seq_con}. We find some $u^m := \tilde u^{m} + u^{m,\nabla}$ where 
$$ \tilde u^m \in  L^\infty(0,T; L^2_\sigma(\Omega^m)), \quad u^{m,\nabla}(t) \in L^\infty(0,T ; L^2(\Omega))  $$
such that the extensions by zero of $\tilde u^{m}$ and $u^{m,\nabla}$ are accumulation points of the extensions by zero of $\tilde u^{m,n}$ and $u^{m,n,\nabla}$ weakly * in $L^\infty(0,T ; L^2(\R^3))$. Moreover, $D(\tilde u^m) \in L^2(0,T ; L^2(\Omega^m))$ and  $D(u^{m,\nabla}) \in L^2(0,T ; L^2(\Omega^m))$. Also $r_i[u^{m,n}]$ and $x_i[u^{m,n}]$ converge weakly in $H^2(0,T)$ and strongly in $H^1(0,T)$ to $r_i[u^m]$ and $x_i[u^m]$ for all $i \in \overline{1,4N}$.  Still reasoning as in Section \ref{sec_seq_con}, one can send $n$ to infinity in all linear terms of \eqref{VF_prescribed_m_n}. One also can take easily  the limit of one extra nonlinear term:  
\begin{align*} 
& \frac12 \sum_{i=1}^{N} \int_{0}^{t}  \int_{\partial B_{i}\left(  s\right)  } \left(u^{m,n} \cdot \mathfrak{n} - \dot{r}_i - \dot{x}_i \cdot \mathfrak{n}\right) (\psi  \cdot u^{m,n}) \mathrm{d}s  \\
  \xrightarrow[n \rightarrow +\infty]{} & \frac12 \sum_{i=1}^{N} \int_{0}^{t}  \int_{\partial B_{i}\left(  s\right)  } \left(u^{m} \cdot \mathfrak{n} - \dot{r}_i - \dot{x}_i \cdot \mathfrak{n}\right) (\psi  \cdot u^{m}) \mathrm{d}s   
\end{align*}
as $u^{m,n} \cdot \mathfrak{n}\vert_{\pa B_i} = \dot{r}_i[u^{m,n}] + \dot{x}_i[u^{m,n}] \cdot \mathfrak{n}$ converges strongly. 

The only difference is in obtaining local strong compactness for $u^{m,n}$, necessary to take the limit of the convection term. 
Owing to \eqref{estimate_1_compact} in Proposition \ref{compactness}, we recover using Aubin-Lions' lemma the strong convergence of $\left(  \mathcal{S}(t)^{-1}\tilde{u}^{m,n}\right)  _{n\geq1}$ in
$L^{2}(0,T^{\ast};L_{\sigma}^{2}\left(  \Omega^{m}\left(  0\right)  \right)
)$ which by continuity of the operator $\mathcal{S}(t)$ implies strong
convergence of $\left(  \tilde{u}^{m,n}\right)  _{n\geq1}$ in 
$$L^{2}\left(\cup_{t \in [0,T^\ast]} \{t\} \times \Omega^m(t)\right).$$ 
Strong convergence of the gradient part $u^{m,n,\nabla}$ in
$$L^{2}\left(\cup_{t \in [0,T^\ast]} \{t\} \times \left(\Omega(t) \cap K\right)\right), \quad \forall  K \Subset \R^3,$$ 
follows from the uniform bound \eqref{estim_H2_galerkin} and \eqref{q1}-\eqref{q1k}. Eventually, one is able to send $n$ to infinity in all terms of \eqref{VF_prescribed_m_n} to get: for a.e. $t \le T^\ast$
\begin{equation}  \label{VF_prescribed_m}
\begin{aligned}
& \int_{\Omega^m\left(  t\right)  } u^m\left(  t\right)  \cdot\psi\left(
t\right)  +  \int_{B(0,m)^c} u^{m,\nabla}\left(  t\right)  \cdot \psi^\nabla\left(t\right) +  2\nu \int_{0}^{t}\int_{\Omega^m\left(  s\right)  }\mathbb{D}\left(
u^{m}\right)  \cdot\mathbb{D}\left(  \psi\right)  \mathrm{d}s \\
& - \sum_{i=1}^{N}  \int_0^t c_i \left( \tilde{r}_{0,i} + \int_0^s \dashint_{\partial B_{i}\left(  s'\right)} u^{m} \cdot \mathfrak{n} \,  \mathrm{d}s'  \right)^{2-3\gamma} 
\dashint_{\partial B_{i}\left(  s\right)  } \psi  \cdot \mathfrak{n}\mathrm{d}s\\
& - \frac12 \sum_{i=1}^{N} \int_{0}^{t}  \int_{\partial B_{i}\left(  s\right)  } \left(u^{m} \cdot \mathfrak{n} - \dot{r}_i - \dot{x}_i \cdot \mathfrak{n}\right) (\psi  \cdot u^{m}) \mathrm{d}s +  \frac12\int_{\partial B(0,m) }  (u^{m} \cdot \mathfrak{n}) (\psi \cdot u^{m})   \\
& =\int_{\Omega^m(0)}  u^{m}_{0} \cdot\psi\left(0\right) + \int_{B(0,m)^c} u^{\nabla}_0  \cdot \psi^\nabla\left(0\right) \\
& +\int_{0}^{t}\int_{\Omega^m\left(  s\right)  }  u^{m} \cdot (\partial_{t}\psi+ u^{m} \cdot \na \psi)\mathrm{d}s + \int_{0}^{t}\int_{B(0,m)^c}  u^{m,\nabla} \cdot \partial_{t}\psi^\nabla \mathrm{d}s ,
\end{aligned}
\end{equation}
with 
\begin{equation}
    u_0^{m} := \tilde{u}_0^{m} + u_0^\nabla := \sum_{\ell=1}^{+\infty}  U_{0,\ell+4N}^{m} e_\ell^m(0) + \sum_{i=1}^{4N} U_{0,i} \na q_i(0)  = \mathbb{P}_{\Omega^m(0)} \tilde u_0\vert_{\Omega^m(0)} + \sum_{i=1}^{4N} U_{0,i} \na q_i(0). 
\end{equation}
Note that $u^m$  verifies the uniform   bounds 
\begin{equation} \label{energy_estimate_um}
\begin{aligned} 
  \frac{1}{2}\int_{\Omega^{m}\left(  t\right)  }\left\vert \tilde u^{m}\left(
t\right)  \right\vert ^{2} &  +  \frac{1}{2}\int_{\Omega(t)}\left\vert u^{m,\nabla}\left(
t\right)  \right\vert ^{2} + \mu\int_{0}^{t}\int_{\Omega^{m}\left(\tau\right)
}\left\vert \mathbb{D}(u^{m})\right\vert ^{2}\mathrm{d}\tau  \\
& +\sum_{i=1}^{N}\frac{c_i}{3\gamma-3}  r_{i}\left[u^{m}\right]^{3-3\gamma}  \le  \tilde E_0.
\end{aligned}
\end{equation}
as well as (with obvious notations)
\begin{equation} \label{estim_H2_galerkin_um}
    |\dot{U}^{m,\nabla}| \le C(\tilde E_0), \quad \left\Vert \dot{r}_{i}\left[  u^m\right]  \right\Vert
_{H^{1}\left(  0,T\right)  }+\left\Vert \dot{x}_{i}\left[  u^m\right] \right\Vert _{H^{1}\left(  0,T\right)  }\leq C\left(\tilde E_0\right).
\end{equation} 
We remind that equality \eqref{VF_prescribed_m} holds for all $\psi = \Psi_{\ell+4N}(t) e_\ell^m + \sum_{i=1}^{4N} \Psi_i(t) \na q_i$, and so by linearity to all $\psi$ with 
$$ \psi(t) \in V^m(t) := \text{Span}\left(\{e_\ell^m, l \in \N \} \right) \oplus \text{Span}\left(\{\na q_i(t), \: i \in \overline{1,4N}  \} \right)  $$
(and with smooth coefficients in this basis). By a density argument, see Lemma \ref{approximare_spect} in the Appendix, \eqref{VF_prescribed_m} remains true for any  $\psi = \tilde \psi + \sum_{i=1}^N \Psi_i \na q_i $
where $\Psi_i$ is Lipschitz on $[0,T^\ast]$ for all $i$, and where $\tilde \psi : [0,T] \times \R^3  \rightarrow \R^3$ is Lipschitz in time, smooth and compactly supported in $x$, is divergence-free in $\Omega^m$ and satisfies $\tilde \psi(t) \cdot \mathfrak{n}\vert_{\pa \Omega^m(t)} = 0$ (for $m$ large enough so that $B(0,m)$ contains the support of $\tilde \psi$).  

Once we have the variational solution on $[0,T^\ast]$, we can reiterate our asymptotic process $n \rightarrow +\infty$ on $\left[  T^\ast,2T^\ast\right]  $ with initial data $u^{m}\left(T^{\ast}\right)$ (recall that the restriction on $T^\ast$ was only depending  on $T$ and $m$). Strictly speaking, as $u^m$ is defined for a.e.  $t \in [0,T^\ast]$, $u^m(T^\ast)$ is not {\it a priori} defined, in this case we replace $T^\ast$ by some $t^\ast$ very close and belonging to the set of full measure where $u^m$ is pointwise defined.   Proceeding in this way a finite number of steps, we obtain a solution $u^m$ of \eqref{VF_prescribed_m} on the whole interval $\left[0,T\right]$. 

\bigskip
The final part of this asymptotic study is sending $m$ to infinity, so as to recover a weak solution of \eqref{VF_prescribed}. This is reminiscent of the analysis of sequential continuity carried in Section \ref{sec_seq_con}, although much simpler as the boundary of the bubbles is prescribed and independent of $m$. In particular, any $\psi$ considered previously  is a valid test function for $m$ large enough : we do not need to approximate the test function by some $\psi^m$. Also,  the local strong convergence of $(u^m)_{m \in \N}$ results from a more direct application of \cite{Moussa}, as the boundary of the bubbles does not depend on $m$. We skip the details for brevity. Finally, still by density, \eqref{VF_prescribed} remains true for a general test function $\tilde \psi$, see Proposition \ref{Approx_test_function} and Corollary \ref{cor_test_functions_dil}.

\section{Approximate solutions and compactness} \label{sec_approx_compact}
\subsection{Construction of approximate solutions}
We prove in this section  Proposition \ref{prop_uh}, which yields a solution on $[0,T]$, $T$ small enough, to the approximate variational formulation \eqref{VF_prescribed_h}. As explained in Section \ref{sec_strategy}, this solution stems from an iterative use of  Theorem \ref{thm_prescribed}. One splits $[0,T]$ into subintervals of small size $h$, and on each of these small intervals $[kh,(k+1)h]$, one solves a simplified problem of the form \eqref{VF_prescribed}, where the dynamics of the bubbles is prescribed: the bubbles move on $[kh,(k+1)h]$ with constant translating and dilating speeds, given in terms of the velocity field $u^h$ computed at the previous step $[(k-1)h, kh]$, this $u^h$ satisfying a relation of the form 
$$ u^h \cdot \mathfrak{n} = \dot{r}_i[u^h] + \dot{x}^i[u^h] \cdot \mathfrak{n} $$
at the boundary of the $i$-th bubble. See steps i) and ii) in  Section \ref{sec_time_stepping}. Actually, as mentioned in this paragraph, instead of the uniform subdivision $kh$, $k = 0, \dots, \frac{T}{h}$, one may need to consider a slightly non-uniform subdivision $t_{k,h} \approx kh$, to ensure that the "initial data" for the $k$-th step,  $u^h(t_{k,h})$, is well-defined and that we can take $t = t_{k,h}$ in the variational formulation. Indeed, the solution given by Theorem \ref{thm_prescribed} is only defined almost everywhere. For brevity, we do not comment more on this point.  

After the statement of Proposition \ref{prop_uh}, we pointed out that application of Theorem \ref{thm_prescribed} requires the minimal distance between the bubbles to be positive. Hence, for the iteration to work and for Proposition \ref{prop_uh} to hold,   we have to check  that if $T$ is small enough, this separation condition will be satisfied at every time step $[kh,(k+1)h]$. More precisely, we show by induction on $k$ that there exists $T_0$ depending only on the initial data, such that for all $T \le T_0$, for all $k= 0, \dots, \frac{T}{h}-1$, the fields $(X^h,R^h,u^h)$ introduced formally in Section \ref{sec_time_stepping} are indeed well-defined on $[0,(k+1)h]$ and that in particular $u^h$ solves \eqref{VF_prescribed} on $[0,(k+1)h]$.  This is obvious for $k=0$, no matter the choice of $T$: indeed, $R(t) := R_0$ and  $X(t) := X_0$ satisfy the separation condition 
$$ \frac{1}{4}  \inf_{[0,h]} \: \inf_{i \neq j } \, \left(|x_i - x_j| - r_i - r_j \right) \ge \delta_0 $$ 
where $\delta_0 > 0$ is  given by \eqref{def_delta0}. This allows to apply Theorem \ref{thm_prescribed}: we get a solution $u[X,R,u_0,\tilde R_0]$ and define eventually $(X^h,R^h,u^h) = (X,R,u)$ on $[0,h]$. Assume  now $k \ge 1$ and that the iteration has provided  $(X^h,R^h,u^h)$ on $[0,kh]$. In particular, $u^h$ satisfies the variational formulation \eqref{VF_prescribed_h} for a.e. $t \in [0,kh]$. 
Adapting the arguments of Section \ref{sec_a_priori}, one finds estimates  analogue to \eqref{upper_bound_dotR} and \eqref{upper_bound_dotX}: for all $t \in (0,kh)$, for all $i \in \overline{1,N}$,
\begin{equation} \label{estimate_dotRh_dotXh}
    |\dot{r}_i[u^h](t)| \le C_R\left(\tilde E_0,  \frac{1}{r^h_i(t)},  \frac{1}{\delta^h(t)}\right), \quad  |\dot{x}_i[u^h](t)| \le C_X\left(\tilde E_0,\frac{1}{r^h_i(t)}, r^h_i(t), \frac{1}{\delta^h(t)} \right)   
\end{equation}
where $r_i[u^h]$ and $x^i[u^h]$ were given in \eqref{def_ri_uh}-\eqref{def_xi_uh}, the energy $\tilde E_0$ was defined in \eqref{def_tilde_E0}, and where 
$$ \delta^h(t) := \frac{1}{4} \left( |x_i^h(t) - x^h_j(t)| - r^h_i(t) - r_j^h(t) \right). $$
The functions $C_R$ and $C_X$ in \eqref{estimate_dotRh_dotXh} are increasing in each of their arguments. Note that by construction, see step ii) in Section \ref{sec_time_stepping}, one has 
\begin{equation} \label{link_R_Ru}
\begin{aligned}
\forall k' \le k-1, \quad \forall t \in [k' h,(k'+1)h], \quad & \left|\dot{r}_i^h(t)\right| \le  \sup_{s \in [(k'-1)h,k'h]} \left|\dot{r}_i[u^h](s)\right|, \\
 & \left|\dot{x}_i^h(t)\right| \le \sup_{s \in [(k'-1)h,k'h]} \left|\dot{x}_i[u^h](s)\right|.
\end{aligned}
\end{equation}
Let 
$$    C_{0,R} :=  C_R\left(\tilde E_0,  \frac{2}{\inf_i r_{0,i}},  \frac{2}{\delta_0}\right), \quad  C_{0,X} := C_X\left(\tilde E_0, \frac{2}{\inf_i r_{0,i}}, 2 \sup_i r_{0,i} ,  \frac{2}{\delta_0}\right).   $$
Let $T_0$ such that 
$$ \sup_i r_{0,i} + C_{0,R} T_0 \le 2 \sup_i r_{0,i},  \quad \inf_i r_{0,i} - C_{0,R} T_0 \ge   \frac{\inf_i r_{0,i}}{2}, \quad   \delta_0 - 2 C_{0,R} T_0 - 2 C_{0,X} T_0 \ge \frac{\delta_0}{2}.   $$
Using relations \eqref{estimate_dotRh_dotXh} and \eqref{link_R_Ru}, one can show inductively on $k'$ that  for all $k' \le k-1$, for all $t \in [k',(k'+1)h]$ (hence for all $t \in [0,kh]$):  
\begin{align*}
&  |\dot{r}_i[u^h](t)| \le C_{0,R}, \quad  |\dot{x}_i[u^h](t)| \le C_{0,X}, \quad \delta^h(t) \ge \delta_0 - 2 C_{0,R} t - 2 C_{0,X} t. 
\end{align*}
{\em as long as $t \le T_0$}. In particular, the functions 
 \begin{align*}
   R(t) := R^h(kh) + \left( \dashint_{(k-1)h}^{kh} \dot{R}[u^h] \right) t, \quad X(t) := X^h(kh) + \left( \dashint_{(k-1)h}^{kh} \dot{X}[u^h] \right) t
     \end{align*}
introduced in Step ii) of the iteration of Section \ref{sec_time_stepping} satisfy the separation condition of Theorem \ref{thm_prescribed}, which allows to complete Step ii) of the  iteration at step $k$.

\bigskip
The goal of the next two sections is to conclude the proof of Theorem \ref{main_thm}, by extracting from the family $(u^h)_{h > 0}$ given in Proposition \ref{prop_uh} a subsequence that converges to a weak solution of \eqref{main}. As approximations $u^h$ are defined on a small time interval only, this weak solution will be defined on $(0,T)$ for small $T$. We will then show that it can be continued until time $T = T_*$ which is either arbitrarily large, or characterized by 
$$ \liminf_{t \rightarrow T_*} \: \inf_{i \neq j} \left(|x_i(t) - x_j(t)| - r_i - r_j\right) = 0.$$

\subsection{Compactness} \label{sec_compactness}
We recall from the previous paragraph that for $T \le T_0$ small enough, one has for $\delta := \frac{\delta_0}{2}$, for all $h$, 
$$ \inf_{[0,T]}  \inf_{i \neq j} |x_i^h(t) - x_j^h(t)| - r_i^h(t) - r_j^h(t) \ge \delta. $$

\paragraph{Uniform bounds.}
The starting point is the energy estimate 
$$ \frac12 \int_{\Omega[X^h,R^h](t)} |u^h(t)|^2 + \nu \int_0^t \int_{\Omega[X^h,R^h](s)} |\mathbb{D}(u^h)(s)|^2 \mathrm{d}s + \sum_{i=1}^N \frac{c_i}{3\gamma -3} r_i[u^h](t)^{3-3\gamma}  \le E_0 $$
see \eqref{def_ri_uh}. This provides uniform bounds on 
$$\|u^h\|_{L^\infty(0,T ; L^2(\Omega[X^h,R^h]))}, \quad \|\na u^h\|_{L^2(0,T ; L^2(\Omega[X^h,R^h]))}$$
 as well as a uniform lower bound 
 $$ \forall 1 \le i \le N, \forall t \in [0,T], \quad r_i[u^h] \ge c(E_0) > 0.$$
Note that for all $k = 0, \frac{T}{h}-1$, from formula \eqref{def_iter_R_X}, 
$$ r_i[u^h](kh) = r^h_i((k+1)h) + r^h_i(0) - r^h_i(h)  = r^h_i((k+1)h)  $$
so that 
$$ \forall 1 \le i \le N, \forall t \in [0,T], \quad r^h_i(t) \ge c(E_0) > 0.$$
From there, one can proceed exactly as in Section \ref{sec_a_priori}, and infer the uniform estimate 
\begin{equation}
   |\dot{R}[u^h](t)| \le C(E_0,\delta), \quad \text{hence } \:  |R[u^h](t)| \le C(E_0,R_0,\delta,T).
\end{equation}
from which it follows straightforwardly that 
\begin{equation}
   |\dot{R}^h(t)| \le C(E_0,\delta), \quad  \text{hence } \:  |R^h(t)| \le C(R_0,E_0,\delta,T).
\end{equation}
As in Section \ref{sec_a_priori}, one can then derive the analogue estimates 
$$   |\dot{X}[u^h](t)| \le C(R_0,E_0,\delta,T), \quad  \text{hence } \:  |X[u^h](t)| \le C(X_0,R_0,E_0,\delta,T)  $$
and in turn
$$    |\dot{X}^h(t)| \le C(R_0,E_0,\delta,T), \quad  \text{hence } \:  |X^h(t)| \le C(X_0,R_0,E_0,\delta,T).  $$
After these uniform Lipschitz estimates, one can derive $H^2(0,T)$ estimates exactly as in Section  \ref{sec_a_priori}. Namely, one can introduce $(\na q^h_i(t))_{1 \le i \le 4N}$  an orthonormal basis of the space $\mathbb{G}(\Omega[X^h,R^h](t))$. Note that the fields $\na q^h_i$ satisfy bounds of the form \eqref{q1} uniformly in $h$. One can then use $\psi^h(t,x) = s(t) \na q^h_i(t)(x)$ in the variational formulation \eqref{VF_prescribed_h}. The only difference is the presence of the extra term 
\begin{align*}
& \left| \frac12 \sum_{i=1}^{N} \int_{0}^{t}  \int_{\partial B_{i}[X^h,R^h]\left(  s\right)  } \left(u^h \cdot \mathfrak{n} - \dot{r}_i^h - \dot{x}_i^h \cdot \mathfrak{n}\right) (\psi^h  \cdot u^h) \mathrm{d}s \right|   \\
& \le C \sup_{[0,T]}\left|  \left( \dot{R}[u^h] ,  \dot{R}^h ,  \dot{X}[u^h] ,  \dot{X}^h \right) \right| \|s\|_{L^2(0,T)} \, \left( \| u^h\|_{L^2(0,T ; H^1(\Omega)} \|\na q_i^h \|_{L^\infty(0,T ;H^1(\Omega)} \right) \\
& \le  C(R_0,X_0,E_0,\delta,T) \, \|s\|_{L^2(0,T)}.
\end{align*}
This results in 
\begin{equation*}
\left\vert \int_{0}^{T}\int_{\Omega\left(  t\right)  }\dot{s}\left(
t\right)  u^h\left(  t\right) \cdot \nabla q^h_{i}\left(  t\right)  \right\vert \leq
C\left(X_0,R_0,E_0,\delta,T\right)  \left\Vert s\right\Vert _{L^{2}%
}.
\end{equation*}
Proceeding as in Section \ref{sec_a_priori}, we find 
$$ \|\dot{R}[u^h]\|_{H^1(0,T)} + \|\dot{X}[u^h]\|_{H^1(0,T)} \le C\left(X_0,R_0,E_0,\delta,T\right).  $$
Note that we do not have the same $H^1(0,T)$ bounds on $\dot{R}^h,\dot{X}^h$ though,  as these are piecewise constant functions, hence not in $H^1$. Still, they obey the uniform $L^{\infty}$ estimates seen before.
\paragraph{Limit $h \rightarrow 0$.} The treatment of the limit $h \rightarrow 0$ in \eqref{VF_prescribed_h} is very close to the treatment of the limit $n \rightarrow +\infty$ in \eqref{VF_n}, already considered in Section \ref{sec_seq_con}. We shall stress the few differences between the two. First, using the uniform bounds of the previous paragraph, and extending $u^h$  by the formula 
$$ \overline{u}^h = u^h  \: \text{ in } \Omega[X^h,R^h], \quad \overline{u}^h(t,x) = \dot{x}_i[u^h](t) + \dot{r}_i[u^h](t) \frac{x - x_i^h(t)}{r_i^h(t)}, \quad x \in B_i^h(t)   $$
we find the following convergences (after possible extraction of a subsequence): 
\begin{align*}
  & \overline{u}^h \rightarrow  \overline{u} \text{ weakly* in } L^\infty(0,T ; L^2(\R^3)), \\
  & (X^h,R^h) \rightarrow  (X,R)   \text{ weakly* in }  W^{1,\infty}([0,T]), \\
  & (X[u^h],R[u^h]) \rightarrow  (\tilde{X},\tilde{R})  \text{ weakly in } H^2(0,T)  
\end{align*}
 From this and the formula
$$ \dot{R}^h(t) =  \sum_{k=1}^{\frac{T}{h}-1} 1_{[kh,(k+1)h)}(t) \, \dashint_{(k-1)h}^{kh} \dot{R}[u^h].
$$
we infer that $\dot{R}^h$ converges uniformly to $\frac{d\tilde{R}}{dt}$, which implies $\tilde{R} = R$.  Similarly,$\dot{X}^h$ converges uniformly to $\frac{d \tilde{X}}{dt}$ and so $\tilde{X} = X$. To summarize: 
\begin{align*}
  & \overline{u}^h \rightarrow  \overline{u} \text{ weakly* in } L^\infty(0,T ; L^2(\R^3)), \\
  & (X^h,R^h) \rightarrow  (X,R)   \text{ strongly in }  W^{1,\infty}([0,T]), \\
  & (X[u^h],R[u^h]) \rightarrow  (X,R)  \text{ weakly in } H^2(0,T) .  
\end{align*}
We denote in short 
$$ B_i = B_i[X,R], \quad \Omega := \Omega[X,R].$$
Exactly as in Section \ref{sec_seq_con}, we fix some small $\eta > 0$ and use that 
$$ \sup_{t \in [0,T)} d_{Haus}(B_i[X^h,R^h](t), B_i(t)) \xrightarrow[h \rightarrow 0]{}  0  $$
so  that for $h$ small enough, for all $i$, for all $t$,  
$$B_i[X^h,R^h](t) \subset B_i^\eta(t), \quad \text{ where } \:   B_i^\eta(t) := B(x_i(t), r_i(t)+\eta), \quad \Omega^\eta:= \R^3 \setminus \cup_{i=1}^N B_i^\eta . $$ 

We now fix some test field  
$$ \psi \in W^{1,\infty}\left([0,T) ; L_{\mathrm{dil}}^{2}(\Omega) \cap H^\infty(\Omega) \right) .$$
We rely on Lemma \ref{lem_diffeo}, and set 
 $\Theta^h(t,x) :=  \Theta[X(t),R(t),X^h(t),R^h(t)](x)$, as well as
\begin{equation} \label{def_psi_h}
 \psi^h(t,x) := \det\left(D\Theta^h(t,x)\right) \left(D\Theta^h(t,x)\right)^{-1}\psi(t,\Theta^h(t,x)).    
\end{equation}
It is shown to be an admissible test function in \eqref{VF_prescribed_h}, and the idea is to  let $h \rightarrow 0$ in each term of this variational formulation, using the domain decomposition
$$ \Omega^h(t) = \Omega^\eta(t) \cup \left(\Omega^h(t) \setminus  \Omega^\eta(t)\right) .$$ 
All terms in \eqref{VF_prescribed_h}  have their analogues in \eqref{VF_n} and can be treated along the lines of Section \ref{sec_seq_con}, except for the fourth one, namely
$$I^{h} :=  \frac12 \sum_{i=1}^{N} \int_{0}^{t}  \int_{\partial B_{i}[X^h,R^h]\left(  s\right)  } \left(u^h \cdot \mathfrak{n} - \dot{r}_i^h - \dot{x}_i^h \cdot \mathfrak{n}\right) (\psi^h \cdot u^h) \mathrm{d}s $$
that we bound by: 
\begin{align*}
|I_h| & \le \frac12 \sum_{i=1}^{N} \left( \| u^h(t,x) \cdot \mathfrak{n} - \dot{r}_i^h(t) - \dot{x}_i^h(t)\cdot \mathfrak{n} \|_{L^\infty(0,T ; L^\infty(\pa B_i[X^h,R^h])} \right. \\
& \hspace{3cm} \left. \|\psi^h\|_{L^2(0,T ; L^2(\pa B_i[X^h,R^h])} \|u^h\|_{L^2(0,T ; L^2(\pa B_i[X^h,R^h])} \right) \\
& \le C  \sup_{[0,T]} \left( \left| \frac{d}{dt}( R[u^h] - R^h) \right|   + \left| \frac{d}{dt}( X[u^h] - X^h) \right| \right) 
\end{align*}
where the right-hand side goes to zero as $h \rightarrow 0$. This concludes the existence of a weak solution to \eqref{main} in short time. 

\subsection{Extension of the solution up to collision}
It remains to prove the last part of the main Theorem \ref{main_thm}, which is standard. We assume that we can not construct a weak solution on an arbitrary time interval $[0,T)$, and introduce the supremum $T^\ast < +\infty$ of all times $T$ such that there exists a weak solution on $[0,T)$.  This supremum is actually a maximum: if we select a sequence of times $T^n \xrightarrow[n \rightarrow +\infty]{} T^\ast$ with $u^n$ a weak solution on $[0,T^n)$, we can apply the sequential continuity result of Section \ref{sec_seq_con} and obtain a solution $u$ on $[0, T^*)$. We now wish to prove that \eqref{collision} holds for $T = T^\ast$. We argue by contradiction and suppose that there exists $\delta > 0$ such that 
\begin{equation} \label{separation_Tast}
\forall t < T^\ast, \quad  \inf_{i \neq j} \left(|x_i(t) - x_j(t)| - r_i(t) - r_j(t)\right) \ge \delta.
\end{equation}
Then, we can select a time $t^\ast$ arbitrary close to $T^\ast$ such that $u(t^\ast)$ is well-defined, with $u(t^\ast) \in L^2_{\textrm{dil}}(\Omega(t^\ast))$. We can then build a weak solution $u'$ on $[t^\ast, t^\ast + T_0]$, with initial condition $u'(t^\ast) = u(t^\ast)$. Due to the separation condition \eqref{separation_Tast}, the time $T_0$ is lower bounded by a constant depending only on $T^\ast$, the initial data and $\delta$.  In particular, it can be taken independent of $t^\ast$. Gluing the two solutions, that is settting 
$$ \overline{u} = u \quad \text{ on } \: [0,t^\ast], \quad  \overline{u} = u' \quad \text{ on } \: [t^\ast, t^\ast + T_0]$$
we obtain a weak solution $\overline{u}$ on $[0, t^\ast+T_0]$, which contradicts the definition of $T^\ast$. 

\bigskip
This concludes the proof of Theorem \ref{main_thm}. 

\section*{Acknowledgements}
The authors thank Franck Boyer for pointing out the connection between the weak formulation \eqref{VF_prescribed} and the strong form \eqref{strong_formulation_prescribed}. They also thank Matthieu Hillairet for an interesting discussion on the formal derivation of system \eqref{main}. Both authors acknowledge the support of Project {\em ComplexFlows}  (ANR-23-EXMA-0004), and of the Project {\em Complexcité} of Université Paris Cité. D.G.-V. acnowledges the support of Project Bourgeons, grant ANR-23-CE40-0014-01 of the French National Research Agency (ANR).
C.B. acknowledges the partial support by the Agence Nationale pour la Recherche grant CRISIS (ANR-20-CE40-0020-01).

\appendix

\section{An arbitrary Lagrangian-Eulerian vector field} \label{appendixA}

In this part of the Appendix, we want to construct a family of diffeomorphisms-forward of the initial configuration
$\Omega[X,R]\left(  0\right)  $ respectively $\Omega[X,R]\left(  0\right) \cap B(0,m)$. The
precise statement was the object of Lemma \ref{diffeoLemma}, thus in the
following lines we present the proof of this result. 

\medskip
\noindent 
\textbf{Proof of Lemma \ref{diffeoLemma}.} For brevity, we shall omit  $[X,R]$ from the notations, and denote
$$ \Omega^m(t) = \Omega(t) \cap B(0,m). $$ 
We remind the definition of $\delta$ in \eqref{separation_condition}. Consider $\omega\in C^{\infty}\left(  \mathbb{R}\right)  $ with
$\operatorname*{Supp}\omega\subset\left[  -1,1\right]  $ and $\int
_{\mathbb{R}}\omega\left(  r\right)  \mathrm{d}r=1$ along with the functions%
\begin{equation}
\chi_{i}\left(  t,r\right)  = \frac{4}{\delta}\int_{\mathbb{R}}\omega\left(  \frac
{4(s-r)}{\delta}\right)  \mathbf{1}_{[0,r_{i}\left(  t\right)  +\frac{\delta
}{2}]}\left(  s\right)  \mathrm{d}s. \label{definitie_chi}%
\end{equation}
They are $C^{\infty}$, $\operatorname*{Supp}%
\chi_{i}\subset\left[  0,r_{i}\left(  t\right)  +\frac{3}{4}\delta\right]  $,
$\chi_{i}\left(  r\right)  =1$ on $\left[  0,r_{i}\left(  t\right)
+\frac{\delta}{4}\right]  $ and
\begin{equation}
\left\vert \chi_{i}^{(n)}\left(  t,r\right)  \right\vert \leq\frac{C_{n}%
}{\delta^{n}}. \label{derivate_chi}%
\end{equation}
We consider%
\begin{equation}
v^{ALE}\left(  t,x\right)  =\sum_{i=1}^{N}\chi_{i}\left(  t,\left\vert
x-x_{i}\left(  t\right)  \right\vert \right)  \left(  \dot{x}_{i}\left(
t\right)  +\frac{\dot{r}_{i}\left(  t\right)  }{r_{i}\left(  t\right)
}\left(  x-x_{i}\left(  t\right)  \right)  \right).
\label{definition_vector_ALE}%
\end{equation}
It satisfies item 1) of Lemma  \ref{diffeoLemma}. We then  let 
\begin{equation}
\left\{
\begin{array}
[c]{l}%
\dot{\Theta}\left(  t,x\right)  =v^{ALE}\left(  t,\Theta\left(  t,x\right)
\right)  ,\\
\Theta\left(  t,x\right)  =x.
\end{array}
\right.  \label{definition_of_our_X}%
\end{equation}

It is well-known that  $\Theta(t , \cdot)$, $t \ge 0$,  defines a family of diffeomorphism from $\mathbb{R}^{3}$ to $\mathbb{R}^{3}$. Let $x \in B\left(x_i(0), r_i(0) + \frac{\delta}{8}\right)$. Owing to the continuity in time we have that $\Theta\left(
t,x\right)  \in B\left(  x_{i}\left(  t\right)  ,r_{i}\left(  t\right)
+\frac{\delta}{4}\right)  $ at least for small time. Expression of $v^{ALE}$ on this set yields  
\[
\dot{\Theta}\left(  t,x\right)  =\dot{x}_{i}\left(  t\right)  +\frac{\dot
{r}_{i}\left(  t\right)  }{r_{i}\left(  t\right)  }\left(  \Theta\left(
t,x\right)  -x_{i}\left(  t\right)  \right)
\]
which can be re-written as%
\[
\frac{d}{dt}\left\{  \frac{\Theta\left(  t,x\right)  -x_{i}\left(  t\right)
}{r_{i}\left(  t\right)  }\right\}  =0\Rightarrow\frac{\Theta\left(
t,x\right)  -x_{i}\left(  t\right)  }{r_{i}\left(  t\right)  }=\frac
{x-x_{i}\left(  0\right)  }{r_{i}\left(  0\right)  },
\]
at least for a short time. By the same kind of reasoning we obtain that the set 
$$ \left\{ t \in \left[  0,T\right], \: \forall x \in B\left(x_i(0), r_i(0) + \frac{\delta}{8}\right) \quad \Theta(t,x) = x_i(t) + \frac{r_i(t)}{r_i(0)}  (x - x_i(0)) \right\} $$ 
is open. As it is closed, it equals in fact the whole interval $\left[0,T\right]$. This proves item 2) of the lemma. From this item, it follows that 
$\Theta(t, \cdot)$ maps each $B_i(0)$ to $B_i(t)$. Moreover, since the vector field $v^{ALE}(t, \cdot)$ vanishes in the exterior of $B\left(0,m\right)  $ for  $m \ge m_{0}(T)$ large enough, we deduce that $\Theta(t,\cdot)$ defines by restriction a  diffeomorphism from $\Omega\left(  0\right) = \mathbb{R}^{3}\backslash\cup B_{i}\left(0\right)$ to $\Omega\left(  t\right)
=\mathbb{R}^{3}\backslash\cup B_{i}\left(  t\right)  $ for all $t\in\left[
0,T\right]$,  for all $m\geq m_{0}$.  This concludes the proof of Lemma \ref{diffeoLemma}.

Lemma \ref{lem_diffeo} which was enounced in
Section \ref{sec_seq_con} can now be obtained as a corollary. 
\medskip
\noindent 
\textbf{Proof of Lemma \ref{lem_diffeo}.} We introduce 
\[
\left(  X\left(  s\right)  ,R\left(  s\right)  \right)  =\left(  1-s\right) \left(  \tilde{X}_{0},\tilde{R}_{0}\right)
 + s \left(  X_{0},R_{0}\right)  .
\]
We then consider the diffeomorphism $\Theta\left(  s,\cdot\right)  :\mathbb{R}%
^{3}\rightarrow\mathbb{R}^{3}$ constructed in Lemma \ref{diffeoLemma}. Then
$\Theta_{0}:=\Theta\left(  1,\cdot\right)  $ verifies all properties of the lemma. 

\section{Functional spaces on moving domains} \label{appendixB}


\subsection{Some useful formulas}

Recall that if we have a family of diffeomorphisms
\[
\Theta\left(  t,\right)  :\Omega_{0}\rightarrow\Theta\left(  t,\Omega
_{0}\right)
\]
then a surface in $\Gamma_{0}\subset\Omega_{0}$ will be transported into the
surface $\Gamma_{t}=\Theta\left(  t,\Gamma_{0}\right)$. Let us denote
$\mathfrak{n}_{t}$ the outward facing normal at $\Gamma_{t}$, and $\Theta_t := \Theta(t,\cdot)$. Suppose that we
describe a portion of $\Gamma_{0}$ with $\psi\left(  x\right)  =0.$ Then
$\Gamma_{t}$ is described locally by $\psi\left(  \Theta_t^{-1}\left(x\right)  \right)  =0$ and therefore the direction of the normal vector at
$x$ pointing outward the domain is given by%
\begin{align*}
\left(  D(\psi\circ\Theta_t^{-1})\left(  x\right)  \right)  ^{T} &  =\left(
D\psi_{|\Theta_t^{-1}(x)}D\Theta_t^{-1}(x)\right)  ^{T}\\
&  =\left(  D\Theta_t^{-1}(x)\right)  ^{T}\left(  D\psi_{|\Theta_t^{-1}%
(x)}\right)  ^{T}=\left(  D\Theta_t^{-1}(x)\right)  ^{T}\mathfrak{n}%
_{0}\left( \Theta_t^{-1}(x)  \right)  \left\Vert \nabla
\psi\left(  \Theta_t^{-1}(x)\right)  \right\Vert \\
&  =\frac{1}{\det D\Theta_t\left(\Theta_t^{-1}\left(x\right)  \right)
}\operatorname{Cof}D\Theta_t\left(\Theta^{-1}_t\left(x\right)  \right)
\mathfrak{n}_{0}\left(  \Theta_t^{-1}\left(x\right)  \right)  \left\Vert
\nabla\psi\left(\Theta_t^{-1}(x)\right)  \right\Vert .
\end{align*}
Of course,
\begin{align}
\mathfrak{n}_{t}\left(  x\right)   &  =\frac{\left(  D(\psi\circ\Theta_t
^{-1})\left( x\right)  \right)  ^{T}}{\left\Vert \left(  D(\psi\circ
\Theta_t^{-1})\left( x\right)  \right)  ^{T}\right\Vert }=\frac
{\operatorname{Cof}D\Theta_t\left(\Theta^{-1}_t\left(x\right)  \right)
\mathfrak{n}_{0}\left(  \Theta_t^{-1}\left(x\right)  \right)  }{\left\Vert
\operatorname{Cof}D\Theta_t\left(\Theta^{-1}_t\left(x\right)  \right)
\mathfrak{n}_{0}\left(  \Theta^{-1}_t\left(x\right)  \right)  \right\Vert
}\nonumber\\
&  =\frac{\operatorname{Cof}D\Theta^{-1}_t\left(x\right)  \mathfrak{n}%
_{0}\left(  \Theta^{-1}_t\left(x\right)  \right)  }{\left\Vert
\operatorname{Cof}D\Theta^{-1}_t\left(x\right)  \mathfrak{n}_{0}\left(
\Theta^{-1}_t\left(x\right)  \right)  \right\Vert }%
.\label{change_of_normal_Euler}%
\end{align}
In Lagrangian coordinates (Lagrangian in
the sense that  $x$ is in the initial space $\Omega(0)$)
\begin{equation}
\mathfrak{n}_{t}\left(  \Theta\left(  t,x\right)  \right)  =\frac
{\operatorname{Cof}D\Theta_t\left(x\right)  \mathfrak{n}_{0}\left(
x\right)  }{\left\Vert \operatorname{Cof}D\Theta_t\left(x\right)
\mathfrak{n}_{0}\left(  x\right)  \right\Vert }.\label{normal_Lagrange}%
\end{equation}

The following formulas are useful for switching from the Eulerian point of
view to the Lagrangian point of view and viceversa. Below, the upperscript $E$
means that the function is defined on $\Theta_{t}\left(  \Omega_{0}\right)  $
while the upperscript $L$ means that the function is defined on $\Omega_{0}%
$.\ The functions we are working with are defined either on the initial
configuration $\Omega_{0}$ either on $\Theta_{t}\left(  \Omega_{0}\right)  .$
Somehow abusing notations we let%
\[
\varphi^{E}=\varphi^{L}\circ\Theta^{-1}_t  \text{ and
}\varphi^{L}=\varphi^{E}\circ\Theta_t.
\]
By differentiating these relations and thanks to  Piola's identity
\[
\partial_{k}\operatorname{Cof}(D\Theta_t)_{ik}=0,
\]
 we have the following formulas%
\begin{align}
\left(  \partial_{i}\varphi^{E}\right)  ^{L}  &  =\partial_{i}\varphi^{E}%
\circ\Theta\left(  t,\cdot\right)  =\frac{1}{\det D\Theta_t}\partial_{k}\left(
(\operatorname{Cof}D\Theta_t)_{ik}\varphi^{L}\right)  ,
\label{euler_to_lagrange}\\
\left(  \partial_{i}\varphi^{L}\right)  ^{E}  &  =\partial_{i}\varphi^{L}%
\circ\Theta^{-1}_t\left(\cdot\right)  =\frac{1}{\det D\Theta^{-1}_t}%
\partial_{k}\left(  (\operatorname{Cof}D\Theta^{-1}_t)_{ik}\varphi^{E}\right)
\label{lagrange_to_euler}%
\end{align}
In particular, it holds true that%

\begin{equation}
\left(  \operatorname{div}v^{E}\right)  \circ\Theta^{-1}_t\left(x\right)
=\left(  \partial_{i}(v^{e})^{i}\right)  \circ\Theta^{-1}_t\left(x\right)
=\frac{1}{\det D\Theta_t\left(x\right)  }\partial_{k}\left(
\operatorname{Cof}(D\Theta_t)_{ik}(v^{L})^{i}\right)  .
\label{change_of_divergence}%
\end{equation}

\subsection{$L^{2}$-spaces on moving domains}

{\bf Proof of Proposition \ref{propSt}.}
We first prove that 
\begin{equation}
\operatorname{div} v=0\Rightarrow\operatorname{div}\mathcal{S}\left(  t\right)  v=0.
\label{consequence1}%
\end{equation}
and
\begin{equation}
v\cdot\mathfrak{n}=0\text{ on }\partial B_{i}\left(  0\right)  \Rightarrow
\mathcal{S}\left(  t\right)  v\cdot\mathfrak{n}_{t}=0. \label{consequence_2}%
\end{equation}
The first property follows immediately from $\left(
\text{\ref{change_of_divergence}}\right)  $ while $\left(
\text{\ref{consequence_2}}\right)  $ will follow from $\left(
\text{\ref{change_of_normal_Euler}}\right)  $ along with the particular form
of $\Theta\left( t,\cdot\right)$:
$$\forall x \in B(x_i(t), r_i(t) + \frac{\delta}{8}), \quad   \Theta(t,x) = x_i(t) + \frac{r_i(t)}{r_i(0)}(x-x_i(0)) $$
Let us prove $\left(
\text{\ref{consequence_2}}\right)  $ by first observing that for all
$x\in\partial B_{i}\left(  t\right)  $, $\Theta^{-1}_t\left(x\right)
\in\partial B_{i}\left(  0\right)  $ and thus if
\[
v\cdot\mathfrak{n}_{0}=0\text{ on }\partial B_{i}\left(  0\right)
\]
for all $i\in \overline{1,N}$ then%
\begin{gather*}
v\left( \Theta_t^{-1}(x)  \right)  \cdot\mathfrak{n}_{0}\left(
\Theta_t^{-1}(x)  \right)  =0\\
\Rightarrow v\left( \Theta_t^{-1}(x)  \right)  \cdot
(D\Theta_t)_{|x=\Theta_t^{-1}(x)}^{T}\operatorname{Cof}(D\Theta_t)_{|x=\Theta_t^{-1}(x)}\mathfrak{n}_{0}\left( \Theta_t^{-1}(x)  \right)  =0.
\end{gather*}
We obtain that%
\[
\left\{  (D\Theta)_{|x=\Theta_t^{-1}(x)}v\left(  \Theta^{-1}\left(
t,x\right)  \right)  \right\}  \cdot\left\{  \operatorname{Cof}(D\Theta
)_{|x=\Theta_t^{-1}(x)}\mathfrak{n}_{0}\left( \Theta_t^{-1}(x)
\right)  \right\}  =0.
\]
Taking into account $\left(  \text{\ref{normal_Lagrange}}\right)  $ and the formula for $\mathcal{S}(t)v$ in \eqref{S(t)1}, we end up with
$\left(  \text{\ref{consequence_2}}\right)  $. This shows that $\mathcal{S}(t)$ sends $L^2_\sigma(\Omega(0))$ to $L^2_\sigma(\Omega(t))$. Symmetrically, $\mathcal{S}(t)^{-1}$ defined in \eqref{S-1(t)1} sends $L^2_\sigma(\Omega(t))$ to $L^2_\sigma(\Omega(0))$ and it is straightforward from the explicit expressions that this is the inverse of $\mathcal{S}(t)$.

In order to give a bound on their norms, we just observe that%
\begin{align*}
\left\Vert \mathcal{S}\left(  t\right)  v\right\Vert _{L^{2}\left(
\Omega\left(  t\right)  \right)  }^{2}  &  =\int_{\Omega\left(  t\right)
}\left\vert \dfrac{D\Theta_t\left(\Theta_t^{-1}(x)\right)  }{\det
D\Theta_t\left(\Theta_t^{-1}(x)  \right)  }v(\Theta_t^{-1}(x))\right\vert ^{2}\mathrm{d}x\\
&  =\int_{\Omega\left(  0\right)  }\left\vert D\Theta_t\left(x\right)
v(x)\right\vert ^{2}\frac{\mathrm{d}x}{\det D\Theta_t\left(x\right)  }\leq
C\left(  t\right)  \left\Vert v\right\Vert _{L^{2}\left(  \Omega\left(
0\right)  \right)  }^{2},
\end{align*}
while similar computations shows that%
\[
\left\Vert \mathcal{S}\left(  t\right)  ^{-1}v\right\Vert _{L^{2}\left(
\Omega\left(  0\right)  \right)  }\leq C\left(  t\right)  \left\Vert
v\right\Vert _{L^{2}\left(  \Omega\left(  t\right)  \right)  }.
\]

Finally, taking into account the fact that for $|x| \ge m_0$,  $\Theta(t,x) = x$, it is clear that $\mathcal{S}(t)$, resp. $\mathcal{S}(t)^{-1}$,  maps $L_{\sigma}^{2}\left(
\Omega^{m}\left(  0\right)  \right)  $ into $L_{\sigma}^{2}\left(  \Omega
^{m}\left(  t\right)  \right)  $, resp. $L_{\sigma}^{2}\left(
\Omega^{m}\left(  t\right)  \right)  $ into $L_{\sigma}^{2}\left(  \Omega
^{m}\left(  0\right)  \right)  $ with norms independent of $m \ge m_0$. This concludes the proof of the proposition. 

\bigskip
The next proposition is very important for the estimates regarding the time
derivative of the solutions built in Section \ref{sec_prescribed}.

\begin{proposition}
\label{Prop_transport}We let $e\left(  t\right)  =\mathcal{S}\left(  t\right)
e\left(  0\right)  $ where $e\left(  0\right)  \in L_{\sigma}^{2}\left(
\Omega\left(  0\right)  \right)  $. Then%
\begin{equation}
\partial_{t}e+\operatorname{div}\left(  v^{ALE} \otimes e \right)  =(\nabla
v^{ALE})^{T}e. \label{transport_o}%
\end{equation}
In general
\[
\partial_{t}\mathcal{S}\left(  t\right)  u\left(  t\right)
+\operatorname{div}\left(  v^{ALE} 
\otimes  \mathcal{S}\left(  t\right)  u\left(  t\right) \right)  =(\nabla v^{ALE})^{T}\mathcal{S}\left(  t\right)
u\left(  t\right)  +\mathcal{S}\left(  t\right)  \partial_{t}u\left(
t\right)  .
\]

\end{proposition}

{\em Proof. }Apply the time derivative to the relation%
\[
e\left(  t,\Theta\left(  t,x\right)  \right)  \det D\Theta\left(  t,x\right)
=D\Theta\left(  t,x\right)  e_{0}\left(  x\right)
\]
and use the formulas%
\begin{align*}
\partial_{t}D\Theta\left(  t,x\right)   &  =Dv^{ALE}\left(  t,\Theta\left(
t,x\right)  \right)  D\Theta\left(  t,x\right) \\
\partial_{t}\det D\Theta\left(  t,x\right)   &  =\operatorname*{trace}\left(
\operatorname{Cof}^{T}\left(  D\Theta\right)  \partial_{t}D\Theta\left(
t,x\right)  \right)
\end{align*}
along with $\left(  \text{\ref{change_of_divergence}}\right)  $ and $\left(
\text{\ref{definition_of_our_X}}\right)  $.

\subsection{$H^1$-spaces in moving domains \label{SpectProjEstim}}

In this section we show various properties of spectral
projectors constructed with the help of the mobile Riesz basis introduced in Section \ref{sec_prescribed}. This is notably useful in order to go from a weak formulation
which holds for finite linear combinations of vector $e_{\ell}^{m}$ to
infinite linear combinations.

\medskip
\noindent
{\bf Proof of Proposition \ref{equiv_Sobolev}.}
Observe that with the help of the relation $\left(
\text{\ref{euler_to_lagrange}}\right)  $ for all $i,j\in\overline{1,3}$ we
have that%
\[
\frac{1}{\det D\Theta}\left\{  \partial_{i}(\mathcal{S}\left(  t\right)
v)_{j}\right\}  \circ\Theta=\frac{1}{\det D\Theta}\partial_{k}\left(
\operatorname{Cof}D\Theta_{ik}\frac{D\Theta_{j\ell}}{\det D\Theta}v_{\ell
}\right)  .
\]
From this and the bounds for $D\Theta$, we obtain that
\begin{align*}
\left\Vert D(\mathcal{S}\left(  t\right)  v)\right\Vert _{L^{2}\left(
\Omega^{m}(t)\right)  }^{2}  &  =\sum_{i=1}^{3}\sum_{j=1}^{3}\int_{\Omega
^{m}(0)}\left\vert \left\{  \partial_{i}(\mathcal{S}\left(  t\right)
v)_{j}\right\}  \circ\Theta\right\vert ^{2}\det D\Theta\mathrm{d}x\\
&  \leq C\left(  t\right)  \left(  \left\Vert v\right\Vert _{L^{2}\left(
\Omega^{m}\left(  0\right)  \right)  }+\left\Vert Dv\right\Vert _{L^{2}\left(
\Omega^{m}\left(  0\right)  \right)  }\right)  ,
\end{align*}
where $C\left(  t\right)  $ depends on the $L^{\infty}$-norms of derivatives
up to order two for $\Theta$. In order to prove the reverse inequality, let us observe that using $\left(  \text{\ref{lagrange_to_euler}%
}\right)  $ we have that%
\begin{align*}
\partial_{i}v_{\ell}\circ\Theta^{-1}  &  =\frac{1}{\det D\Theta^{-1}}%
\partial_{k}\left(  (\operatorname{Cof}D\Theta^{-1})_{ik}v_{\ell}\circ
\Theta^{-1}\right) \\
&  =\frac{1}{\det D\Theta^{-1}}\partial_{k}\left(  (\operatorname{Cof}%
D\Theta^{-1})_{ik}\left(  (D\Theta^{-1})\det D\Theta^{-1}S\left(  t\right)
v\right)  _{\ell}\right)
\end{align*}
and since
\[
\left\Vert Dv\right\Vert _{L^{2}\left(  \Omega^{m}\left(  0\right)  \right)
}^{2}=\left\Vert Dv\circ\Theta^{-1}\left(  \det D\Theta^{-1}\right)
^{\frac{1}{2}}\right\Vert _{L^{2}\left(  \Omega^{m}(t)\right)  }^{2}%
\]
we obtain the validity of $\left(  \text{\ref{norm_H1_equiv}}\right)  $. This concludes the proof of the proposition.

\bigskip
We have the following

\begin{proposition}
Consider $\bar{t}\in\lbrack0,T)$, $f\in C\left(  \left[  \bar{t},T\right]
;L^{2}(\Omega^{m}(\bar{t}))\right)  $ and
\[
\left\{
\begin{array}
[c]{l}%
-\Delta q\left(  s\right)  =f\left(  s\right)  \text{ in }\Omega^{m}\left(
\bar{t}\right)  ,\\
\frac{\partial q}{\partial \mathfrak{n}}\left(  s\right)  =0\text{ on }\partial\Omega
^{m}\left(  \bar{t}\right)  .
\end{array}
\right.
\]
Then, there exists a constant depending only on $m$ and $T$ such that
\begin{equation}
\left\Vert \nabla^{2}q\left(  s\right)  \right\Vert _{L^{2}(\Omega(\bar{t}%
))}\leq C^{\ast}\left(  m,T\right)  \left\Vert f\left(  \bar{t}\right)
\right\Vert _{L^{2}(\Omega^{m}(\bar{t}))}.\label{elliptic_regularity}%
\end{equation}

\end{proposition}

Next, let us present a proof for Proposition \ref{Norm_H1_Spect}.

\medskip
\noindent 
{\bf Proof of Proposition \ref{Norm_H1_Spect}.} 
Observe that
\[
\Psi_{\ell}:=\left\langle \psi,e_{\ell}^{m}(t)\right\rangle _{L^{2}\left(
\Omega^{m}(t)\right)  }=\left\langle D\Theta_t^{T}\psi\circ\Theta,e_{\ell}%
^{m}(0)\right\rangle _{L^{2}\left(  \Omega^{m}(0)\right)  }%
\]
and thus for all $n\in\mathbb{N}^{\ast}$
\begin{align*}
&  \sum_{\ell=1}^{n}(1+(\lambda_{j}^{m})^{2})\left\langle \psi,e_{\ell}%
^{m}(t)\right\rangle _{L^{2}\left(  \Omega^{m}(t)\right)  }^{2}\\
&  =\sum_{\ell=1}^{n}(1+(\lambda_{j}^{m})^{2})\left\langle \mathbb{P}%
_{\Omega^{m}(0)}(D\Theta_t^{T}\psi\circ\Theta),e_{\ell}^{m}(0)\right\rangle
_{L^{2}\left(  \Omega^{m}(t)\right)  }^{2}=\\
&  =\sum_{\ell=1}^{n}\left\langle \mathbb{P}_{\Omega^{m}(0)}(D\Theta_t^{T}%
\psi\circ\Theta),e_{\ell}^{m}(0)\right\rangle _{L^{2}\left(  \Omega
^{m}(t)\right)  }^{2}+\left\langle \mathbb{D(P}_{\Omega^{m}(0)}(D\Theta_t
^{T}\psi\circ\Theta)),\mathbb{D}(e_{\ell}^{m}(0))\right\rangle ^{2}\\
&  =\left\Vert \mathbb{P}_{\Omega^{m}(0)}(D\Theta_t^{T}\psi\circ\Theta
)\right\Vert _{H_{\sigma}^{1}\left(  \Omega^{m}(0)\right)  }\leq
C(\Theta,m)\left\Vert \psi\right\Vert _{H_{\sigma}^{1}\left(  \Omega
^{m}(t)\right)  }%
\end{align*}
where $\mathbb{P}_{\Omega^{m}(0)}$ is the Leray projector on divergence free
vector fields associated to $\Omega^{m}(0)$. We conclude that the series
$\left(  \text{\ref{series_H1}}\right)  $ is convergent.

Now, as explained in the remark that followed Proposition \ref{Norm_H1_Spect},
we can ensure the control of $\left\Vert \mathbb{P}_{\Omega^{m}(0)}%
(D\Theta_t^{T}\psi\circ\Theta)\right\Vert _{H_{\sigma}^{1}\left(  \Omega
^{m}(0)\right)  }$ only for a small time $t>0$.

Let us observe that by definition%
\begin{equation}
\mathbb{P}_{\Omega^{m}(0)}(D\Theta-t^{T}\psi\circ\Theta)=D\Theta^{T}\psi
\circ\Theta-\nabla q\label{decomposition_leray}%
\end{equation}
with%
\[
\left\{
\begin{array}
[c]{l}%
\Delta q=\operatorname{div}(D\Theta_t^{T}\psi\circ\Theta) = \operatorname{div}(D\Theta_t^{T}\psi\circ\Theta - \mathcal{S}(t)^{-1}\psi) \: \text{ in } \: \Omega^{m}\left(  0\right) \\
\frac{\partial q}{\partial \mathfrak{n}}=D\Theta_t^{T}\psi\circ\Theta\cdot\mathfrak{n}%
_{0}\text{ on }\partial\Omega^{m}\left(  0\right).
\end{array}
\right.
\]
However, as $D\Theta_t = \frac{r_i(t)}{r_i(0)} \textrm{Id}$ on $\partial B_{i}(t)$ and  $D\Theta_t =   \textrm{Id}$  on $\pa B\left(
0,m\right)  $, we have that%
\[
D\Theta_t^{T}\psi\circ\Theta\cdot\mathfrak{n}_{0}=0.
\]
Of course,%
\[
\left\Vert \nabla q\right\Vert _{L^{2}\left(  \Omega^{m}(0)\right)  }%
\leq\left\Vert D\Theta_t^{T}\psi\circ\Theta-\mathcal{S}(t)^{-1}\psi\right\Vert
_{L^{2}\left(  \Omega^{m}(0)\right)  }\leq C\left\Vert \Theta\left(  t\right)
-\mathrm{Id}\right\Vert _{W^{1\infty}\left(  \Omega^{m}(0)\right)  }\left\Vert
\psi\right\Vert _{L^{2}\left(  \Omega^{m}(t)\right)  }.
\]
with $C$ independent of time. By elliptic regularity see $\left(
\text{\ref{elliptic_regularity}}\right)  $ we have that%
\begin{align*}
\left\Vert \nabla^{2}q\right\Vert _{L^{2}\left(  \Omega^{m}(t)\right)  } &
\leq C\left(  T,m\right)  \left\Vert \operatorname{div}(D\Theta_t^{T}\psi
\circ\Theta_t -\mathcal{S}(t)^{-1}\psi\ )\right\Vert _{L^{2}\left(  \Omega^{m}(0)\right)  } \\
& \le C\left(
T,m\right)  \left\Vert D\Theta_t^{T}\psi\circ\Theta-\mathcal{S}(t)^{-1}%
\psi)\right\Vert _{H^{1}\left(  \Omega^{m}(0)\right)  }\\
&  \leq C\left(  T,m\right)  \left\Vert \Theta_t  -\mathrm{Id}%
\right\Vert _{W^{2,\infty}\left(  \Omega^{m}(0)\right)  }\left\Vert
\psi\right\Vert _{H^{1}_\sigma\left(  \Omega^{m}(t)\right)  }.
\end{align*}
Finally, we use $\left(  \text{\ref{decomposition_leray}}\right)  $ along with
the two last estimates in order to obtain that%
\begin{align*}
\left\Vert \mathbb{P}_{\Omega^{m}(0)}(D\Theta_t^{T}\psi\circ\Theta)\right\Vert
_{H_{\sigma}^{1}\left(  \Omega^{m}(0)\right)  } &  \geq\left\Vert D\Theta_t
^{T}\psi\circ\Theta\right\Vert _{H_{\sigma}^{1}\left(  \Omega^{m}(0)\right)
}-\left\Vert \nabla q\right\Vert _{H_{\sigma}^{1}\left(  \Omega^{m}(0)\right)
}\\
&  \geq\left\Vert \psi\right\Vert _{H_{\sigma}^{1}\left(  \Omega
^{m}(t)\right)  }-C\left(  T,m\right)  \left\Vert \Theta_t
-\mathrm{Id}\right\Vert _{W^{2,\infty}}\left\Vert \psi\right\Vert _{H_{\sigma
}^{1}\left(  \Omega^{m}(t)\right)  }.%
\end{align*}
This concludes the proof of Proposition \ref{Norm_H1_Spect}.

\bigskip
Let us introduce the family of operators%
\[
\left\{
\begin{array}
[c]{l}%
\mathbb{P}^{m,n}\left(  t\right)  :L_{\sigma}^{2}\left(  \Omega^{m}\left(
t\right)  \right)  \rightarrow H_{\sigma}^{1}\left(  \Omega^{m}\left(
t\right)  \right)  ,\\
\mathbb{P}^{m,n}\left(  t\right)  \left(  u\right)  =\mathcal{S}\left(
t\right)  [%
{\textstyle\sum\limits_{\ell=1}^{n}}
\left\langle \mathcal{S}^{-1}\left(  t\right)  u,e_{\ell}^{m}\left(  0\right)
\right\rangle e_{\ell}^{m}\left(  0\right)  ]=%
{\textstyle\sum\limits_{\ell=1}^{n}}
\left\langle \mathcal{S}^{-1}\left(  t\right)  u,e_{\ell}^{m}\left(  0\right)
\right\rangle e_{\ell}^{m}\left(  t\right)  .
\end{array}
\right.
\]
For all $v\in H_{\sigma}^{1}\left(  \Omega^{m}\left(  0\right)  \right)  $
\begin{equation}
\lim_{n\rightarrow\infty}\left\Vert v-\mathbb{P}^{m,n}\left(  0\right)
v\right\Vert _{H_{\sigma}^{1}(\Omega^m(0))}.\label{projection_uniform_k_ini}%
\end{equation}
The same remains true for $\mathbb{P}^{m,n}\left(  t\right)  $ with $t>0$,
more precisely

\begin{proposition} 
\begin{enumerate}
\item 
For all $t \in [0,T]$, for all $v\in H_{\sigma}^{1}\left(  \Omega^{m}\left(  t\right)  \right)  $%
\begin{equation}
\lim_{n\rightarrow\infty}\left\Vert \left(  \mathrm{Id}-\mathbb{P}%
^{m,n}(t)\right)  v\right\Vert _{H_{\sigma}^{1}(\Omega^{m}(t))}=0.
\label{projection_uniform_k}%
\end{equation}
\item More generally, for all $v$ such that $S(t)^{-1} v \in C([0,T] ; H^1_\sigma(\Omega^m(0)))$, 
$$ \lim_{n\rightarrow\infty} \sup_{t \in [0,T]}  \left\Vert \left(  \mathrm{Id}-\mathbb{P}^{m,n}(t)\right)  v(t)\right\Vert _{H_{\sigma}^{1}(\Omega^{m}(t))}=0. $$
\end{enumerate}
\end{proposition}

Proof. It is a consequence of Proposition \ref{norm_H1_equiv} since%
\begin{align*}
\left\Vert \left(  \mathrm{Id}-\mathbb{P}^{m,n}(t)\right)  v\right\Vert
_{H^{1}(\Omega^{m}(t))}  &  \leq C\left\Vert \mathcal{S}(t)^{-1}\left(
\mathrm{Id}-\mathbb{P}^{m,n}(t)\right)  v\right\Vert _{H^{1}(\Omega^{m}(t))}\\
&  =C\left\Vert \mathcal{S}(t)^{-1}v-%
{\textstyle\sum\limits_{\ell=1}^{n}}
\left\langle \mathcal{S}^{-1}\left(  t\right) v ,e_{\ell}^{m}\left(  0\right)
\right\rangle e_{\ell}^{m}\left(  0\right)  \right\Vert _{H^{1}(\Omega
^{m}(0))}%
\end{align*}
and the last term from the above identity tends to $0$ when $n$ goes to infinity owing to $\left(  \text{\ref{projection_uniform_k_ini}}\right) $.
The second item follows from the abstract result of \cite[Lemma 2.4]{Chemin}:

\begin{lemma} \label{lem_Chemin}
Let $H$ be a Hilbert space and let $\left(  A_{n}\right)  _{n\in\mathbb{N}}$
be a bounded sequence of linear operators on $H$ such that for all $h\in H$
\[
\lim_{n\rightarrow\infty}\left\Vert A_{n}h-h\right\Vert _{H}=0.
\]
Then, if $\psi\in C(\left[  0,T\right]  ;H)$ we have that%
\[
\lim_{n\rightarrow\infty}\sup_{t\in\left[  0,T\right]  }\left\Vert A_{n}%
\psi\left(  t\right)  -\psi\left(  t\right)  \right\Vert _{H}=0.
\]
\end{lemma}
One applies it with $H = H^1_\sigma(\Omega^m(0))$, $A_n = \mathbb{P}^{m,n}(0)$, $\psi(t) = \mathcal{S}(t)^{-1} v(t)$, and conclude with Proposition \ref{norm_H1_equiv}.

\begin{proposition} \label{approximare_spect}
For all $\tilde \psi$ with  $\: \tilde \psi \in L^\infty(0,T ; L^2_\sigma(\Omega^m)), \:  \pa_t \tilde \psi, \: \na \tilde  \psi   \in L^2(0,T ; L^2(\Omega^m))$, 
\begin{align*}
\sup_{t \in [0,T]} \| \tilde \psi(t) - \mathbb{P}^{m,n}(t)  \tilde \psi(t)\|_{L^2(\Omega_m(t))} & \xrightarrow[n \rightarrow +\infty]{} 0, \\
\| (\pa_t, \na) (\tilde \psi - \mathbb{P}^{m,n}(t)  \tilde \psi )\|_{L^2(0,T ; L^2(\Omega_m))} &\xrightarrow[n \rightarrow +\infty]{} 0. 
\end{align*}
\end{proposition}

{\em Proof.} We write  $\displaystyle \mathbb{P}^{m,n}(t)\tilde \psi(t)  = \sum_{\ell = 1}^n \tilde \Psi_\ell(t) e_\ell^m(t)$.  We start with the proof of the second item.  By \eqref{projection_uniform_k}, $\na  \mathbb{P}^{m,n} \tilde \psi$ converges to  $\na \tilde \psi$ in $L^2(\Omega^m(t))$ for a.e. $t$, but also in $L^2(0,T ; L^2(\Omega^m))$ (for instance by dominated convergence). As regards the time derivative, we deduce from Proposition \ref{Prop_transport} applied with $u = S(\cdot)^{-1}\tilde \psi$ that 
$$ \pa_t (S(\cdot)^{-1}\tilde \psi) \in L^2(0,T ; L^2(\Omega^m(0)). $$
As $\tilde \Psi_\ell(t) = \langle S(t)^{-1} \tilde \psi(t) , e_\ell^m(0) \rangle$, it follows that the sequence of time derivatives $(\tilde \Psi_\ell')_{\ell \in \N}$ belongs to $L^2(0,T ; l^2(\N))$ and so the series  $\sum_{\ell} \tilde \Psi_\ell'  e_\ell^m$ converges in $L^2(0,T ; L^2_\sigma(\Omega^m)$ by Proposition \ref{prop_Riesz_Basis} . 
From  this property and from the identity 
$$\pa_t \left( \mathbb{P}^{m,n}\tilde \psi \right) = \sum_{\ell=1}^n \tilde \Psi_\ell' e^m_\ell +  v^{ALE} \cdot \na \mathbb{P}^{m,n}\tilde \psi - \mathbb{P}^{m,n}\tilde \psi \cdot \na v^{ALE}  $$
deduced from Proposition \ref{Prop_transport}, we get that $\| \pa_t (\tilde \psi - \mathbb{P}^{m,n}(t)  \tilde \psi )\|_{L^2(0,T ; L^2(\Omega_m))} \rightarrow  0$. 
It remains to prove the first item. We first notice that $S(t)^{-1} \tilde \psi(t) \in H^1(0,T ; L^2(\Omega^m(0)))$ hence to $C([0,T] ; L^2_\sigma(\Omega^m(0)))$. We then apply Lemma \ref{lem_Chemin}, this time with $H = L^2_\sigma(\Omega^m(0))$, $A_n = \mathbb{P}^{m,n}(0)$, $\psi(t) = \mathcal{S}(t)^{-1} \tilde \psi(t)$. The result follows, as $\|S(t)\|_{L^2(\Omega^m(0)) \rightarrow  L^2(\Omega^m(t))}  \le C$.   

%
%

\section{Approximation of test functions} \label{appendixC}

In the following, we let
\[
B_{R,2R}=\left\{  x\in\mathbb{R}^{3}:R<\left\vert x\right\vert <2R\right\}  ,
\]
and%
\[
L_{0}^{2}\left(  B_{R,2R}\right)  =\left\{  \varphi\in L^{2}\left(
B_{R,2R}\right)  :\int_{B_{R,2R}}\varphi(x)\mathrm{d}x=0\right\}  .
\]
The following proposition can be essentially found in \cite[chapter 3]{Galdi}: 

\begin{proposition}
\label{Bogovski}

\begin{enumerate}
\item There exists a linear continious operator%
\[
\mathcal{B}:L_{0}^{2}\left(  B_{R,2R}\right)  \rightarrow W_{0}^{1,2}\left(
B_{R,2R}\right)
\]
such that%
\[
\left\{
\begin{array}
[c]{l}%
\operatorname{div}\mathcal{B}\left[  \varphi\right]  =\varphi,\\
\left\Vert \nabla\mathcal{B}\left[  \varphi\right]  \right\Vert _{L^{2}%
(B_{R,2R})}\leq c\left\Vert \varphi\right\Vert _{L^{2}(B_{R,2R})}%
\end{array}
\right.
\]
with $c$ independent of $R$, and such that if $\varphi\in C_{c}^{\infty}\left(
B_{R,2R}\right)  $ with $\int_{B_{R,2R}}\varphi(x)\mathrm{d}x=0$ then 
$\mathcal{B}\left[  \varphi\right]  \in C_{c}^{\infty}\left(  B_{R,2R}\right)
$.

\item If $\varphi\in C_{c}^{\infty}\left(  \left[  0,T\right]  \times
B_{R,2R}\right)  $ then the application
\[
\left(  t,x\right)  \rightarrow\mathcal{B}\left[  \varphi(t)\right](x)
\]
is $C^{\infty}\left(  \left[  0,T\right]  \times B_{R,2R}\right)  $,
$\mathcal{B}\left[  \varphi(t)\right]  \in C_{c}^{\infty}\left(
B_{R,2R}\right)  $ and there exists $c$ independent of $R$ such that
\[
\left\Vert \nabla\partial_{t}^{(k)}\mathcal{B}\left[  \varphi\right]
\right\Vert _{L^{2}(B_{R,2R})}\leq c\left\Vert \partial_{t}^{(k)}%
\varphi\right\Vert _{L^{2}(B_{R,2R})}.
\]

\end{enumerate}
\end{proposition}

See Theorem $\mathrm{III}.3.1$, Theorem
$\mathrm{III}.3.4$ and Remark $\mathrm{III}.3.3$ in \cite{Galdi}.

Next we state

\begin{proposition} 
\label{Approx_test_function} Consider $X: [0,T] \rightarrow \R^{3N}$, $R : [0,T] \rightarrow \R^N$ Lipschitz. Let 
$$ B_i(t) = B(x_i(t), r_i(t)), \quad i \in \overline{1,N}, \quad \Omega(t) = \R^3 \setminus \cup_{i=1}^N \overline{B_i(t)}, \quad t \in [0,T]. $$
Let 
\[ \tilde \psi \in  L^{\infty}\left(0,T;L^{2}\left(  \Omega \right)\right), \:  \text{ with } \: \partial_{t}\tilde \psi,\nabla \tilde \psi\in L^{2}\left(0,T;L^{2}\left(  \Omega  \right)\right),
\]
such that  for all $t\in\left[  0,T\right]  $, $\tilde \psi\left(  t\right)
\cdot\mathfrak{n}\vert_{\pa \Omega(t)}=0$ and $\operatorname{div}\tilde \psi\left(  t\right)  =0$.
 Then, for all $\varepsilon>0$ there exists some $\tilde \psi_{\varepsilon} : [0,T] \times \R^3$, Lipschitz in $t$ and smooth and  compactly supported in $x$ such that for all $t\in\left[  0,T\right]  $, $\tilde \psi
_{\varepsilon}\left(  t\right)  \cdot\mathfrak{n}\vert_{\pa \Omega(t)}=0$, $\operatorname{div}%
\tilde \psi_{\varepsilon}\left(  t\right)  =0$,  and%
\begin{equation} \label{approx_tilde_psi}
\sup_{t\in\left[  0,T\right]  }\left\Vert \tilde \psi(t)-\tilde \psi_{\varepsilon
}(t)\right\Vert _{L^{2}\left(  \Omega(t)\right)  }+\left\Vert \tilde \psi
-\tilde \psi_{\varepsilon}\right\Vert _{H^{1}(Q([0,T]))}\leq\varepsilon.
\end{equation}

\end{proposition}

{\em Proof.} By working with $\mathcal{S}(t)^{-1} \tilde \psi$ instead of $\tilde \psi$, see Appendix \ref{appendixB}, it is enough to prove the result for the fixed domain $\Omega(0)$ instead of the moving domain $\Omega(t)$. Indeed, if $\tilde \psi_0^\eps$ is an approximation of $\mathcal{S}(t)^{-1} \tilde \psi$ on $\Omega(0)$ with all the properties listed above, then $\mathcal{S}(t)\tilde \psi_0^\eps(t)$ will be an approximation of $\tilde \psi(t)$ on $\Omega(t)$, with all the properties listed above.  Note that the limited Lipschitz regularity in time comes from the limited Lipschitz regularity of the
diffeomorphism $\Theta$. Therefore, we assume from now on that 
\[ \tilde \psi \in  L^{\infty}\left(0,T;L^{2}\left(  \Omega(0) \right)\right), \:  \text{ with } \: \partial_{t}\tilde \psi,\nabla \tilde \psi\in L^{2}\left(0,T;L^{2}\left(  \Omega(0)  \right)\right). 
\]

{\em Step 1.} We show that we can restrict to $\tilde \psi$ which is smooth in $t$ and $x$.  First, by a standard mollification process in time, we can approximate $\tilde \psi$ by a function which is smooth in time: in other words,  we can assume 
$$ \tilde \psi \in C^\infty\left([0,T]; H^1_{\sigma}(\Omega(0))\right).$$
From there, for regularization in space,  one introduces a continuous  extension operator 
$E: H^1(\Omega(0)) \rightarrow H^1(\R^3)$ and set   $ \psi^n(t)  =  \rho^n \star (E \tilde \psi)$ for $\rho_n$ a smooth and compactly supported approximation of unity in $\R^3$. It is smooth in $t$ and $x$ and satisfies 
$$ \sup_{t \in [0,T]} \|\psi_n(t) -\tilde \psi(t)\|_{L^2(\Omega(0))} + \| \psi^n - \tilde \psi \|_{H^s(0,T ; H^1(\Omega(0)))} \xrightarrow[n \rightarrow +\infty]{} 0 \quad \forall s \ge 0. $$
However, $\psi_n$ is not  necessarily divergence free and tangent at the boundary, so that we need to add the corrector $w^n = \na q^n$ where 
\begin{align*}
 \Delta q^n &  = - \div \psi^n  =   - \div (\psi_n - \psi)  \: \text{ in } \:  \Omega(0), \\
 \pa_{\mathfrak{n}} q^n \vert_{\pa \Omega(0)} & = - \psi^n \cdot \mathfrak{n}\vert_{\pa \Omega(0)} = - (\psi^n - \psi) \cdot \mathfrak{n}  \vert_{\pa \Omega(0)}.  
 \end{align*}
Note that the source term $\div(\psi^n - \tilde \psi)$ has compact support in $\overline{\Omega(0)}$.  Standard considerations yield a smooth solution $q^n$ with 
$\|\na w^n = \na^2 q_n\|_{L^2(\Omega(0))} \le C \| \psi^n - \psi\|_{H^1(\Omega(0))}$
and a similar estimate for the time derivative. 
 It follows that $\tilde \psi^n := \psi^n + w^n$ is smooth, divergence-free, tangent at the boundary and satisfies 
$$ \sup_{t \in [0,T]} \|\tilde \psi_n(t) -\tilde \psi(t)\|_{L^2(\Omega(0))} + \| \tilde \psi^n - \tilde \psi \|_{H^s(0,T ; H^1(\Omega(0)))} \xrightarrow[n \rightarrow +\infty]{} 0 \quad \forall s \ge 0. $$
In other words, one can restrict to a smooth $\tilde \psi$. 

{\em Step 2.} We finally show that we can approximate the smooth field $\tilde \psi$ by a smooth field with compact support (still divergence-free and tangent at the boundary).  Let $R_{0}$  such that $\Omega\left(0\right)^c \subset B\left(
0,R_{0} \right)$. 
For all $R>R_{0}$ we consider $\chi_{R}\in C_{c}^{\infty
}\left(  \mathbb{R}^{3}\right)  $, such that  $\operatorname*{Supp}\chi_{R}\subset$
$B\left(  0,2R\right)  $, $\chi_{R}=1$ on $B\left(  0,R\right)  $ and
\[
\left\Vert \nabla\chi_{R}\right\Vert _{L^{\infty}(\mathbb{R}^{3})}\leq
\frac{c_{1}}{R}.
\]
Observe that since for all $k\geq0$ it holds that
\[
\int_{\partial B\left(  0,R\right)  \cup\partial B\left(  0,2R\right)  }%
\chi_{R}\tilde \psi\left(  t\right)  \cdot\mathfrak{n}=0\text{.}%
\]
This comes from the fact that $\chi_{R}$ vanishes on $\partial B\left(
0,2R\right)  $ while%
\begin{align*}
\int_{\partial B\left(  0,R\right)  }\chi_{R}\tilde \psi\left(  t\right)
\cdot\mathfrak{n}\mathrm{d}\sigma &  =\int_{\partial B\left(  0,R\right)
}\tilde \psi\left(  t\right)  \cdot\mathfrak{n}\mathrm{d}\sigma\\
&  =\int_{\partial B\left(  0,R\right)  }\tilde \psi\left(  t\right)  \cdot
\mathfrak{n}\mathrm{d}\sigma+\sum_{i=1}^{N}\int_{\partial B_{i}(0)}\tilde \psi\left(
t\right)  \cdot\mathfrak{n}\mathrm{d}\sigma\\
&  =\int_{B\left(  0,R\right)  \backslash\cup B_{i}(0)}\operatorname{div}%
\tilde \psi\left(  t\right)  =0.
\end{align*}
Owing to Proposition \ref{Bogovski} we can construct $\omega_{R}\in
C_{c}^{\infty}\left(  \left[  0,T\right]  \times B_{R,2R}\right)  $
\[
\left\{
\begin{array}
[c]{l}%
\operatorname{div}\omega_{R}\left(  t\right)  =-\operatorname{div}(\chi
_{R}\tilde \psi\left(  t\right)  ),\\
\omega_{R}\left(  t\right)  =0\text{ on }\partial B\left(  0,R\right)  \\
\omega_{R}\left(  t\right)  =0\text{ on }\partial B\left(  0,2R\right)
\end{array}
\right.
\]
with%
\[
\left\Vert \nabla\partial_{t}^{(k)}\omega_{R}\left(  t\right)  \right\Vert
_{L^{2}\left(  B_{R,2R}\right)  }\leq c_{2}\left\Vert \partial_{t}^{(k)}%
\tilde \psi\left(  t\right)  \cdot\nabla\chi_{R}\right\Vert _{L^{2}\left(
B_{R,2R}\right)  }%
\]
with $c_{2}$ independent of $R$. Moreover, by Sobolev-Poincar\'{e}'s
inequality, there exists $c_{3}$ independent of $R$ such that
\begin{align*}
\left\Vert \partial_{t}^{(k)}\omega_{R}\left(  t\right)  \right\Vert
_{L^{2}\left(  B_{R,2R}\right)  } &  \leq c_{3}R\left\Vert \nabla\partial
_{t}^{(k)}\omega_{R}\left(  t\right)  \right\Vert _{L^{2}\left(  \Omega\left(0\right)  \right)  }\leq c_{2}c_{3}R\left\Vert \partial_{t}^{(k)}\tilde \psi\left(
t\right)  \cdot\nabla\chi_{R}\right\Vert _{L^{2}\left(  B_{R,2R}\right)  }\\
&  \leq c_{1}c_{2}c_{3}\left\Vert \partial_{t}^{(k)}\tilde \psi\left(  t\right)
\right\Vert _{L^{2}\left(  B_{R,2R}\right)  }.
\end{align*}
By slightly abusing notations, we observe that $\tilde \psi_{R}:=\chi_{R}\tilde \psi
+\omega_{R}$ verifies that for all $t\in\left[  0,T\right]  $%
\[
\operatorname{div}\tilde \psi_{R}\left(  t\right)  =0,\text{ }\operatorname*{Supp}%
\tilde \psi_{R}\left(  t\right)  \subset B\left(  0,2R\right)  \cap\Omega\left(
0\right)  .
\]
Observe that%
\begin{align}
&  \left\Vert \nabla\tilde \psi-\nabla\tilde \psi_{R}\right\Vert _{L^{2}(0,T;L^{2}%
(\Omega\left(0\right)  ))}\nonumber\\
&  \lesssim\left\Vert (1-\chi_{R})\nabla\tilde \psi\right\Vert _{L^{2}(0,T;L^{2}%
(\Omega\left(0\right)  ))}+\frac{1}{R}\left\Vert \tilde \psi\right\Vert
_{L^{2}(0,T;L^{2}(\Omega\left(0\right)  ))}+\left\Vert \nabla\omega
_{R}\right\Vert _{L^{2}(0,T;L^{2}(\Omega\left(0\right)  ))}.\label{grad}%
\end{align}
Moreover,
\begin{align}
\left\Vert \partial_{t}\tilde \psi-\partial_{t}\tilde \psi_{R}\right\Vert _{L^{2}%
(0,T;L^{2}(\Omega\left(0\right)  ))} &  \lesssim\left\Vert (1-\chi
_{R})\partial_{t}\tilde \psi\right\Vert _{L^{2}(0,T;L^{2}(\Omega\left(0\right)
))}+\left\Vert \partial_{t}\omega_{R}\right\Vert _{L^{2}(0,T;L^{2}%
(\Omega\left(0\right)  ))}\nonumber\\
&  \lesssim\left\Vert (1-\chi_{R})\partial_{t}\tilde \psi\right\Vert _{L^{2}%
(0,T;L^{2}(\Omega\left(0\right)  ))}+\left\Vert \partial_{t}\tilde \psi\right\Vert
_{L^{2}(0,T;L^{2}(B_{R,2R}))}.\label{dt}%
\end{align}
Observe that for all $t\in\left[  0,T\right]  $%
\begin{equation}
\left\Vert \tilde \psi\left(  t\right)  -\tilde \psi_{R}\left(  t\right)  \right\Vert
_{L^{2}(\Omega(0))}\lesssim\left\Vert \left(  1-\chi_{R}\right)  \tilde \psi\left(
t\right)  \right\Vert _{L^{2}(\Omega(0))}+\left\Vert \tilde \psi\left(  t\right)
\right\Vert _{L^{2}(B_{R,2R})}.\label{norm_L2}%
\end{equation}

Combing $\left(  \text{\ref{grad}}\right)  $, $\left(  \text{\ref{dt}}\right)
$ and $\left(  \text{\ref{norm_L2}}\right)  $ we obtain the conclusion.

\begin{corollary} \label{cor_test_functions_dil}
For all $\psi \in L^2(0,T ; L^2_{\mathrm{dil}}(\Omega))$ with $\pa_t \psi, \na \psi \in L^2(0,T; L^2(\Omega))$, for all $\eps > 0$, there exists 
$$\psi^\eps \in  L^2(0,T ; L^2_{\mathrm{dil}}(\Omega)), \quad \psi^\eps \in W^{1,\infty}(0,T ; H^s(\Omega)) \: \forall s \ge 0. $$
such that 
\begin{equation} 
\sup_{t\in\left[  0,T\right]  }\left\Vert  \psi(t)- \psi_{\varepsilon
}(t)\right\Vert _{L^{2}\left(  \Omega(t)\right)  }+\left\Vert  \psi
-  \psi_{\varepsilon}\right\Vert _{H^{1}(Q([0,T]))}\leq\varepsilon.
\end{equation}
\end{corollary}
{\em Proof.} We can decompose 
$$ \psi(t,x)  = \tilde \psi(t,x) + \sum_{i=1}^{4N} \Psi_i(t)  \na q_i(t,x)  $$
where $\tilde \psi$ is as in the previous proposition. We then take 
$$ \psi^\eps(t,x)  = \tilde \psi^\eps(t,x) + \sum_{i=1}^{4N} \Psi_i^\eps(t)  \na q_i(t,x)  $$
where $\tilde \psi^\eps$ is given by the previous proposition, while $\Psi_i^\eps$ is a mollification in time of $\Psi_i$. 

{\small
\bibliographystyle{siam}
\bibliography{weak_solutions}

\begin{thebibliography}{10}

\bibitem{Beale}
{\sc J.~T. Beale}, {\em The initial value problem for the navier-stokes
  equations with a free surface}, Communications on Pure and Applied
  Mathematics, 34 (1981), pp.~359--392.

\bibitem{Biro}
{\sc Z.~Biro and J.~Velazquez}, {\em Analysis of a free boundary problem
  arising in bubble dynamics}, SIAM Journal on Mathematical Analysis, 32
  (2000), pp.~142--171.

\bibitem{Boyer}
{\sc F.~Boyer and P.~Fabrie}, {\em Mathematical Tools for the Study of the
  Incompressible Navier-Stokes Equations andRelated Models}, vol.~183, Springer
  Science \& Business Media, 2012.

\bibitem{BurteaGavrilyukPerrin2024}
{\sc C.~Burtea, S.~L. Gavrilyuk, and C.~Perrin}, {\em Hamilton's principle of
  stationary action in multiphase flow modeling}, in Mathematical methods and
  modeling for mixtures of fluids and interface evolution, Paris:
  Soci{\'e}t{\'e} Math{\'e}matique de France (SMF), 2024, pp.~77--106.

\bibitem{Nous}
{\sc C.~Burtea and D.~Gerard-Varet}, {\em Derivation of fluid-bubbles
  interaction models}.
\newblock to appear, 2025.

\bibitem{Caflisch}
{\sc R.~E. Caflisch, M.~J. Miksis, G.~C. Papanicolaou, and L.~Ting}, {\em
  Effective equations for wave propagation in bubbly liquids}, Journal of Fluid
  Mechanics, 153 (1985), pp.~259--273.

\bibitem{Sarka2}
{\sc N.~V. Chemetov and {\v{S}}.~Ne{\v{c}}asov{\'a}}, {\em The motion of the
  rigid body in the viscous fluid including collisions. global solvability
  result}, Nonlinear Analysis: Real World Applications, 34 (2017),
  pp.~416--445.

\bibitem{Sarka}
{\sc N.~V. Chemetov, {\v{S}}.~Ne{\v{c}}asov{\'a}, and B.~Muha}, {\em
  Weak-strong uniqueness for fluid-rigid body interaction problem with slip
  boundary condition}, Journal of mathematical physics, 60 (2019).

\bibitem{Chemin}
{\sc J.-Y. Chemin, B.~Desjardins, I.~Gallagher, and E.~Grenier}, {\em
  Mathematical geophysics}, vol.~32 of Oxford Lecture Series in Mathematics and
  its Applications, The Clarendon Press, Oxford University Press, Oxford, 2006.
\newblock An introduction to rotating fluids and the Navier-Stokes equations.

\bibitem{DuerinckxGloria}
{\sc M.~Duerinckx and A.~Gloria}, {\em Corrector equations in fluid mechanics:
  Effective viscosity of colloidal suspensions}, Archive for Rational Mechanics
  and Analysis, 239 (2021), pp.~1025--1060.

\bibitem{Feireisl}
{\sc E.~Feireisl}, {\em On the motion of rigid bodies in a viscous
  incompressible fluid}, in Nonlinear Evolution Equations and Related Topics:
  Dedicated to Philippe B{\'e}nilan, Springer, 2003, pp.~419--441.

\bibitem{Galdi}
{\sc G.~P. Galdi}, {\em An introduction to the mathematical theory of the
  {N}avier-{S}tokes equations, steady-state problems}, Springer Monographs in
  Mathematics, Springer, New York, 2~ed., 2011.

\bibitem{Gerard-VaretHillairet2}
{\sc D.~G{\'e}rard-Varet and M.~Hillairet}, {\em Existence of weak solutions up
  to collision for viscous fluid-solid systems with slip}, Communications on
  Pure and Applied Mathematics, 67 (2014), pp.~2022--2076.

\bibitem{Gerard-VaretHillairet1}
{\sc D.~Gerard-Varet and M.~Hillairet}, {\em Analysis of the viscosity of
  dilute suspensions beyond einstein's formula}, Arch Rational Mech Anal,
  (2020), pp.~1349---1411.

\bibitem{SueurGlass}
{\sc O.~Glass and F.~Sueur}, {\em Uniqueness results for weak solutions of
  two-dimensional fluid--solid systems}, Archive for Rational Mechanics and
  Analysis, 218 (2015), pp.~907--944.

\bibitem{Grandmont}
{\sc C.~Grandmont and Y.~Maday}, {\em Existence for an unsteady
  fluid-structureinteraction problem}, ESAIM: Mathematical Modelling and
  Numerical Analysis, 34 (2000), pp.~609--636.

\bibitem{Guo-Tice}
{\sc Y.~Guo and I.~Tice}, {\em Local well-posedness of the viscous surface wave
  problem without surface tension}, Analysis \& PDE, 6 (2013), pp.~287--369.

\bibitem{Mazzucato}
{\sc B.~M. Haines and A.~L. Mazzucato}, {\em A proof of {E}instein's effective
  viscosity for a dilute suspension of spheres}, SIAM J. Math. Anal., 44
  (2012), pp.~2120--2145.

\bibitem{Henrot}
{\sc A.~Henrot and M.~Pierre}, {\em Variation et optimisation de formes},
  vol.~48 of Math\'ematiques \& Applications (Berlin) [Mathematics \&
  Applications], Springer, Berlin, 2005.
\newblock Une analyse g\'eom\'etrique. [A geometric analysis].

\bibitem{Hillairet}
{\sc M.~Hillairet}, {\em Lack of collision between solid bodies in a 2d
  incompressible viscous flow}, Communications in Partial Differential
  Equations, 32 (2007), pp.~1345--1371.

\bibitem{Hillairet1}
{\sc M.~Hillairet, H.~Mathis, and N.~Seguin}, {\em Analysis of compressible
  bubbly flows. part i: construction of a microscopic model}, ESAIM:
  Mathematical Modelling and Numerical Analysis, 57 (2023), pp.~2835--2863.

\bibitem{Hillairet2}
\leavevmode\vrule height 2pt depth -1.6pt width 23pt, {\em Analysis of
  compressible bubbly flows. part ii: Derivation of a macroscopic model},
  ESAIM: Mathematical Modelling and Numerical Analysis, 57 (2023),
  pp.~2865--2906.

\bibitem{Hofer}
{\sc R.~H\"ofer and R.~Schubert}, {\em The influence of einstein's effective
  viscosity on sedimentation at very small particle volume fraction}, Annales
  de l'Institut Henri Poincar{\'e} C, Analyse non lin{\'e}aire, 38 (2021),
  pp.~1897--1927.

\bibitem{Richard-Juan}
{\sc R.~M. H{\"o}fer and J.~J. Vel{\'a}zquez}, {\em The method of reflections,
  homogenization and screening for poisson and stokes equations in perforated
  domains}, Archive for Rational Mechanics and Analysis, 227 (2018),
  pp.~1165--1221.

\bibitem{Weinstein}
{\sc C.-C. Lai and M.~I. Weinstein}, {\em Free boundary problem for a gas
  bubble in a liquid, and exponential stability of the manifold of spherically
  symmetric equilibria}, Archive for Rational Mechanics and Analysis, 247
  (2023), p.~100.

\bibitem{Weinstein2}
\leavevmode\vrule height 2pt depth -1.6pt width 23pt, {\em Thermal relaxation
  toward equilibrium and periodically pulsating spherical bubbles in an
  incompressible liquid}, Nonlinear Analysis, 238 (2024), p.~113397.

\bibitem{Weinstein3}
\leavevmode\vrule height 2pt depth -1.6pt width 23pt, {\em Asymmetric
  deformations of a perturbed spherical bubble in an incompressible fluid},
  SIAM Journal on Applied Mathematics, 85 (2025), pp.~1212--1236.

\bibitem{Lannes}
{\sc D.~Lannes}, {\em The water waves problem: mathematical analysis and
  asymptotics}, vol.~188, American Mathematical Soc., 2013.

\bibitem{leighton}
{\sc T.~Leighton}, {\em From seas to surgeries, from babbling brooks to baby
  scans: The acoustics of gas bubbles in liquids}, International Journal of
  Modern Physics B, 18 (2004), pp.~3267--3314.

\bibitem{MasmoudiRousset1}
{\sc N.~Masmoudi and F.~Rousset}, {\em Uniform regularity for the
  navier--stokes equation with navier boundary condition}, Archive for Rational
  Mechanics and Analysis, 203 (2012), pp.~529--575.

\bibitem{MasmoudiRousset2}
{\sc N.~Masmoudi and F.~Rousset}, {\em Uniform regularity and vanishing
  viscosity limit for the free surface navier--stokes equations}, Archive for
  Rational Mechanics and Analysis, 223 (2017), pp.~301--417.

\bibitem{Mecherbet}
{\sc A.~Mecherbet}, {\em Sedimentation of particles in stokes flow}, Kinetic
  and Related Models, 12 (2019), pp.~995--1044.

\bibitem{Moussa}
{\sc A.~Moussa}, {\em Some variants of the classical {A}ubin-{L}ions lemma}, J.
  Evol. Equ., 16 (2016), pp.~65--93.

\bibitem{Muha}
{\sc B.~Muha and S.~{\v{C}}ani{\'c}}, {\em Existence of a weak solution to a
  fluid--elastic structure interaction problem with the navier slip boundary
  condition}, Journal of Differential Equations, 260 (2016), pp.~8550--8589.

\bibitem{Sarka3}
{\sc {\v{S}}.~Ne{\v{c}}asov{\'a}, M.~Ramaswamy, A.~Roy, and
  A.~Schl{\"o}merkemper}, {\em Motion of a rigid body in a compressible fluid
  with navier-slip boundary condition}, Journal of Differential Equations, 338
  (2022), pp.~256--320.

\bibitem{Neustupa}
{\sc J.~Neustupa and P.~Penel}, {\em A weak solution to the navier--stokes
  system with navier’s boundary condition in a time-varying domain}, in
  Recent Developments of Mathematical Fluid Mechanics, Springer, 2016,
  pp.~375--400.

\bibitem{Prosperetti}
{\sc A.~Prosperetti}, {\em Bubbles}, Physics of fluids, 16 (2004),
  pp.~1852--1865.

\bibitem{Pruss}
{\sc J.~Pr{\"u}ss and G.~Simonett}, {\em Moving interfaces and quasilinear
  parabolic evolution equations}, vol.~105, Springer, 2016.

\bibitem{SanMartin}
{\sc J.~A. San~Mart{\'\i}n, V.~Starovoitov, and M.~Tucsnak}, {\em Global weak
  solutions{\'s}for the two-dimensional motion{\'s}of several rigid
  bodies{\'s}in an incompressible viscous fluid}, Archive for Rational
  Mechanics and analysis, 161 (2002), pp.~113--147.

\bibitem{SanchezPalencia}
{\sc E.~Sanchez-Palencia}, {\em On the asymptotics of the fluid flow past an
  array of fixed obstacles}, Int. J. Eng. Sci., 20 (1982), pp.~1291--1301.

\bibitem{Takahashi}
{\sc T.~Takahashi}, {\em Analysis of strong solutions for the equations
  modeling the motion of a rigid-fluid system in a bounded domain}, Adv.
  Differential Equations, 8 (2003), pp.~1499--1532.

\bibitem{Tani}
{\sc A.~Tani}, {\em Small-time existence for the three-dimensional
  navier-stokes equations for an incompressible fluid with a free surface},
  Archive for rational mechanics and analysis, 133 (1996), pp.~299--331.

\bibitem{Gavrilyuk}
{\sc V.~Teshukov and S.~Gavrilyuk}, {\em Kinetic model for the motion of
  compressible bubbles in a perfect fluid}, European Journal of
  Mechanics-B/Fluids, 21 (2002), pp.~469--491.

\bibitem{Smereka}
{\sc N.~Wang and P.~Smereka}, {\em Effective equations for sound and void wave
  propagation in bubbly fluids}, SIAM Journal on Applied Mathematics, 63
  (2003), pp.~1849--1888.

\bibitem{Marcel}
{\sc M.~Zodji}, {\em Etude qualitative de l'interface entre deux fluides
  non-miscibles}.
\newblock PhD thesis, available at
  https://sites.google.com/imsp-uac.org/marcel-zodji/research, 2025.

\end{thebibliography}
}

\end{document}